\documentclass[10pt,a4paper]{amsart}
\usepackage{amsfonts}
\usepackage{amsthm}
\usepackage{amssymb}
\usepackage{amsmath}
\usepackage{upref}
\usepackage{mathrsfs}
\usepackage{color}
\usepackage[citecolor=blue,colorlinks=true]{hyperref}
\usepackage[left=3cm,right=3cm,top=4cm,bottom=3cm]{geometry}

\usepackage{enumitem}

\newcommand{\blue}[1]{{\color{black} #1 }}

\newcommand{\del}[1]{}

\newtheorem{thm}{Theorem}[section]
\newtheorem{lemma}[thm]{Lemma}
\newtheorem{prop}[thm]{Proposition}
\newtheorem{cor}[thm]{Corollary}
\theoremstyle{definition}
\newtheorem{defin}[thm]{Definition}
\theoremstyle{remark}
\newtheorem{rem}[thm]{Remark}

\numberwithin{equation}{section}

\newcommand{\me}{\mathrm{e}}

\newcommand{\dif}{\mathrm{d}}
\newcommand{\dd}{\mathrm{d}}
\newcommand{\totdif}{\mathrm{D}}
\newcommand{\mf}{\mathscr{F}}
\newcommand{\mr}{\mathbb{R}}
\newcommand{\R}{\mathbb{R}}
\newcommand{\prst}{\mathbb{P}}
\newcommand{\stred}{\mathbb{E}}
\newcommand{\p}{\mathbb{P}}
\newcommand{\E}{\mathbb{E}}
\newcommand{\ind}{\mathbf{1}}
\newcommand{\mn}{\mathbb{N}}
\newcommand{\N}{\mathbb{N}}

\newcommand{\mz}{\mathbb{Z}}
\newcommand{\mt}{\mathbb{T}^N}
\newcommand{\T}{{\mathbb{T}^N}}

\DeclareMathOperator{\supp}{supp}

\DeclareMathOperator{\diver}{div}
\DeclareMathOperator*{\esssup}{ess\,sup}

\newcommand{\bl}{\big\langle}
\newcommand{\br}{\big\rangle}

\begin{document}
 
\title[Well-posedness and regularity for quasilinear SPDE]{Well-posedness and regularity for quasilinear degenerate parabolic-hyperbolic SPDE}

\author{Benjamin Gess}
\address[B. Gess]{Max Planck Institute for Mathematics in the Sciences, Inselstra\ss e 22, 04103 Leipzig, Germany \& Fakult\"at f\"ur Mathematik, Universit\"at Bielefeld, D-33501 Bielefeld, Germany}
\email{bgess@mis.mpg.de}

\author{Martina Hofmanov\'a}
\address[M. Hofmanov\'a]{Technical University Berlin, Institute of Mathematics, Stra\ss e des 17. Juni 136, 10623 Berlin, Germany}
\email{hofmanov@math.tu-berlin.de}

\begin{abstract}
We study quasilinear degenerate parabolic-hyperbolic stochastic partial differential equations with general multiplicative noise within the framework of kinetic solutions. Our results are twofold: First, we establish new regularity results based on averaging techniques. Second, we prove the existence and uniqueness of solutions in a full $L^1$ setting requiring no growth assumptions on the nonlinearities. In addition, we prove a comparison result and an $L^1$-contraction property for the solutions, \blue{ generalizing the results obtained in \cite{dehovo}.}
\end{abstract}

\subjclass[2010]{60H15, 35R60}
\keywords{quasilinear degenerate parabolic stochastic partial differential equation, kinetic formulation, kinetic solution, velocity averaging lemmas, renormalized solutions}

\date{\today}

\maketitle

\section{Introduction}

We study the regularity and well-posedness of quasilinear degenerate parabolic-hyperbolic SPDE of the form
\begin{equation}\label{eq}
\begin{split}
\dif u+\diver(B(u))\dif t&=\diver(A(u)\nabla u )\dif t+\varPhi(x,u)\dif W, \qquad x\in\mt,\,t\in(0,T),\\
u(0)&=u_0,
\end{split}
\end{equation}
where $W$ is a cylindrical Wiener process, $u_0 \in L^1(\T)$, $B \in C^2(\R,\R^N)$, $A \in C^1(\R,\R^{N\times N})$ takes values in the set of symmetric non-negative definite matrices and $\Phi(x,u)$ are Lipschitz continuous diffusion coefficients.

Equations of this form arise in a wide range of applications including the convection-diffusion of an ideal fluid in porous media. The addition of a stochastic noise is often used to account for numerical, empirical or physical uncertainties. In view these applications, we aim to treat \eqref{eq} under general assumptions on the coefficient $A,B$ and initial data $u_0$. In particular, the coefficients are not necessarily linear nor of linear growth and $A$ is not necessarily strictly elliptic. Hence, in particular, we include stochastic scalar conservation laws 
$$
\dif u+\diver(B(u))\dif t=\varPhi(x,u)\dif W
$$
and stochastic porous media equations 
$$
\dif u+\diver(B(u))\dif t= \Delta u^{[m]}\dif t+\varPhi(x,u)\dif W,
$$
with {$m > 2$} and \blue{$u^{[m]}:=sgn(u)|u|^m$}.

One of the main points of this paper is to provide a full $L^1$ approach to \eqref{eq}. That is, we prove regularity estimates and well-posedness for \eqref{eq} assuming no higher moments. More precisely, only $u_0 \in L^1(\T)$ and {\em no} growth assumptions on the nonlinearities $A, B$ are assumed, \textcolor{black}{ in contrast to the previous work \cite{dehovo}}. In particular, no Lipschitz continuity (and thus linear growth) assumptions on $A, B$ are supposed. This causes severe difficulties: Firstly, the weak form of \eqref{eq} is not necessarily well-defined since $A(u), B(u)$ are not necessarily in $L^1_{loc}(\T)$ for $u\in L^1(\T)$. Therefore, renormalized solutions have to be considered (cf.\ \cite{diperna2, wittbold2,wittbold3}). Secondly, in order to prove the uniqueness of $L^1$ entropy solutions an equi-integrability condition or, equivalently, a decay condition for the entropy defect measure is required (see a more detailed discussion below). The usual decay condition used in the deterministic case is not applicable in the stochastic case and a new condition and proof has to be found. Thirdly, in the stochastic case, the usual proof of existence of entropy solutions relying on the Crandall-Liggett theory of $m$-accretive operators in $L^1(\T)$ cannot be applied (cf.\ \cite{CL71,C72,car}). Instead, the construction of entropy solutions presented in this paper relies on new regularity estimates based on averaging techniques. The application of averaging techniques and the resulting regularity results are new for parabolic-hyperbolic SPDE of the type \eqref{eq}.  \\
 On the other hand, $L^1(\T)$ is a natural space to consider the well-posedness for SPDE of the type \eqref{eq} since the operators $\diver(B(\cdot)),\  \diver(A(\cdot)\nabla \cdot)$ are accretive in $L^1(\T)$ (cf.\ the discussion of the $e$-property below). In addition, and in contrast to the deterministic case, restricting to bounded solutions and hence, by localization, to Lipschitz continuous coefficients $A,B$ in \eqref{eq} does not seem to be sensible in the stochastic case, since in general no uniform $L^\infty$ bound will be satisfied by solutions to \eqref{eq}, due to the unboundedness of the driving noise $W$. \\
{\textcolor{black}{In the previous work \cite{dehovo}, SPDE of the type \eqref{eq} have been considered via a kinetic approach under more restrictive assumptions. More precisely, high moment bounds $u_0\in \bigcap_{p\ge 1} L^p(\Omega;L^p(\T))$, boundedness of the diffusion matrix $A$ and polynomial growth of $B''$ had to be assumed, in particular ruling out application to porous media equations. Due to these more restrictive assumptions all of the above mentioned difficulties do not appear in \cite{dehovo}. While we follow the principle setup to prove uniqueness of kinetic solutions, the proof has to be significantly extended in order to incorporate the necessity to work with renormalized solutions and the above mentioned weaker decay condition of the entropy defect measure. The essential difficulty in the proof of existence of solutions is the derivation of uniform estimates in some (fractional) Sobolev space. The respective arguments of  \cite{dehovo} do not apply in the more general setting considered here. We therefore take a different route by adapting the (deterministic) averaging techniques by Tadmor and Tao \cite{tadmor} to SPDE, which is entirely new. 
}}

As a particular example, \eqref{eq} contains stochastic porous media equations
\begin{equation}\label{eq:intro-PME}
\dif u=\Delta u^{[m]}\dif t+\varPhi(x,u)\dif W, \quad\text{with }{m > 2}.
\end{equation}
Stochastic porous media equations have attracted a lot of interest in recent years (cf.\ e.g.\ \cite{RRW07, PR07, BDPR08, RW08} and the references therein). All of these results rely on an $H^{-1}$ approach, that is, on treating $\Delta (\cdot)^{[m]}$ as a monotone operator in $H^{-1}$. In contrast to the deterministic case, an $L^1$ approach to stochastic porous media equations had not yet been developed, since an analog of the concept of mild solutions in the Crandall-Liggett theory of $m$-accretive operators (cf.\ \cite{V07,car}) could not be found. However, the $L^1$ framework offers several advantages: Firstly, more general classes of SPDE may be treated, secondly, contractive properties in $L^1$ norm are sometimes better than those in $H^{-1}$ norm. We next address these points in more detail.\\ 
Concerning the class of SPDE, informally speaking, the $H^{-1}$ approach relies on applying $(-\Delta)^{-1}$ to \eqref{eq:intro-PME} which then allows to use the monotonicity of $\phi(u):=u^{[m]}$ in order to prove the uniqueness of solutions. While this works well for the operator $\Delta \phi(\cdot)$, the reader may easily check that this approach fails in the presence of hyperbolic terms $\diver B(u)$ as in \eqref{eq} and can only be applied to reaction diffusion equations
  \begin{equation}\label{eq:intro-PME-2}
  \dif u=\Delta u^{[m]}\dif t+f(u)dt+\varPhi(x,u)\dif W, \quad\text{with }{m > 2}.
  \end{equation}
under unnecessarily strong assumptions on the reaction term $f$ (cf.\ e.g.\ \cite{DPRRW06,RRW07} where \eqref{eq:intro-PME-2} with $f$ satisfying rather restrictive assumptions has been considered). Roughly speaking, the problem is that the Nemytskii operator $u\mapsto f(u)$ is not necessarily monotone in $H^{-1}$ even if $f$ is a monotone function. This changes drastically in the $L^1$ setting, since both $u\mapsto \diver B(u)$ and $u\mapsto f(u)$ are accretive operators on $L^1$ under relatively mild assumptions. In this paper, we resolve these issues by establishing a full $L^1$ approach to \eqref{eq} based on entropy/kinetic methods. In particular, this extends available results on stochastic porous media equations by allowing hyperbolic terms $\diver B(u)$ and our framework immediately\footnote{We choose not to include the details on the treatment of reaction terms $f(u)$ in this paper, since their treatment is similar to the noise terms $\Phi(u)dW$.} extends to reaction terms $u \mapsto f(u)$ assuming only that $f$ is weakly monotone and $C^2$.

We proceed by stating the main well-posedness result obtained in this paper, see Theorem \ref{thm:uniqueness}, Theorem \ref{thm:ex} below. The precise framework will be given in Section \ref{sec:prelim} below and for specific examples see Section \ref{sec:examples}.

\begin{thm}\label{thm:main}
Let $u_0\in L^1(\mt)$ and assume that $A^\frac{1}{2}$ is $\gamma$-H\"older continuous for some $\gamma > \frac{1}{2}$. 
Then, kinetic solutions to \eqref{eq} are unique. Moreover, if $u_1,\,u_2$ are kinetic solutions to \eqref{eq} with initial data $u_{1,0}$ and $u_{2,0}$, respectively, then 
   \begin{equation*}
   \esssup_{t\in[0,T]}\stred\|(u_1(t)-u_2(t))^+\|_{L^1(\mt)}\leq \|(u_{1,0}-u_{2,0})^+\|_{L^1(\mt)}.
   \end{equation*}
Assume  in addition that $A,B$ satisfy a non-degeneracy assumption (cf.\ \eqref{eq:non-deg-local} below).
Then there exists a unique kinetic solution $u$ to \eqref{eq} satisfying $u\in C([0,T];L^1(\T))$, $\p$-a.s., and for all $p,q\in[1,\infty)$ there exists a constant $C>0$ such that
\begin{equation*}
   \E \esssup_{t\in [0,T]}\|u(t)\|_{L^p}^{pq} \le C (1+\|u_0\|_{L^p}^{pq}).
\end{equation*}
\end{thm}

The second direction of advantages of the $L^1$ approach lies in dynamical properties. A natural question for stochastic porous media equations is their long-time behavior, that is, the existence and uniqueness of invariant measures, mixing properties etc. If $u$, $v$ are two solutions to \eqref{eq:intro-PME} with initial conditions $u_0,v_0$ respectively, then 
\begin{equation}\label{eq:contraction_in_h-1}
   \E \|u(t)-v(t)\|_{H^{-1}} \le e^{Ct} \|u_0-v_0\|_{H^{-1}} \quad \forall t\ge 0,
\end{equation}
for some constant $C>0$. The constant $C$ corresponds to the Lipschitz norm of $u\mapsto \Phi(u)$ as a map from $H^{-1}$ to $L_2(U;H^{-1})$. In particular, the dynamics induced by \eqref{eq:intro-PME}, in general, will not be non-expanding in $H^{-1}$. In contrast, we show that 
  $$\E \|u(t)-v(t)\|_{L^1} \le \|u_0-v_0\|_{L^1}\quad\forall t\ge 0,$$
that is, in the $L^1$ setting we can choose the constant $C$ in \eqref{eq:contraction_in_h-1} to be zero. In particular, this implies the $e$-property (cf.\ \cite{KPS10}) for the associated Markovian semigroup $P_t f(x) := \E f(u_t^x)$ on $L^1(\T)$. The $e$-property has proven vital in the proof of existence and uniqueness of invariant measures for SPDE with degenerate noise (cf.\  \cite{GT15,GT14,ESR12,KPS10}).

For $x\in L^1(\T)$ let
  $$ P_x := (u^x_\cdot)_*\p,$$
that is, $P_x$ is the law of $u^x_\cdot$ on $C([0,\infty);L^1(\T))$, where $u^x_\cdot$ denotes the kinetic solution to \eqref{eq} with initial condition $x$. We equip $C([0,\infty);L^1(\T))$ with the canonical filtration $\mathcal{G}_t$ and evaluation maps $\pi_t(w):=w(t)$ for $w\in C([0,\infty);L^1(\T))$, $t\ge0$.
As in \cite{daprato}, using Theorem \ref{thm:main}, we obtain
\begin{cor} 
  The family $\{P_x\}_{x\in L^1}$ is a time-homogeneous Markov process on $C([0,\infty);L^1(\T))$ with respect to $\mathcal{G}_t$, i.e.
  \begin{equation*}
	 \E_x(F(\pi_{t+s})|\mathcal{G}_s)=\E_{\pi_s}(F(\pi_t))\quad P_x\text{-a.s.}
  \end{equation*}
  In addition, $\{P_x\}_{x\in L^1(\T)}$ is Feller and satisfies the $e$-property (cf.\ \cite{KPS10}).
\end{cor}

As mentioned above, we prove new regularity estimates for kinetic solutions to \eqref{eq} of the type
 $$u(t) \in W^{\alpha,1}(\T)\qquad \text{for a.e.}\;\; (\omega,t),$$
for some $\alpha > 0$, based on stochastic velocity averaging lemmas. \blue{In particuar, these estimates provide a smoothing property with respect to the initial data, which typically is the first step towards the construction of an invariant measure.} 

Even in the case of pure stochastic porous media equations \eqref{eq:intro-PME} this extends previously available regularity results. For related deterministic results, see \cite{bouchut, diperna, jabin, tadmor}, for stochastic hyperbolic conservation laws see \cite{debus2}. Our approach is mainly based on \cite{tadmor}, but substantial difficulties due to the stochastic integral have to be overcome. Indeed, in most of the deterministic results (with the notable exception of \cite{bouchut}), the time variable does not play a special role and is regarded as another space variable and, in particular, space-time Fourier transforms are employed in the proofs. This changes in the stochastic case due to the irregularity of the noise in time. Therefore, it was argued in \cite{debus2} that these methods are not suitable for the stochastic case and instead the approach of \cite{bouchut} which does not rely on Fourier transforms in time was employed. 
In the present paper, we rely on different arguments: We put forward averaging lemmas that rely on space-time Fourier transforms, Littlewood-Paley decomposition and a careful analysis of each of the appearing terms. As a consequence, we are able to estimate the stochastic integral as well as the kinetic measure term directly by averaging techniques, without any additional damping (as compared to \cite{bouchut,debus2}). Moreover, our averaging lemmas apply to the case of nonhomogeneous equations, that is, PDEs with zero, first and second order terms and {\em multiplicative noise}.

More precisely, as a corollary of our main regularity result for \eqref{eq}, see Theorem \ref{thm:regularity} and Corollary \ref{cor:regularity} below, we obtain
\begin{thm}\label{thm:into_regularity} Assume that $A,B$ satisfy a non-degeneracy condition (cf.\ \eqref{eq:non-deg-local} below) and are of polynomial growth of order $p$. Let $u$ be a kinetic solution to \eqref{eq}. Then
   \begin{align*}
      \E\|u\|_{L^{1}([0,T]; W^{s,1}(\T))}
      &\lesssim
         \|u_0 \|_{L^{2p+3}_{x}}^{2p+3} + 1,
      \end{align*}
    for some $s>0$.
\end{thm}

We now proceed with a more detailed discussion of the comparison to the proof of well-posedness of entropy solutions for deterministic parabolic-hyperbolic PDE \blue{in the $L^1$ setting} (cf.\ \cite{chen}) 
\begin{equation}\label{eq:intro-det-par-hyp}
  \dif u+\diver(B(u))\dif t=\diver(A(u)\nabla u )\dif t.
\end{equation}
The inclusion of stochastic perturbation causes several additional difficulties. First, the proof of existence of solutions in \cite{chen} relies on the (simple) proof of $BV$ regularity of solutions to \eqref{eq:intro-det-par-hyp}. Such a $BV$ estimate is not known in the stochastic case and does not seem to be easy to obtain (for a discussion of the necessity of such estimates in the construction of a solution see Section \ref{sec:ex} below). \blue{In the $L^p$ setting of \cite{dehovo}, the $BV$ regularity was replaced by $W^{\sigma,1}$ regularity which held true for smooth initial conditions based on similar calculations as in the uniqueness proof. Since we do not suppose any growth assumptions on the coefficients $A,B$ in \eqref{eq} these arguments cannot be used here anymore.} Therefore, we instead rely on regularity obtained based on averaging techniques. \blue{In contrast to \cite{dehovo} the obtained regularity estimates do not only prove preservation of some regularity of the initial condition, but yield a regularizing effect. Indeed, the initial condition does not need to be smooth for the averaging lemma to imply regularity for positive times. This fact in particular permits to reduce the number of approximation layers used in the proof of existence.}\\ Second, the equi-integrability estimates encoded in the decay properties of the kinetic measure in the deterministic situation, that is, in the assumption (cf.\ \cite[Definition 2.2 (iv)]{chen})
  $$\lim_{|\xi|\to \infty} \int m(t,x,\xi) dtdx =0$$
do not seem to be suitable in the stochastic case, since the {\em multiplicative} noise term $\Phi(x,u)$ is less well behaved in terms of these estimates. Indeed, the corresponding proof of a-priori estimates proceeds along different lines than in the deterministic case (cf.\ Proposition \ref{prop:decay1} below). Therefore, we replace these decay estimates by the weaker decay condition 
\blue{\begin{equation}\label{eq:m_decay}
  \lim_{\ell \to \infty} \frac{1}{2^\ell}\E \int 1_{2^\ell\leq|\xi|\leq 2^{\ell+1}} m(t,x,\xi)dtdxd\xi =0
\end{equation}}
and prove the uniqueness of kinetic solutions under this weaker assumption.  
\blue{The above difficulty is due to working in an $L^1$-setting. In the situation of initial data/solutions in $L^2$, as in \cite{dehovo}, the kinetic measure can be shown to have a.s.\ finite mass which allows to replace \eqref{eq:m_decay} by a stronger assumption.}

The kinetic approach to (deterministic) scalar conservation laws was introduced by Lions, Perthame, Tadmor in \cite{lions} and extended to parabolic-hyperbolic PDE in \cite{chen}, including PDE of porous media type. In the stochastic case, the well-posedness of such PDE had not previously been shown. Under more restrictive assumptions, as outlined above, the well-posedness to \eqref{eq} has been obtained in \cite{dehovo}.
Special cases of SPDE of the type \eqref{eq} have attracted a lot of interest in recent years. For deterministic hyperbolic conservation laws, see \cite{vov1, vov, kruzk, lpt1, lions, perth, tadmor-1}. Stochastic degenerate parabolic equations were studied in \cite{BVW, dehovo, degen} and stochastic conservation laws in \cite{bauzet, karlsen, debus2, debus, DV, feng, bgk, holden, kim, stoica, wittbold}. Recently, also scalar conservation laws driven by rough paths have been considered in \cite{friz2, H16, DGHT}. Other types of stochastic scalar conservation laws, for which randomness enters in form of a random flux have been considered in \cite{lps, lps1, gess, GS}. Stochastic quasilinear parabolic-hyperbolic SPDE with random flux have been considered in \cite{GS16}.

The paper is organized as follows. In Section \ref{sec:prelim}, we introduce the precise framework and the concept of kinetic solutions. Our main regularity result will be proven in Section \ref{sec:averaging}. This is then used in Section \ref{sec:L1} to prove the well-posedness for kinetic solutions.

\section{Preliminaries}\label{sec:prelim}

\subsection{Notation}\label{sec:notation}

In this paper, we use the brackets $\langle\cdot,\cdot \rangle$ to denote the duality between the space of distributions over $\mt \times\mr$ and $C_c^\infty(\mt \times\mr)$ and the duality between $L^p(\mt \times\mr)$ and $L^q(\mt \times\mr)$. If there is no danger of confusion, the same brackets will also denote the duality between $L^p(\mt )$ and $L^q(\mt )$.
By $\mathcal{M}([0,T]\times\mt\times\R)$ we denote the set of Radon measures on $[0,T]\times\mt\times\R$ and $\mathcal{M}^+([0,T]\times\mt\times\R)$  then contains  nonnegative Radon measures and $\mathcal{M}_b([0,T]\times\mt\times\R)$ contains finite measures. We also use the notation
$$n(\phi)=\int_{[0,T]\times\mt\times\R}\phi(t,x,\xi)\,\dif n(t,x,\xi),\quad n\in\mathcal{M}([0,T]\times\mt\times\R),\,\phi\in C_c([0,T]\times\mt\times\R).$$
In order to signify that $n\in\mathcal{M}([0,T]\times\mt\times\R)$ is only considered on $[0,T]\times\mt\times D$ for some compact set $D\subset\R$ we write $n\ind_D$. In particular,
$$\|n\ind_D\|_{\mathcal{M}_{t,x,\xi}}=\int_{[0,T]\times\mt\times D}\dif |n|(t,x,\xi).$$

The differential operators of gradient $\nabla$, divergence $\diver$ and Laplacian $\Delta$ are always understood with respect to the space variable $x$.
For two matrices $A,B$ of the same size we set
  $$A:B := \sum_{ij}a_{ij}b_{ij}.$$
Throughout the paper, we use the term {\em representative} for an element of a class of equivalence.

Finally, we use the letter $C$ to denote a generic constant that might change from one line to another. We also employ the notation $x\lesssim y$ if there exists a constant $C$ independent of the variables under consideration such that $x\leq C y$ and we write $x\sim y$ if $x\lesssim y$ and $y\lesssim x$. \blue{By $x\lesssim_z y$ we mean that the corresponding proportional constant depends on the quantity $z$.}

\subsection{Setting}

We now give the precise assumptions on each of the terms appearing in the above equation \eqref{eq}. We work on a finite-time interval $[0,T],\,T>0,$ and consider periodic boundary conditions: $x\in\mt $ where $\mt=\R^N|(2\pi\mz^N) $ is the $N$-dimensional torus. For the flux $B$ we assume
\begin{equation}\label{eq:ass_B}
  B=(B_1,\dots,B_N) \in C^2(\R,\R^N)
\end{equation}
and we set $b = \nabla B$. 
The diffusion matrix $A=(A_{ij})_{i,j=1}^N\in C^1(\R;\R^{N\times N})$ is assumed to be symmetric, positive semidefinite and its square root $\sigma := A^\frac{1}{2}$ is assumed to be locally $\gamma$-H\"older continuous for some $\gamma>1/2$, that is, for all $R>0$ there is a constant $C=C(R)$ such that
\begin{equation}\label{sigma}
|\sigma(\xi)-\sigma(\zeta)|\leq C(R)|\xi-\zeta|^\gamma\quad\forall \xi,\zeta\in 
\mr,\,|\xi|,|\zeta|\le R.
\end{equation}
We will further require a non-degeneracy condition for the symbol $\mathcal{L}$ associated to the kinetic form of \eqref{eq} 
  $$\mathcal{L}(iu,in,\xi):=i(u+b(\xi)\cdot n)+n^*A(\xi)n.$$

 For $J,\delta>0$ and $\eta \in C_b^\infty(\R)$ nonnegative let 
$$\omega^\eta_{\mathcal L}(J;\delta):=\sup_{\substack{u\in\mr,n\in \mz^N\\|n|\sim J}}|\Omega^\eta_{\mathcal L}(u,n;\delta)|,\qquad \Omega^\eta_{\mathcal L}(u,n;\delta):=\{\xi\in\supp \eta;\,|\mathcal L(iu,in,\xi)|\leq \delta\}$$
and $\mathcal{L}_\xi:=\partial_\xi\mathcal{L}$. 
We suppose that there exist $\alpha \in (0,1)$, $\beta>0$ and a measurable map $\vartheta \in L^\infty_{loc}(\R;[1,\infty))$ such that
\begin{equation}\label{eq:non-deg-local}
\begin{split}
\omega^\eta_{\mathcal L}(J;\delta)&\lesssim_{{\eta}} \bigg(\frac{\delta}{J^\beta}\bigg)^\alpha\\
\sup_{\substack{u\in\mr,n\in \mz^N\\|n|\sim J}}\sup_{\xi\in \supp \eta}
\frac{|\mathcal L_\xi(iu,in,\xi)|}{\vartheta(\xi)}
&\lesssim_{{\eta}} J^{\beta},\quad \forall \delta>0, J\gtrsim 1.
\end{split}\end{equation}
The requirement of a suitable non-degeneracy condition is classical in the theory of averaging lemmas and therefore it will be essential for Theorem \ref{thm:regularity}. The localization $\eta$ and the weight $\vartheta$ give two possibilities to control the  growth of $\mathcal{L}_\xi$ in $\xi$. In the proof of existence in Subsection \ref{sec:proof_ex}, we employ \eqref{eq:non-deg-local} with $\vartheta\equiv 1$ and  $\eta$ compactly supported which allows to obtain regularity of the localized average $\int_\R\chi_u(\xi)\eta(\xi)\dif\xi$ without  any further integrability assumptions on $u$. On the contrary, with a suitable choice of $\vartheta$, we may consider $\eta\equiv 1$ to obtain regularity of $u$ itself provided it possesses certain additional integrability. We refer the reader to Subsection \ref{sec:examples} for further discussion of \eqref{eq:non-deg-local} as well as for application to particular examples.

Regarding the stochastic term, let $(\Omega,\mf,(\mf_t)_{t\geq0},\prst)$ be a stochastic basis with a complete, right-continuous filtration. Let $\mathcal{P}$ denote the predictable $\sigma$-algebra on $\Omega\times[0,T]$ associated to $(\mf_t)_{t\geq 0}$. The initial datum $u_0$ is $\mf_0$-measurable and the process $W$ is a cylindrical Wiener process, that is, $W(t)=\sum_{k\geq1}\beta_k(t) e_k$ with $(\beta_k)_{k\geq1}$ being mutually independent real-valued standard Wiener processes relative to $(\mf_t)_{t\geq0}$ and $(e_k)_{k\geq1}$ a complete orthonormal system in a separable Hilbert space $\mathfrak{U}$. In this setting we can assume without loss of generality that the $\sigma$-algebra $\mf$ is countably generated and $(\mf_t)_{t\geq 0}$ is the filtration generated by the Wiener process and the initial condition.
For each $z\in L^2(\mt )$ we consider a mapping $\,\varPhi(z):\mathfrak{U}\rightarrow L^2(\mt )$ defined by $\varPhi(z)e_k=g_k(\cdot,z(\cdot))$. We suppose that $g_k\in C(\mt \times\mr)$ and there exists a sequence $(\alpha_k)_{k\geq 1}$ of positive numbers satisfying $D:=\sum_{k\geq 1}\alpha_k^2<\infty$ such that
\begin{align}
|g_k(x,0)|+|\nabla_xg_k(x,\xi)|+|\partial_\xi g_k(x,\xi)| &\leq \alpha_k,\qquad\forall x\in\mt ,\,\xi\in\mr. \label{linrust}
\end{align}
Note that it follows from \eqref{linrust} that
\begin{equation}\label{growth_new}
|g_k(x,\xi)|\leq \alpha_k(1+|\xi|),\qquad \forall x\in\mt,\,\xi\in\mr.
\end{equation}
and
\begin{equation}\label{skorolip}
\sum_{k\geq1}|g_k(x,\xi)-g_k(y,\zeta)|^2\leq C\big(|x-y|^2+|\xi-\zeta|^2\big),\qquad\forall x,y\in\mt ,\,\xi,\zeta\in\mr.
\end{equation}
Consequently, denoting $G^2(x,\xi)=\sum_{k\geq 1}|g_k(x,\xi)|^2$ it holds
$$
G^2(x,\xi)\leq 2D(1+|\xi|^2)\qquad \forall x\in\mt,\,\xi\in\mr.
$$
The conditions imposed on $\varPhi$, particularly assumption \eqref{linrust}, imply that
$$\varPhi:L^2(\mt )\longrightarrow L_2(\mathfrak{U};L^2(\mt )),$$
where $L_2(\mathfrak{U};L^2(\mt ))$ denotes the collection of Hilbert-Schmidt operators from $\mathfrak{U}$ to $L^2(\mt )$. Thus, given a predictable process $u\in L^2(\Omega;L^2(0,T;L^2(\mt )))$, the stochastic integral $t\mapsto\int_0^t\varPhi(u)\dif W$ is a well defined process taking values in $L^2(\mt )$ (see \cite{daprato} for a detailed construction).

Finally, we define the auxiliary space $\mathfrak{U}_0\supset\mathfrak{U}$ via
$$\mathfrak{U}_0=\bigg\{v=\sum_{k\geq1}\alpha_k e_k;\;\sum_{k\geq1}\frac{\alpha_k^2}{k^2}<\infty\bigg\},$$
endowed with the norm
$$\|v\|^2_{\mathfrak{U}_0}=\sum_{k\geq1}\frac{\alpha_k^2}{k^2},\qquad v=\sum_{k\geq1}\alpha_k e_k.$$
Note that the embedding $\mathfrak{U}\hookrightarrow\mathfrak{U}_0$ is Hilbert-Schmidt. Moreover, trajectories of $W$ are $\prst$-a.s. in $C([0,T];\mathfrak{U}_0)$ (see \cite{daprato}).

\subsection{Kinetic solutions}

Let us introduce the definition of kinetic solution as well as the related definitions used throughout this paper. It is a generalization of the concept of kinetic solution studied in \cite{dehovo}, which is suited for establishing well-posedness in the $L^1$-framework, that is, for initial conditions in $L^1(\Omega;L^1(\T))$. In that case, the corresponding kinetic measure is not finite and one can only prove suitable decay at infinity.

\begin{defin}[Kinetic measure]\label{meesL1}
A mapping $m$ from $\Omega$ to $\mathcal{M}^+([0,T]\times\mt \times\mr)$, the set of nonnegative Radon measures over $[0,T]\times\mt \times\mr$, is said to be a kinetic measure provided
\begin{enumerate}
 \item For all $\psi\in C_c([0,T)\times\mt \times\mr)$, the process 
$$\int_{[0,t]\times \mt \times\mr}\psi(s,x,\xi)\,\dif m(s,x,\xi)$$
is predictable.
 \item Decay of $m$ for large $\xi$: it holds true that
\begin{equation*}\label{infinityL1}
\lim_{\ell\rightarrow\infty}\frac{1}{2^\ell}\stred\,m(A_{2^\ell})=0,
\end{equation*}
where
$$A_{2^\ell}=[0,T]\times\mt \times\{\xi\in\mr;\,2^\ell\leq|\xi|\leq 2^{\ell+1}\},$$
\end{enumerate}
\end{defin}

\begin{defin}[Kinetic solution]\label{kinsolL1}
A map $u\in L^1(\Omega\times[0,T],\mathcal{P},\dif \prst\otimes\dif t;L^1(\mt  ))$ is called a kinetic solution to \eqref{eq} with initial datum $u_0$ if the following conditions are satisfied
\begin{enumerate}
\item For all $\phi\in C^\infty_c(\mr)$, $\phi\geq0$,
$$\diver\int_0^u\phi(\zeta)\sigma(\zeta)\,\dif\zeta\in L^2(\Omega\times[0,T]\times\mt  ).$$ 
\item For all $\phi_1,\phi_2\in C_c^\infty(\mr)$, $\phi_1,\phi_2\geq0$, the following chain rule formula holds true
\begin{equation}\label{eq:chainruleL1}
\diver\int^u_0\phi_1(\zeta)\phi_2(\zeta)\sigma(\zeta)\,\dif \zeta=\phi_1(u)\diver\int^u_0\phi_2(\zeta)\sigma(\zeta)\,\dif \zeta\quad\text{in}\quad  L^2(\Omega\times[0,T]\times\mt  ).
\end{equation}
\item 
\blue{Let $\phi\in C^\infty_c(\R)$, $\phi\geq0$, and let $n^\phi:\Omega\rightarrow\mathcal{M}^+([0,T]\times\mt )$ be defined as follows: for all $\varphi\in C_c^\infty([0,T] \times\mt )$,
\begin{equation}\label{eq:par_meas}
n^\phi(\varphi)=\int_0^T\int_{\mt  }\varphi(t,x)\bigg|\diver\int_0^u\sqrt{\phi(\zeta)}\,\sigma(\zeta)\,\dif \zeta\bigg|^2\dif x\,\dif t.
\end{equation}
There exists a kinetic measure $m$ such that for all $\varphi\in C^\infty_c([0,T]\times\mt)$, $\varphi\geq0$, and $\phi\in C^\infty_c(\R)$, $
\phi\geq0$, it holds $m(\varphi\phi)\geq n^\phi(\varphi)$, $\prst$-a.s.,}
%
and, in addition, the pair $(f=\ind_{u>\xi},m)$\footnote{\blue{Here $\ind_{u>\xi}$ is considered as a function of four variables, namely $(\omega,t,x,\xi)\mapsto \ind_{u(\omega,t,x)>\xi}$. }} satisfies, for all $\varphi\in C_c^\infty([0,T)\times\mt  \times\mr),$ $\prst\text{-a.s.}$,
\begin{align}\label{eq:kinformulL1}
\begin{split}
\int_0^T&\big\langle f(t),\partial_t\varphi(t)\big\rangle\dif t+\big\langle f_0,\varphi(0)\big\rangle+\int_0^T\big\langle f(t),b\cdot\nabla\varphi(t)\big\rangle\dif t+\int_0^T\big\langle f(t),A:\totdif^2 \varphi(t)\big\rangle\dif t\\
=&-\sum_{k\geq1}\int_0^T\int_{\mt  }g_k\big(x,u(t,x)\big)\varphi\big(t,x,u(t,x)\big)\dif x\,\dif\beta_k(t)\\
&-\frac{1}{2}\int_0^T\int_{\mt  }G^2\big(x,u(t,x)\big)\partial_\xi\varphi\big(t,x,u(t,x)\big)\dif x\,\dif t+m(\partial_\xi\varphi).
\end{split} 
\end{align}
\end{enumerate}
\end{defin}

The definition of a kinetic solution given in Definition \ref{kinsolL1} generalizes the definition of kinetic solutions given in \cite[Definition 2.2]{dehovo} which applies to the case of high integrability, that is, for $u \in L^p(\Omega;L^p([0,T]\times\mt ))$ for all $p\ge 1$.
\blue{We note that in particular the  chain rule \eqref{eq:chainruleL1} is weaker than the corresponding version \cite[(2.5)]{dehovo}. As a consequence, the parabolic dissipation $n^\phi$ does not necessarily define a measure with respect to the variable $\xi$. It was already mentioned above that the kinetic measure $m$ is generally not a finite measure and the decay assumption from Definition \ref{meesL1}, (ii), is weaker than the one in \cite[Definition 2.1, (ii)]{dehovo}.}

\begin{rem}
   Let $u \in L^p(\Omega;L^p([0,T]\times\mt ))$ for all $p\ge 1$. Then, $u$ is a kinetic solution to \eqref{eq} in the sense of \cite[Definition 2.2]{dehovo} if and only if $u$ is a kinetic solution in the sense of Definition \ref{kinsolL1}.
\end{rem}

\begin{rem}
We emphasize that a kinetic solution is, in fact, a class of equivalence in $L^1(\Omega\times[0,T];L^1(\mt  ))$ so not necessarily a stochastic process in the usual sense. The term {\em representative} is then used to denote an element of this class of equivalence.
\end{rem}

Let us conclude this section with two related definitions.

\begin{defin}[Young measure]
Let $(X,\lambda)$ be a finite measure space. A mapping $\nu$ from $X$ to the set of probability measures on $\mr$ is said to be a Young measure if, for all $\psi\in C_b(\mr)$, the map $z\mapsto\nu_z(\psi)$ from $X$ into $\mr$ is measurable. We say that a Young measure $\nu$ vanishes at infinity if, for all $p\geq 1$,
\begin{equation*}
\int_X\int_\mr|\xi|^p\dif\nu_z(\xi)\,\dif\lambda(z)<\infty.
\end{equation*}
\end{defin}

\begin{defin}[Kinetic function]
Let $(X,\lambda)$ be a finite measure space. A measurable function $f:X\times\mr\rightarrow[0,1]$ is said to be a kinetic function if there exists a Young measure $\nu$ on $X$ vanishing at infinity such that, for $\lambda$-a.e. $z\in X$, for all $\xi\in\mr$,
$$f(z,\xi)=\nu_z(\xi,\infty).$$
\end{defin}

\begin{rem}
Note, that if $f$ is a kinetic function then $\partial_\xi f=-\nu$ for $\lambda$-a.e. $z\in X$. Similarly, let $u$ be a kinetic solution of \eqref{eq} and consider $f=\ind_{u >\xi}$. We have $\,\partial_\xi f=-\delta_{u=\xi}$, where $\nu=\delta_{u=\xi}$ is a Young measure on $\Omega\times[0,T]\times\mt$.
Throughout the paper, we will often write $\nu_{t,x}(\xi)$ instead of $\delta_{u(t,x)=\xi}$.
\end{rem}

\subsection{Applications}\label{sec:examples}
In this section we consider the model example of a convection-diffusion SPDE with polynomial nonlinearities, that is, let $N=1$ and consider
$$\dif u+\partial_x \left(\frac{u^k}{k}\right)\dif t=\partial_x\big(|u|^{m-1}\partial_x u\big)\dif t+\varPhi(x,u)\dif W,$$
i.e.\ \eqref{eq} with $b(\xi)=B'(\xi)=\xi^{k-1}$, $A(\xi)=|\xi|^{m-1}$, for $k \ge 2$, $m>2.$ Hence, 
$$\mathcal{L}(iu,in,\xi)=i(u+\xi^{k-1}n)+|\xi|^{m-1}n^2,$$
and
\begin{equation}\label{eq:cond2}
|\mathcal{L}_\xi(iu,in,\xi)|\lesssim |\xi|^{k-2}|n|+|\xi|^{m-2}n^2.
\end{equation}
For $\eta\in C^\infty_b(\R)$ and $u\in\mr,\,n\in \mz,\,|n|\sim J$ we consider
$$ \Omega^\eta_{\mathcal L}(u,n;\delta)=\{\xi\in\supp\eta;\,|i(u+\xi^{k-1} n)+|\xi|^{m-1}n^2|\leq \delta\}$$
and observe
$$\Omega^\eta_{\mathcal L}(u,n;\delta)\subset\Omega_A\cap \Omega_b, $$
where
\begin{align*}
 \Omega_A&:=\{\xi\in\supp\eta;\,|\xi|^{m-1}|n|^2\leq \delta\},
 \\\Omega_b&:=\{\xi\in\supp\eta;\,|i(u+\xi^{k-1} n)|\leq \delta\}.
\end{align*}
Note that the set $\Omega_A$ is localized around $0$ in the sense that
$$\Omega_A=\left\{\xi\in\supp\eta;\,|\xi|\leq \left(\frac{\delta}{J^2}\right)^\frac{1}{m-1}\right\},$$
whereas the set $\Omega_b$ is moving according to the value of $u$:
$$\Omega_b=\left\{\xi\in\supp\eta;\,\left(\frac{u-\delta}{J}\right)^\frac{1}{k-1}\leq\xi\leq \left(\frac{u+\delta}{J}\right)^\frac{1}{k-1}\right\}.$$
In view of the second part of the condition \eqref{eq:non-deg-local} we choose $\beta=2$ whenever a second order operator is present. Therefore we set $\beta=2$ and $\alpha=\frac{1}{m-1}$, which yields the first part of \eqref{eq:non-deg-local} independently of $\eta$
\begin{equation}\label{eq:66}
\omega_{\mathcal{L}}^\eta(J;\delta)\lesssim\left(\frac{\delta}{J^2}\right)^\frac{1}{m-1}.
\end{equation}
Regarding the second condition, it is necessary to control the $\xi$-growth in \eqref{eq:cond2}. Our formulation of the nondegeneracy condition \eqref{eq:non-deg-local} offers two ways of doing so: either using a (compactly supported) localization $\eta$ or a weight $\vartheta$. Using the first approach, Theorem \ref{thm:regularity} yields regularity of the localized average $\bar{\eta}(u)=\int_\R\chi_{u(t,x)}(\xi)\eta(\xi)\dif\xi$ without any further integrability assumptions on the solution $u$. On the other hand, the second approach allows to obtain regularity of the solution $u$ itself, i.e. setting $\eta\equiv1$, but requires higher integrability of $u$. To be more precise, in the case of \eqref{eq:cond2} we set $\vartheta(\xi)=1+|\xi|^{k\vee m-2}$ and assume that $u\in L^p(\Omega\times[0,T]\times\T)$ for $p=2(k\vee m-2)+3$.

In the case of a purely hyperbolic equation with a polynomial nonlinearity $b(\xi)=\xi^{k-1}$, $ k\geq2,$ we obtain
$$\Omega_{\mathcal{L}}^\eta(u,n;\delta)=\left\{\xi\in\supp\eta;\,\frac{u-\delta}{J}\leq\xi^{k-1}\leq \frac{u+\delta}{J}\right\},$$
which implies the first condition in \eqref{eq:non-deg-local} independently of $\eta$ with $\alpha=\frac{1}{k-1}$, $\beta=1$. For the second condition we proceed  the same way as above.

\section{Regularity}
\label{sec:averaging}

In this section we establish a regularity result for solutions to \eqref{eq}, based on averaging techniques. Throughout this section we use the following notation: for a kinetic solution $u$, let $\chi:=\chi_u=\ind_{u>\xi}-\ind_{0>\xi}.$
Then we have, in the sense of distributions,
\begin{equation}\label{eq:f_dist_form}
   \partial_t \chi+b(\xi)\cdot\nabla \chi-A(\xi):D^2 \chi=\partial_\xi q-\sum_{k=1}^\infty (\partial_\xi \chi) g_k\dot{\beta_k}+\sum_{k=1}^\infty \delta_0 g_k\dot{\beta_k},
\end{equation}
where $q=m-\frac{1}{2}G^2\delta_{u=\xi}.$ For $\eta \in C^\infty_b(\R)$ let $\bar\eta \in C^\infty$ be such that $\bar \eta'=\eta$ and $\bar\eta(0)=0$. We then have
$$\bar\eta(u)=\int_\R\chi_{u(t,x)}(\xi)\eta(\xi)\,\dif\xi.$$

\begin{thm}\label{thm:regularity} 
Assume \eqref{eq:ass_B}, \eqref{linrust}. 
 Let $\eta \in C^\infty_b(\R;\R_+)$ and assume that there are $\alpha \in (0,1)$, $\beta>0$ and a measurable map $\vartheta \in L^\infty_{loc}(\R;[1,\infty))$ such that \eqref{eq:non-deg-local} is satisfied.
Let $\varTheta_{\eta}:\R\to\R_+$ such that $\varTheta_{\eta}' = (|\xi|^2+1)\vartheta^2(\xi)(\eta(\xi)+|\eta'|(\xi))$. If $u$ is a kinetic solution to \eqref{eq} then
   $$\bar\eta(u)=\int_\R\chi_{u(t,x)}(\xi)\eta(\xi)\,\dif\xi 
   \in L^{r}(\Omega\times[0,T]; W^{s,r}(\T)),\qquad s<\frac{\alpha^2\beta}{6(1+2\alpha)}.$$
   with $\frac{1}{r}>\frac{1-\theta}{2}+\frac{\theta}{1}$, $\theta = \frac{\alpha}{4+\alpha}$ and
   \begin{equation}\begin{aligned}\label{eq:loc_regularity}
   \|\bar\eta(u)\|_{L^{r}(\Omega\times[0,T]; W^{s,r}(\T))}
   \lesssim_\eta& 
     \|\bar\eta(|u_{0}|)\|^{1/2}_{L^1_{\omega,x}} + \|\varTheta_{\eta}(|u|) \|_{L^1_{\omega,t,x}}^{1/2}+\sup_{0\leq t\leq T}\|\bar\eta(|u|)\|_{L^1_{\omega,x}}\\
     &+ \| m\vartheta(\eta+|\eta'|)\|_{L^1_{\omega}\mathcal{M}_{t,x,\xi}} + 1.
   \end{aligned}\end{equation}
   where the constant in the inequality depends on $b$ and $A$ via the constants appearing in \eqref{eq:non-deg-local} only and on $\eta$ only via its $C^1$ norm.
\end{thm}

\begin{rem}\label{rem:123} If $\eta$ is compactly supported, then we may always take $\vartheta \equiv 1$ in \eqref{eq:non-deg-local}. Furthermore, in this case the right hand side in \eqref{eq:loc_regularity} is always finite.
\end{rem}

In order to deduce regularity for $u$ itself we choose $\eta \equiv 1$. If $\vartheta$ is a polynomial of order $p$, then $\varTheta_\eta$ is a polynomial of order $2p+3$ and by Lemma \ref{lem:energy_bounds} below we have $\|\varTheta_\eta(|u|) \|_{L^1_{\omega,t,x}}^\frac{1}{2}\lesssim \|u_0\|_{L^{2p+3}_{\omega,t,x}}^{\frac{2p+3}{2}}+1$ and $\| m\vartheta\|_{L^1_{\omega}\mathcal{M}_{t,x,\xi}} \lesssim \|u_0\|_{L^{p+2}}^{p+2} +1$. In conclusion, we obtain
\begin{cor}\label{cor:regularity}
   Suppose \eqref{eq:non-deg-local} is satisfied for $\eta \equiv 1$ and $\vartheta$ being a polynomial of order $p$. Let $u$ be the kinetic solution\footnote{Well-posedness of kinetic solutions to \eqref{eq} is proved in Section \ref{sec:L1} below.} to \eqref{eq}. Then
   \begin{align*}
      \|u\|_{L^{r}(\Omega\times[0,T]; W^{s,r}(\T))}
      &\lesssim
         \|u_0 \|_{L^{2p+3}_{\omega,x}}^{2p+3} + 1.
      \end{align*}
\end{cor}

\begin{proof}[Proof of Theorem \ref{thm:regularity}]
The proof proceeds in several steps. In the first step, the solution $\chi=\chi_u=f-\ind_{0>\xi}$ is decomposed into Littlewood-Paley blocks $\chi_J$ and subsequently each Littlewood-Paley block is decomposed according to the degeneracy of the symbol $\mathcal{L}(iu,in,\xi)$. This decomposition of $f$ serves as the basis of the following averaging techniques. In the second step, each part of the decomposition is estimated separately, relying on the non-degeneracy condition \eqref{eq:non-deg-local}. In the last step, these estimates are combined and interpolated in order to deduce the regularity of $f$. 

The principle idea of the above decomposition of $f$ follows \cite{tadmor}. However, the stochastic integral in \eqref{eq} leads to additional difficulties and requires a different treatment of the time-variable. This is resolved here by passing to the mild form (cf.\ \eqref{eq:mild_form_fJphi} below) and then estimating all occurring terms separately, interpolating the estimates in the end. 

\subsection*{\texorpdfstring{Decomposition of $\chi$}{Decomposition}}

We introduce a cut-off in time, that is, let $\phi=\phi^\lambda\in C^1([0,\infty))$ such that $0\leq \phi\leq 1$, $\phi\equiv 1$ on $[0,T-\lambda]$, $\phi\equiv 0$ on $[T,\infty)$ and $|\partial_t\phi|\leq \frac{1}{\lambda}$ for some $\lambda\in (0,1)$ to be eventually sent to $0$. For notational simplicity, we omit the superscript $\lambda$ in the following computations and let it only reappear at the end of the proof, where the passage to the limit in $
\lambda$  is discussed.

Then, $\chi\phi$ solves, in the sense of distributions,
\begin{equation}\label{eq:fphi_dist_form}
	\partial_t (\chi\phi)+b(\xi)\cdot\nabla (\chi\phi)-A(\xi):D^2(\chi\phi)=\partial_\xi (\phi q)-\sum_{k=1}^\infty \partial_\xi (\chi\phi) g_k\dot{\beta_k}+\sum_{k=1}^\infty \delta_0 \phi g_k\dot{\beta_k}+\chi\partial_t\phi.
\end{equation}
Next, we decompose $\chi$ into Littlewood-Paley blocks $\chi_J$, such that the Fourier transform in space $\widehat {\chi_J}$ is supported by frequencies $|n|\sim J$ for $J$ dyadic. This is achieved by taking a smooth partition of unity $1\equiv\varphi_0(z)+\sum_{J\gtrsim 1}\varphi(J^{-1}z)$ such that $\varphi_0$ is a bump function supported inside the ball $|z|\leq 2$ and $\varphi$ is a bump function supported in the annulus $\frac{1}{2}\leq|z|\leq 2$, and setting
\begin{align*}
\chi_0(t,x,\xi)&:=\mathcal{F}_x^{-1}\big[\varphi_0\left(n\right)\hat \chi(t,n,\xi)\big](x),\\
\chi_J(t,x,\xi)&:=\mathcal{F}_x^{-1}\bigg[\varphi\left(\frac{n}{J}\right)\hat \chi(t,n,\xi)\bigg](x),\qquad J\gtrsim 1.
\end{align*}
This leads to the decomposition
  $$\chi={\chi_0+}\sum_{J\gtrsim 1}\chi_J.$$
The regularity of $\chi_0$ being trivial, we only focus on the estimate of $\chi_J$ for $J\gtrsim 1$.  
Localizing \eqref{eq:fphi_dist_form} in Littlewood-Paley blocks yields
\begin{equation*}\begin{aligned}
	\partial_t (\chi_J\phi)+b(\xi)\cdot\nabla (\chi_J\phi)-A(\xi):D^2(\chi_J\phi)
	=&\partial_\xi (\phi q_J)-\sum_{k=1}^\infty \partial_\xi (\chi\phi g_k)_J\dot{\beta_k}
	+\sum_{k=1}^\infty  (\chi\phi \partial_\xi g_k)_J\dot{\beta_k}
	\\&+\sum_{k=1}^\infty \delta_0 \phi (g_k)_J\dot{\beta_k}+\chi_J\partial_t\phi.
\end{aligned}\end{equation*}
After a preliminary step of regularization, we may test by $S^*(T-t)\varphi$ for $\varphi\in C(\T)$ in \eqref{eq:f_dist_form}, where $S(t)$ denotes the solution semigroup to the linear operator 
  $$\chi\mapsto b(\xi)\cdot\nabla \chi-A(\xi):D^2 \chi.$$
This leads to the mild form
\begin{align}
\label{eq:mild_form_fJphi}
(\chi_J\phi)(t)
&= S(t)\chi_{J}(0)+\int_0^t S(t-s)\partial_\xi (\phi q_J) \,\dif s-\sum_{k=1}^\infty\int_0^t S(t-s)\partial_\xi (g_k \chi\phi)_J \,\dif \beta_k(s)\nonumber\\
 &\quad+\sum_{k=1}^\infty\int_0^t S(t-s)((\partial_\xi g_k)\chi\phi)_J \,\dif \beta_k(s) +\sum_{k=1}^\infty\int_0^t S(t-s)\delta_0 \phi g_{k,J}\,\dif \beta_k(s)\\
 &\quad+\int_0^t S(t-s)\chi_J\partial_t\phi \,\dif s,\nonumber
\end{align}
where we have used 
\begin{align*}
&\varphi\left(\frac{n}{J}\right)\mathcal{F}_x\bigg[\int_0^tS(t-s)\partial_\xi (\chi\phi)\,g_k\,\dif \beta_k(s)\bigg](n)\\
&=\varphi\left(\frac{n}{J}\right)\int_0^t\me^{-(ib(\xi)\cdot n+n^*A(\xi)n)(t-s)}\partial_\xi \mathcal{F}_x(g_k \chi\phi)(s,n,\xi)\,\dif\beta_k(s)\\
&-\varphi\left(\frac{n}{J}\right)\int_0^t\me^{-(ib(\xi)\cdot n+n^*A(\xi)n)(t-s)} \mathcal{F}_x((\partial_\xi g_k) \chi\phi)(s,n,\xi)\,\dif\beta_k(s)\\
&=\int_0^t\me^{-(ib(\xi)\cdot n+n^*A(\xi)n)(t-s)}\partial_\xi \mathcal{F}_x(g_k \chi\phi)_J(s,n,\xi)\,\dif \beta_k(s)\\
&-\int_0^t\me^{-(ib(\xi)\cdot n+n^*A(\xi)n)(t-s)} \mathcal{F}_x((\partial_\xi g_k) \chi\phi)_J(s,n,\xi)\,\dif \beta_k(s)
\end{align*}
and
\begin{align*}
&\varphi\left(\frac{n}{J}\right)\mathcal{F}_x\bigg[\int_0^tS(t-s)\delta_0\phi g_k\,\dif \beta_k(s)\bigg](n)\\
&=\varphi\left(\frac{n}{J}\right)\int_0^t\me^{-(ib(\xi)\cdot n+n^*A(\xi)n)(t-s)}\delta_0 \phi\widehat {g_k}(n,\xi)\,\dif \beta_k(s)\\
&=\int_0^t\me^{-(ib(\xi)\cdot n+n^*A(\xi)n)(t-s)}\delta_0 \phi\widehat{ g_{k,J}}(n,\xi)\,\dif \beta_k(s).
\end{align*}

For $J\gtrsim 1$ fixed, we next decompose the action in $\xi$-variable according to the degeneracy of the operator
$\mathcal{L}(iu,in,\xi).$
Namely, for $K$ dyadic, let $1\equiv\psi_0(z)+\sum_{K\gtrsim 1}\psi_1(K^{-1}z)$ be a smooth partition of unity such that $\psi_0$ is a bump function supported inside the ball $|z|\leq 2$ and $\psi_1$ is a bump function supported in the annulus $\frac{1}{2}\leq|z|\leq2$, and write
\begin{align*}
\ind_{0\leq t}(\chi_J\phi)(t,x,\xi)&=\mathcal{F}_{tx}^{-1}\bigg[\psi_0\left(\frac{\mathcal{L}(iu,in,\xi)}{\delta}\right)\mathcal{F}_{tx}\big[\ind_{0\leq t}(\chi_J\phi)\big](u,n,\xi)\bigg](t,x)\\
&\quad+\sum_{K\gtrsim 1}\mathcal{F}_{tx}^{-1}\bigg[\psi_1\left(\frac{\mathcal{L}(iu,in,\xi)}{\delta K}\right)\mathcal{F}_{tx}\big[\ind_{0\leq t}(\chi_J\phi)\big](u,n,\xi)\bigg](t,x)\\
&=:\chi_J^{(0)}(t,x,\xi)+\sum_{K\gtrsim 1}\chi_J^{(K)}(t,x,\xi).
\end{align*}
Hence, we consider the decomposition
\begin{align*}
 \ind_{0\leq t} \chi\phi=\ind_{0\leq t}\Big(\chi_0+ \sum_{J\gtrsim 1}\chi_J\Big)\phi =\ind_{0\leq t}\chi_0\phi+ \sum_{J\gtrsim 1} \left( \chi_J^{(0)}(t,x,\xi)+\sum_{K\gtrsim 1}\chi_J^{(K)}(t,x,\xi)\right).
\end{align*}
Since $\psi_0$ is supported at the degeneracy, we will apply a trivial estimate. However, $\psi_1$ is supported away from the degeneracy and therefore we may use the equation and the non-degeneracy assumption \eqref{eq:non-deg-local}. From \eqref{eq:mild_form_fJphi} we obtain 
\begin{align*}
	\chi_J^{(K)}(t,x,\xi)&=\mathcal{F}_{tx}^{-1}\psi_1\left(\frac{\mathcal{L}(iu,in,\xi)}{\delta K}\right)\mathcal{F}_{tx}\bigg[\ind_{0\leq t}S(t)\chi_{0,J}+\ind_{0\leq t}\int_0^t S(t-s)\partial_\xi(\phi q_J) \,\dif s\\
	&\quad\quad-\ind_{0\leq t}\sum_{k=1}^\infty\int_0^t S(t-s)\partial_\xi (g_k\chi\phi)_J \,\dif \beta_k(s)+\ind_{0\leq t}\sum_{k=1}^\infty\int_0^t S(t-s) ((\partial_\xi g_k)\chi\phi)_J \,\dif \beta_k(s)\\
	&\quad\quad+\ind_{0\leq t}\sum_{k=1}^\infty\int_0^t S(t-s)\delta_0 \phi g_{k,J}\,\dif \beta_k(s)+\ind_{0\leq t}\int_0^t S(t-s)\chi_J\partial_t\phi \,\dif s\bigg](t,x).
\end{align*}
Multiplying the above by $\eta\in C^\infty_b(\R)$ and integrating over $\xi\in\R$, we set
\begin{align*}
\int_\R \chi_J^{(K)}(t,x,\xi)\eta(\xi)\dif\xi=:I_1+I_2-I_3+I_4+I_5+I_6
\end{align*}
and we estimate the right hand side term by term below.
Note that since
\begin{align*}
&\mathcal{F}_{tx}\bigg[\ind_{0\leq t}\int_0^t S(t-s)\partial_\xi (g_k \chi\phi)_J \,\dif \beta_k(s)\bigg]\\
&=\mathcal{F}_t\bigg[\ind_{0\leq t}\int_0^t\me^{-(ib(\xi)\cdot n+n^*A(\xi)n)(t-s)}\partial_\xi\widehat{(g_k \chi\phi)}_J(s,\xi,n)\,\dif \beta_k(s)\bigg]\\
&=\frac{1}{(2\pi)^{1/2}}\int\int\ind_{0\leq s\leq t}\,\me^{-(ib(\xi)\cdot n+n^*A(\xi)n)(t-s)}\partial_\xi\widehat{(g_k \chi\phi)}_J\,\dif \beta_k(s)\,\me^{- itu}\dif t\\
&=\frac{1}{(2\pi)^{1/2}}\int\int\ind_{t-s\geq 0}\,\me^{-(ib(\xi)\cdot n+n^*A(\xi)n)(t-s)}\,\me^{- itu}\dif t\,\ind_{0\leq s}\,\partial_\xi\widehat{(g_k \chi\phi)}_J(s,\xi,n)\,\dif \beta_k(s)\\
&=\int\ind_{r\geq 0}\,\me^{-(ib(\xi)\cdot n+n^*A(\xi)n)r}\,\me^{- iru}\dif r\,\frac{1}{(2\pi)^{1/2}}\int_0^\infty\,\partial_\xi\widehat{(g_k \chi\phi)}_J(s,\xi,n)\me^{- isu}\,\dif \beta_k(s)\\
&=\frac{1}{i(u+b(\xi)\cdot n)+n^*A(\xi)n}\,\frac{1}{(2\pi)^{1/2}}\int_0^\infty\,\partial_\xi\widehat{(g_k \chi\phi)}_J(s,\xi,n)\me^{- isu}\,\dif \beta_k(s),
\end{align*}
we have
\begin{align*}
  I_3 = {\frac{1}{(2\pi)^{1/2}}}\frac{1}{(\delta K)}\int_\mr\mathcal{F}_{tx}^{-1}\bigg[\tilde\psi\left(\frac{\mathcal{L}(iu,in,\xi)}{\delta K}\right)\sum_{k=1}^\infty\int_0^\infty\me^{-isu}\partial_\xi\widehat {(g_k\chi\phi)}_{J}(s,\xi,n)\,\dif \beta_k(s)\bigg]\eta(\xi)\dif \xi,
\end{align*}
where
$$\tilde\psi(z):=\psi_1(z)/z.$$ We argue similarly for the remaining terms $I_i$, e.g.\ for $I_2$ we note that
\begin{align*}
\mathcal{F}_{tx}\bigg[\ind_{0\leq t}\int_0^t S(t-s)\partial_\xi (\phi q_J) \,\dif s\bigg]&=\frac{1}{i(u+b(\xi)\cdot n)+n^*A(\xi)n}\,\mathcal{F}_{tx} [\ind_{0\leq t}\partial_\xi (\phi q_J)]
\end{align*}

\subsection*{Estimating \texorpdfstring{$I_i$, $i=1,\dots,6$}{I1,...,I6}}

\subsubsection*{Estimate of \texorpdfstring{$I_{1}$}{I1}}

Using Plancherel and H\"older's inequality we observe
\begin{align*}
\begin{aligned}
\|I_1\|_{L^2_{t,x}}^2
&=\frac{1}{2\pi(\delta K)^2}\bigg\|\int_\mr\mathcal{F}_{tx}^{-1}\bigg[\tilde\psi\bigg(\frac{\mathcal{L}(iu,in,\xi)}{\delta K}\bigg)\hat \chi_{J}(0,n,\xi)\bigg]\eta(\xi)\dif\xi\bigg\|_{L^2_{t,x}}^2\\
&=\frac{1}{2\pi(\delta K)^2}\int_u\sum_n\bigg|\int_\mr\tilde\psi\bigg(\frac{\mathcal{L}(iu,in,\xi)}{\delta K}\bigg)\hat \chi_{J}(0,n,\xi)\eta(\xi)\dif\xi\bigg|^2\dif u\\
&\leq \frac{1}{2\pi(\delta K)^2}\int_u\sum_n\int_\mr\bigg|\tilde\psi\bigg(\frac{\mathcal{L}(iu,in,\xi)}{\delta K}\bigg)\bigg|^2\ind_{\supp \eta}\dif\xi\\
&\qquad\quad\times\int_\mr\ind_{\{|u+b(\xi)\cdot n|^2+|n^*A(\xi)n|^2<(2\delta K)^2\}}\big|\hat \chi_{J}(0,n,\xi)\big|^2\eta^2(\xi)\dif \xi\dif u.
\end{aligned}
\end{align*}
Then using \eqref{eq:non-deg-local} and
\begin{align}\label{intu}
\int_u\ind_{\{|u+b(\xi)\cdot n|^2+|n^*A(\xi)n|^2<(2\delta K)^2\}}\dif u\leq\int_u\ind_{\{|u|^2<(2\delta K)^2\}}\dif u\lesssim \delta K
\end{align}
we obtain
\begin{align}\label{est:I1}
\begin{aligned}
\|I_1\|_{L^2_{t,x}}^2\lesssim \frac{1}{\delta K}\left(\frac{\delta K}{J^\beta}\right)^\alpha\|\chi_{J}(0)\eta\|_{L^2_{x,\xi}}^2.
\end{aligned}
\end{align}

\subsubsection*{Estimate of $I_{2}$}

First, we integrate by parts to obtain
\begin{align*}
\begin{aligned}
\|I_2\|_{L^1_{t}W^{-\epsilon,q_\epsilon}_x}
&=\frac{1}{\delta K}\bigg\|\int_\mr\mathcal{F}_{tx}^{-1}\bigg[\tilde\psi\left(\frac{\mathcal{L}(iu,in,\xi)}{\delta K}\right)\mathcal{F}_{tx}(\ind_{0\leq t}\phi\partial_\xi q_{J})\bigg]\eta(\xi)\dif\xi\bigg\|_{L^1_{t}W^{-\epsilon,q_\epsilon}_x}\\
&\leq\frac{1}{(\delta K)^2}\bigg\|\int_\mr\mathcal{F}_{tx}^{-1}\bigg[\tilde\psi'\left(\frac{\mathcal{L}(iu,in,\xi)}{\delta K}\right)\frac{\mathcal{L}_\xi(iu,in,\xi)}{\vartheta(\xi)} \mathcal{F}_{tx}(\ind_{0\leq t}\phi \vartheta q_{J})(u,\xi,n)\bigg]\eta(\xi)\dif \xi\bigg\|_{L^1_{t}W^{-\epsilon,q_\epsilon}_x}\\
&\quad+\frac{1}{(\delta K)^2}\bigg\|\int_\mr\mathcal{F}_{tx}^{-1}\bigg[\tilde\psi\left(\frac{\mathcal{L}(iu,in,\xi)}{\delta K}\right)\mathcal{F}_{tx}(\ind_{0\leq t}\phi q_{J})(u,\xi,n)\bigg]\eta'(\xi)\dif \xi\bigg\|_{L^1_{t}W^{-\epsilon,q_\epsilon}_x}
\end{aligned}
\end{align*}
We apply  Corollary \ref{cor:mult} to estimate the second term on the right hand side. For the first one, we first note that \eqref{eq:non-deg-local} implies that (for simplicity restricting to the case $\beta =2$ while $\beta=1$ can be handled analogously)
$$
\frac{|b'_{i}(\xi)|}{\vartheta(\xi)}\lesssim_\eta |n|,\qquad\frac{|A_{ij}'(\xi)|}{\vartheta(\xi)}\lesssim_\eta 1,\qquad i,j\in\{1,\dots, N\}.
$$
Since $\mathcal{L}_\xi(iu,in,\xi)$ is a polynomial in $n$, we may apply 
\cite[Lemma 2.2]{BCD} to deduce that$$m(n,\xi):=\frac{\mathcal{L}_\xi(iu,in,\xi)}{\vartheta(\xi)}=\frac{ib'(\xi)\cdot n+n^* A'(\xi)n}{\vartheta(\xi)}$$
localized to $|n|\sim J$, $\xi \in \supp\eta$, is an $L^1$-Fourier multiplier with norm bounded by $J^\beta$. Revisiting the proofs of Lemma \ref{lemma:multiplier} and Corollary \ref{cor:mult} with the multiplier $\psi\left(\frac{m(u,n,\xi)}{\delta}\right)$ replaced by $\psi\left(\frac{m(u,n,\xi)}{\delta}\right)m(n,\xi)$ then allows to estimate the first term on the right hand side to obtain
\begin{align*}
\begin{aligned}
\|I_2\|_{L^1_{t}W^{-\epsilon,q_\epsilon}_x}
&\lesssim \frac{1}{(\delta K)^2}J^{\beta}\|\phi\vartheta q_{J}\eta\|_{\mathcal{M}_{t,x,\xi}}+\frac{1}{(\delta K)^2}\|\phi q_{J}\eta'\|_{\mathcal{M}_{t,x,\xi}}\lesssim \frac{1}{(\delta K)^2}J^{\beta}\|\phi\vartheta q_{J}(\eta+|\eta'|)\|_{\mathcal{M}_{t,x,\xi}},
\end{aligned}
\end{align*}
where
$$
\frac{N}{q'_\epsilon}<\epsilon<1<q_\epsilon<\frac{N}{N-\epsilon}
$$
and $\epsilon$ is chosen sufficiently small.
Consequently,
\begin{align}\label{est:I2}
\begin{aligned}
\|I_2\|_{L^1_{t,x}}
&\lesssim \frac{1}{(\delta K)^2}J^{\beta+\epsilon}\|\phi\vartheta q_{J}(\eta+|\eta'|)\|_{\mathcal{M}_{t,x,\xi}}.
\end{aligned}
\end{align}

\subsubsection*{Estimate of $I_{3}$}
Using Plancherel and It\^o's formula, we note that
\begin{align*}
\E\|I_3\|_{L^2_{t,x}}^2&=\frac{1}{2\pi(\delta K)^2}\stred\bigg\|\int_\mr\mathcal{F}_{tx}^{-1}\bigg[\tilde\psi\left(\frac{\mathcal{L}(iu,in,\xi)}{\delta K}\right)\sum_{k=1}^\infty\int_0^\infty\me^{-isu}\partial_\xi\widehat {(g_k\chi\phi)}_{J}(s,\xi,n)\,\dif \beta_k(s)\bigg]\eta(\xi)\dif \xi\bigg\|_{L^2_{t,x}}^2\\
&=\frac{1}{2\pi(\delta K)^2}\int_u\sum_n\stred\bigg|\sum_{k=1}^\infty\int_0^\infty\int_\mr\tilde\psi\left(\frac{\mathcal{L}(iu,in,\xi)}{\delta K}\right)\me^{-isu}\partial_\xi\widehat {(g_k\chi\phi)}_{J}(s,\xi,n)\eta(\xi)\dif \xi\,\dif \beta_k(s)\bigg|^2\dif u\\
&=\frac{1}{2\pi(\delta K)^2}\int_u\sum_n\stred\int_0^\infty\sum_{k=1}^\infty\bigg|\int_\mr\tilde\psi\left(\frac{\mathcal{L}(iu,in,\xi)}{\delta K}\right)\me^{-isu}\partial_\xi\widehat {(g_k\chi\phi)}_{J}(s,\xi,n)\eta(\xi)\dif \xi\bigg|^2\dif s\dif u\\
&\lesssim\frac{1}{(\delta K)^4}\,\int_u\sum_n\stred\int_0^\infty\sum_{k=1}^\infty\bigg|\int_\mr\tilde\psi'\left(\frac{\mathcal{L}(iu,in,\xi)}{\delta K}\right)\frac{\mathcal{L}_\xi(iu,in,\xi)}{\vartheta(\xi)}\widehat {(g_k\chi\vartheta\phi)}_{J}(s,\xi,n)\eta(\xi)\dif \xi\bigg|^2\dif s\dif u\\
&\quad+\frac{1}{(\delta K)^4}\,\int_u\sum_n\stred\int_0^\infty\sum_{k=1}^\infty\bigg|\int_\mr\tilde\psi\left(\frac{\mathcal{L}(iu,in,\xi)}{\delta K}\right)\widehat {(g_k\chi\phi)}_{J}(s,\xi,n)\eta'(\xi)\dif \xi\bigg|^2\dif s\dif u\\
\end{align*}
Hence, by \eqref{eq:non-deg-local} it follows that
\begin{align*}
&\E\|I_3\|_{L^2_{t,x}}^2\\
&\lesssim\frac{1}{(\delta K)^4}\,\stred\int_0^\infty\int_u\sum_n\int_\mr\bigg|\tilde\psi'\left(\frac{\mathcal{L}(iu,in,\xi)}{\delta K}\right)\frac{\mathcal{L}_\xi(iu,in,\xi)}{\vartheta(\xi)}\bigg|^2\ind_{\supp \eta}\dif\xi\\
&\quad\quad\times\int_\mr\ind_{\{|u+b(\xi)\cdot n|^2+|n^*A(\xi)n|^2<(2\delta K)^2}\}\sum_{k=1}^\infty\big|\widehat {(g_k\chi\vartheta\phi)}_{J}(s,\xi,n)\big|^2\eta^2(\xi)\dif \xi\dif u\dif s\\
&\quad+\frac{1}{(\delta K)^4}\,\stred\int_0^\infty\int_u\sum_n\int_\mr\bigg|\tilde\psi\left(\frac{\mathcal{L}(iu,in,\xi)}{\delta K}\right)\bigg|^2\ind_{\supp \eta}\dif\xi\\
&\quad\quad\times\int_\mr\ind_{\{|u+b(\xi)\cdot n|^2+|n^*A(\xi)n|^2<(2\delta K)^2}\}\sum_{k=1}^\infty\big|\widehat {(g_k\chi\vartheta\phi)}_{J}(s,\xi,n)\big|^2|\eta'(\xi)|^2\dif \xi\dif u\dif s\\
&\lesssim \frac{1}{(\delta K)^4}\left(\frac{\delta K}{J^\beta}\right)^{\alpha}J^{2\beta}\\
&\quad\quad\times\stred\int_0^\infty\int_u\sum_n\int_{\mr}\ind_{\{|u+b(\xi)\cdot n|^2+|n^*A(\xi)n|^2<(2\delta K)^2\}}\sum_{k=1}^\infty\big|\widehat {(g_k\chi\vartheta\phi)}_{J}(s,\xi,n)\big|^2\eta^2(\xi)\dif \xi\dif u\dif s\\
&+ \frac{1}{(\delta K)^4}\left(\frac{\delta K}{J^\beta}\right)^{\alpha}\\
&\quad\quad\times\stred\int_0^\infty\int_u\sum_n\int_{\mr}\ind_{\{|u+b(\xi)\cdot n|^2+|n^*A(\xi)n|^2<(2\delta K)^2\}}\sum_{k=1}^\infty\big|\widehat {(g_k\chi\vartheta\phi)}_{J}(s,\xi,n)\big|^2|\eta'(\xi)|^2\dif \xi\dif u\dif s
\end{align*}
and due to \eqref{intu} we obtain
\begin{align}\label{est:I3}
\begin{aligned}
\E\|I_3\|_{L^2_{t,x}}^2&\lesssim \frac{1}{(\delta K) ^3}\left(\frac{\delta K}{J^\beta}\right)^{\alpha}J^{2\beta}\,\stred\int_0^\infty\sum_n\int_{\R}\sum_{k=1}^\infty\big|\widehat {(g_k\chi\vartheta\phi)}_{J}(s,\xi,n)\big|^2(\eta^2+|\eta'|^2)\,\dif \xi\dif s\\
&\le \frac{1}{(\delta K)^3}\left(\frac{\delta K}{J^\beta}\right)^{\alpha}J^{2\beta}\sum_{k=1}^\infty\stred\|(g_k\chi\vartheta\phi)_{J}(\eta+|\eta'|)\|_{L^2_{t,x,\xi}}^2.
\end{aligned}
\end{align}

\subsubsection*{Estimate of $I_{4}$}
By Plancherel and It\^o's formula we have
\begin{align*}
&\E\|I_4\|_{L^2_{t,x}}^2\\
&=\frac{1}{2\pi(\delta K)^2}\stred\bigg\|\int_\mr\mathcal{F}_{tx}^{-1}\bigg[\tilde\psi\left(\frac{\mathcal{L}(iu,in,\xi)}{\delta K}\right)\sum_{k=1}^\infty\int_0^\infty\me^{-isu} \mathcal{F}_x{((\partial_\xi g_k)\chi\phi)}_{J}(s,\xi,n)\,\dif \beta_k(s)\bigg]\eta(\xi)\dif \xi\bigg\|_{L^2_{t,x}}^2\\
&=\frac{1}{2\pi(\delta K)^2}\int_u\sum_n\stred\bigg|\sum_{k=1}^\infty\int_0^\infty\int_\mr\tilde\psi\left(\frac{\mathcal{L}(iu,in,\xi)}{\delta K}\right)\me^{-isu}\mathcal{F}_x {((\partial_\xi g_k)\chi\phi)}_{J}(s,\xi,n)\eta(\xi)\dif \xi\,\dif \beta_k(s)\bigg|^2\dif u\\
&=\frac{1}{2\pi(\delta K)^2}\int_u\sum_n\stred\int_0^\infty\sum_{k=1}^\infty\bigg|\int_\mr\tilde\psi\left(\frac{\mathcal{L}(iu,in,\xi)}{\delta K}\right)\me^{-isu}\mathcal{F}_x {((\partial_\xi g_k)\chi\phi)}_{J}(s,\xi,n)\eta(\xi)\dif \xi\bigg|^2\dif s\dif u\\
\end{align*}
Thus, it follows from \eqref{eq:non-deg-local}, \eqref{intu} and \eqref{linrust} that
\begin{align}\label{est:I4}
\begin{aligned}
&\E\|I_4\|_{L^2_{t,x}}^2\\
&\leq \frac{1}{(\delta K)^2}\int_u\sum_n\int_\mr\bigg|\tilde\psi\bigg(\frac{\mathcal{L}(iu,in,\xi)}{\delta K}\bigg)\bigg|^2\ind_{\supp \eta}\dif\xi\\
&\qquad\quad\times\E\int_0^\infty\int_\mr\ind_{\{|u+b(\xi)\cdot n|^2+|n^*A(\xi)n|^2<(2\delta K)^2\}}\sum_{k=1}^\infty\big|\mathcal{F}_x {((\partial_\xi g_k)\chi\phi)}_{J}(s,\xi,n)\big|^2\eta^2(\xi)\dif \xi\,\dif s\dif u\\
&\lesssim \frac{1}{\delta K}\left(\frac{\delta K}{J^\beta}\right)^\alpha\sum_{k=1}^\infty\E\|{(\partial_\xi g_k)\chi\eta\phi)}_{J}\|_{L^2_{t,x,\xi}}^2.
\end{aligned}
\end{align}

\subsubsection*{Estimate of $I_{5}$}

We have
\begin{align}\label{est:I5}
\begin{aligned}
\E&\|I_5\|_{L^2_{t,x}}^2\\
&=\frac{1}{2\pi(\delta K)^2}\stred\bigg\|\int_\mr\mathcal{F}_{tx}^{-1}\bigg[\tilde\psi\left(\frac{\mathcal{L}(iu,in,\xi)}{\delta K}\right)\sum_{k=1}^\infty\int_0^\infty\me^{-isu}\delta_0\phi\widehat {g_{k,J}}(n,\xi)\,\dif \beta_k(s)\bigg]\eta(\xi)\dif \xi\bigg\|_{L^2_{t,x}}^2\\
&=\frac{1}{2\pi(\delta K)^2}\stred\bigg\|\int_\mr\tilde\psi\left(\frac{\mathcal{L}(iu,in,\xi)}{\delta K}\right)\sum_{k=1}^\infty\int_0^\infty\me^{-isu}\delta_0\phi\widehat {g_{k,J}}(n,\xi)\,\dif \beta_k(s)\eta(\xi)\dif \xi\bigg\|_{L^2_{u,n}}^2\\
&=\frac{\eta(0)^2}{2\pi(\delta K)^2}\int_u\sum_n\stred\bigg|\sum_{k=1}^\infty\int_0^\infty\tilde\psi\left(\frac{\mathcal{L}(iu,in,0)}{\delta K}\right)\me^{-isu}\phi\widehat {g_{k,J}}(n,0)\,\dif \beta_k(s)\bigg|^2\dif u\\
&=\frac{\eta(0)^2}{2\pi(\delta K)^2}\int_u\sum_n\int_0^\infty\sum_{k=1}^\infty\bigg|\tilde\psi\left(\frac{\mathcal{L}(iu,in,0)}{\delta K}\right)\phi\widehat {g_{k,J}}(n,0)\bigg|^2\dif s\dif u\\
&\lesssim_\eta \frac{1}{\delta K}\sum_{k=1}^\infty\|g_{k,J}(\cdot,0)\|_{L^2_{x}}^2,
\end{aligned}
\end{align}
where we also used \eqref{intu}.

\subsubsection*{Estimate of $I_{6}$}

It holds
\begin{align}\label{est:I6}
\begin{aligned}
\|I_6\|_{L^1_{t,x}}
&\lesssim\frac{1}{\delta K}\bigg\|\int_\mr\mathcal{F}_{tx}^{-1}\bigg[\tilde\psi\left(\frac{\mathcal{L}(iu,in,\xi)}{\delta K}\right)\mathcal{F}_{tx}(\ind_{0\leq t}\chi_{J}\partial_t\phi)\bigg]\eta(\xi)\dif\xi\bigg\|_{L^1_{t,x}}\\
&\lesssim \frac{1}{\delta K}\|\partial_t \phi \chi_{J}\eta\|_{L^1_{t,x,\xi}}.
\end{aligned}
\end{align}

\subsection*{Estimating \texorpdfstring{$\chi_J^{(K)}$}{chi\_JK}}
Using the above estimates \eqref{est:I1}, \eqref{est:I2}, \eqref{est:I3}, \eqref{est:I4}, \eqref{est:I5}, \eqref{est:I6}
we deduce
\begin{align*}
\bigg\| \int_\mr \chi^{(K)}_J\,\eta\,\dif \xi\bigg\|_{L^1_{\omega,t,x}}
&=\sup_{\substack{\varphi\in L^\infty_{\omega,t,x}\\\|\varphi\|_{L^\infty_{\omega,t,x}}\leq 1}}\stred\bigg\langle \int_\mr \chi^{(K)}_J\,\eta\,\dif \xi,\varphi\bigg\rangle\\
&\lesssim_\eta 
\left(\frac{1}{\delta K}\right)^\frac{1}{2}\left(\frac{\delta K}{J^\beta}\right)^\frac{\alpha}{2}\|\chi_{J}(0)\eta\|_{L^2_{\omega,x,\xi}} + \frac{1}{(\delta K)^2}J^{\beta+\epsilon} \|\phi\vartheta q_J(\eta+|\eta'|)\|_{L^1_{\omega}\mathcal{M}_{t,x,\xi}}\\
&\quad + \del{T^\frac{1}{2}}
\left(\frac{1}{(\delta K)^3}\right)^\frac{1}{2}\left(\frac{\delta K}{J^\beta}\right)^{\frac{\alpha}{2}}J^{\beta}\left(\sum_{k=1}^\infty\|(g_k\chi\vartheta\phi)_{J}(\eta+|\eta'|)\|_{L^2_{\omega,t,x,\xi}}^2\right)^{1/2}\\
&\quad+ 
\left(\frac{1}{\delta K}\right)^\frac{1}{2}\left(\frac{\delta K}{J^\beta}\right)^\frac{\alpha}{2}\left(\sum_{k=1}^\infty\|{((\partial_\xi g_k)\chi\eta\phi)}_{J}\|_{L^2_{\omega,t,x,\xi}}^2\right)^{1/2} \\
&\quad+ \del{T^\frac{1}{2}}
\left(\frac{1}{\delta K}\right)^\frac{1}{2}\left(\sum_{k=1}^\infty\|g_{k,J}(\cdot,0)\|_{L^2_x}^2\right)^{1/2} + \frac{1}{\delta K}\|\partial_t \phi \chi_{J}\eta\|_{L^1_{\omega,t,x,\xi}}.
\end{align*}
Hence,
\begin{align*}
\sum_{K\gtrsim 1}\bigg\| \int_\mr \chi^{(K)}_J\,\eta\,\dif \xi\bigg\|_{L^1_{\omega,t,x}}
&\lesssim \delta^{\frac{1}{2}(\alpha-1)}J^{-\frac{\beta\alpha}{2}}\|\chi_{J}(0)\eta\|_{L^2_{\omega,x,\xi}}\sum_{K\gtrsim 1}K^{\frac{1}{2}(\alpha-1)} \\
&\quad+ \delta^{-2}J^{\beta+\epsilon}\|\phi \vartheta q_J(\eta+|\eta'|)\|_{L^1_{\omega}\mathcal{M}_{t,x,\xi}}\sum_{K\gtrsim 1}K^{{-2}}\\
&\quad+ \del{T^\frac{1}{2}}
\delta^{-\frac{3-\alpha}{2}}J^{\frac{2-\alpha}{2}\beta}\left(\sum_{k=1}^\infty\|(g_k\chi\vartheta\phi)_{J}(\eta+|\eta'|)\|_{L^2_{\omega,t,x,\xi}}^2\right)^{1/2}\sum_{K\gtrsim 1}K^{-\frac{3-\alpha}{2}}\\
&\quad+\del{T^\frac{1}{2}}
\delta^{-\frac{1}{2}(1-\alpha)}J^{-\frac{\beta\alpha}{2}}\left(\sum_{k=1}^\infty\|{((\partial_\xi g_k)\chi\eta\phi)}_{J}\|_{L^2_{\omega,t,x,\xi}}^2\right)^{1/2}\sum_{K\gtrsim 1}K^{\frac{1}{2}(\alpha-1)}\\
&\quad+ \del{T^\frac{1}{2}}
\delta^{-\frac{1}{2}}\left(\sum_{k=1}^\infty\|g_{k,J}(\cdot,0)\|_{L^2_x}^2\right)^{1/2}\sum_{K\gtrsim 1}K^{-\frac{1}{2}}+\delta^{-1}\|\partial_t \phi \chi_{J}\eta\|_{L^1_{\omega,t,x,\xi}}\sum_{K\gtrsim 1}K^{-1}.
\end{align*}
Recall that the parameter $K$ was chosen dyadic. Therefore, since all the powers of $K$ appearing on the right hand side are negative, we deduce
\begin{align*}
\sum_{K\gtrsim 1}\bigg\| \int_\mr \chi^{(K)}_J\,\eta\,\dif \xi\bigg\|_{L^1_{\omega,t,x}}
&\lesssim
\del{T^\frac{1}{2}}\delta^{\frac{1}{2}(\alpha-1)}J^{-\frac{\beta\alpha}{2}}\|\chi_{J}(0)\eta\|_{L^2_{\omega,x,\xi}} + \delta^{{-2}}J^{\beta+\epsilon}\|\phi\vartheta q_J(\eta+|\eta'|)\|_{L^1_{\omega}\mathcal{M}_{t,x,\xi}}\\
&\quad+ \delta^{{-\frac{3-\alpha}{2}}}J^{\frac{2-\alpha}{2}\beta}\left(\sum_{k=1}^\infty\|(g_k\chi\vartheta\phi)_{J}(\eta+|\eta'|)\|_{L^2_{\omega,t,x,\xi}}^2\right)^{1/2}\\
&\quad+ \delta^{-\frac{1}{2}(1-\alpha)}J^{-\frac{\beta\alpha}{2}}\left(\sum_{k=1}^\infty\|{((\partial_\xi g_k)\chi\eta\phi)}_{J}\|_{L^2_{\omega,t,x,\xi}}^2\right)^{1/2}\\
&\quad+ \del{T^\frac{1}{2}}\delta^{-\frac{1}{2}}\left(\sum_{k=1}^\infty\|g_{k,J}(\cdot,0)\|_{L^2_x}^2\right)^{1/2} +\delta^{-1}\|\partial_t \phi \chi_{J}\eta\|_{L^1_{\omega,t,x,\xi}}
\end{align*}

\subsection*{Estimating \texorpdfstring{$\chi_J$}{chi\_J}}

We let  
  $$\overline{\chi_J\phi\,\eta}:=\int_\mr \chi_J\phi\,\eta\,\dd \xi$$ 
and write
\begin{align*}
\begin{aligned}
K(r, \overline{\chi_J\phi\,\eta})&:=\inf_{\substack{F_J^0\in L^2_{\omega,t,x},F_J^1\in {L^1_{\omega,t,x}}\\\overline{\chi_J\phi\,\eta}=F_J^0+F_J^1}}\Big(\| F_J^0\|_{L^2_{\omega,t,x}}+r\|F_J^1\|_{L^1_{\omega,t,x}}\Big),\qquad r>0,
\end{aligned}
\end{align*}
where $\overline{\chi_J\phi\,\eta}=F^0_J+F^1_J$ with
$$F^0_J:=\int_\mr \chi^{(0)}_J\,\eta\,\dd \xi,\qquad F^1_J:=\sum_{K\gtrsim 1}\int_\mr \chi^{(K)}_J\,\eta\,\dd\xi.$$
By Lemma \ref{lemma:multiplier} and \eqref{eq:non-deg-local} we have
\begin{equation*}\begin{aligned}
\|F_J^{0}\|_{L^2_{t,x}}^2
=&\bigg\|\int_\mr\mathcal{F}_{tx}^{-1}\bigg[\psi_0\left(\frac{\mathcal{L}(iu,in,\xi)}{\delta}\right)\mathcal{F}_{tx}(\chi_J\phi)(u,n,\xi)\bigg]\eta(\xi)\dif\xi\bigg\|_{L^2_{t,x}}^2\\
&\lesssim \left(\frac{\delta}{J^\beta}\right)^{\alpha}\|\chi_{J}\eta\phi\|_{L^2_{t,x,\xi}}^2.
\end{aligned}\end{equation*}
Hence, we obtain
\begin{align*}
K(r, \overline{\chi_J\phi\,\eta})&\lesssim\delta^{\frac{\alpha}{2}}J^{-\frac{\beta\alpha}{2}}\|\chi_J\eta\phi\|_{L^2_{\omega,t,x,\xi}} \\
&\quad+r\del{T^\frac{1}{2}}\delta^{\frac{1}{2}(\alpha-1)}J^{-\frac{\beta\alpha}{2}}\left[\|\chi_{J}(0)\eta\|_{L^2_{\omega,x,\xi}}+\left(\sum_{k=1}^\infty\|{((\partial_\xi g_k)\chi\eta\phi)}_{J}\|_{L^2_{\omega,t,x,\xi}}^2\right)^{1/2}\right] \\
&\quad+ r\delta^{{-2}}J^{\beta+\epsilon}\|\phi\vartheta q_J(\eta+|\eta'|)\|_{L^1_{\omega}\mathcal{M}_{t,x,\xi}}\nonumber \\
&\quad+r\del{T^\frac{1}{2}}\delta^{{-\frac{3-\alpha}{2}}}J^{\frac{2-\alpha}{2}\beta}\left(\sum_{k=1}^\infty\|(g_k\chi\vartheta\phi)_{J}(\eta+|\eta'|)\|_{L^2_{\omega,t,x,\xi}}^2\right)^{1/2}\\
&\quad+ r\del{T^\frac{1}{2}}\delta^{-\frac{1}{2}}\left(\sum_{k=1}^\infty\|g_{k,J}(\cdot,0)\|_{L^2_x}^2\right)^{1/2} +r\delta^{-1}\|\partial_t \phi \chi_{J}\eta\|_{L^1_{\omega,t,x,\xi}}\nonumber
\end{align*}
and we intend to choose $r$ to equilibrate these bounds. To do so, let $\tau,\kappa >0$ to be chosen later and set 
\begin{align*}
  \delta = r^\tau J^\kappa
\end{align*}
which yields,
\begin{align*}
r^{-\tau\frac{\alpha}{2}} K(r, \overline{\chi_J\phi\,\eta})
&\lesssim  J^{\alpha\frac{\kappa-\beta}{2}}\|\chi_J\eta\phi\|_{L^2_{\omega,t,x,\xi}}\\
&\quad +r^{1-\frac{\tau}{2}}J^{\frac{-\kappa(1-\alpha)-\beta\alpha}{2}}\left[\|\chi_{J}(0)\eta\|_{L^2_{\omega,x,\xi}}+\left(\sum_{k=1}^\infty\|{((\partial_\xi g_k)\chi\eta\phi)}_{J}\|_{L^2_{\omega,t,x,\xi}}^2\right)^{1/2}\right]\\
&\quad+ r^{1-\tau (\frac{4+\alpha}{2})}J^{{-2}\kappa+\beta+\epsilon}\|\phi\vartheta q_J(\eta+|\eta'|)\|_{L^1_{\omega}\mathcal{M}_{t,x,\xi}}\nonumber 
\\
&\quad+r^{1-\frac{3}{2}\tau} J^{{-\frac{3-\alpha}{2}}\kappa+\frac{2-\alpha}{2}\beta}\left(\sum_{k=1}^\infty\|(g_k\chi\vartheta\phi)_{J}(\eta+|\eta'|)\|_{L^2_{\omega,t,x,\xi}}^2\right)^{1/2}\\
&\quad+ r^{1-\tau(\frac{1+\alpha}{2})}J^{-\frac{1}{2}\kappa}\left(\sum_{k=1}^\infty\|g_{k,J}(\cdot,0)\|_{L^2_x} ^2\right)^{1/2}
+r^{1-\tau (\frac{2+\alpha}{2})} J^{-\kappa}\|\partial_t \phi \chi_{J}\eta\|_{L^1_{\omega,t,x,\xi}}.\nonumber
\end{align*}
Optimizing in $\tau,\kappa$ yields
\begin{align*}
  \kappa = \frac{2}{3} \beta,\quad  \tau = \frac{2}{4+\alpha}.
\end{align*}
which obviously can be satisfied. Hence,  with $\theta := \frac{\alpha}{4-\alpha}$,
\begin{align*}
J^{\frac{\alpha\beta}{6}} r^{\theta} K(r, \overline{\chi_J\phi\,\eta})
&\lesssim  \|\chi_J\eta\phi\|_{L^2_{\omega,t,x,\xi}}\\
&\quad +r^{ \frac{3+\alpha}{4+\alpha}} J^{-\frac{1}{3}\beta}\left[\|\chi_{J}(0)\eta\|_{L^2_{\omega,x,\xi}}+\left(\sum_{k=1}^\infty\|{((\partial_\xi g_k)\chi\eta\phi)}_{J}\|_{L^2_{\omega,t,x,\xi}}^2\right)^{1/2}\right] \\
&\quad+ J^{-\frac{2-\alpha}{6}\beta+\epsilon}\|\phi\vartheta q_J(\eta+|\eta'|)\|_{L^1_{\omega}\mathcal{M}_{t,x,\xi}}\nonumber 
\\
&\quad+r^{\frac{1+\alpha}{4+\alpha}} \left(\sum_{k=1}^\infty\|(g_k\chi\vartheta \phi)_{J}(\eta+|\eta'|)\|_{L^2_{\omega,t,x,\xi}}^2\right)^{1/2}\\
&\quad+ r^{\frac{3}{4+\alpha}}J^{-\frac{2-\alpha}{6}\beta}\left(\sum_{k=1}^\infty\|g_{k,J}(\cdot,0)\|_{L^2_x} ^2\right)^{1/2}
+r^{\frac{2}{4+\alpha}} J^{-\frac{4-\alpha}{6}\beta}\|\partial_t \phi \chi_{J}\eta\|_{L^1_{\omega,t,x,\xi}}.\nonumber
\end{align*}
Finally, since for $r$ large we have the elementary estimate
\begin{align*}
K(r,\overline{\chi_J\phi\,\eta})&\leq \|\overline{\chi_J\phi\,\eta}\|_{L^2_{\omega,t,x}},
\end{align*}
we apply the K-method of real interpolation to deduce
\begin{align*}
&J^{\frac{\alpha\beta}{6}}\|\overline{\chi_J\phi\,\eta}\|_{(L^2_{\omega,t,x},L^1_{\omega,t,x})_{\theta,\infty}}\\
&\leq J^{\frac{\alpha\beta}{6}}\|r^{-\theta}K(r,\overline{\chi_J\phi\,\eta})\|_{L^\infty_r}\\
&\leq J^{\frac{\alpha\beta}{6}}\|r^{-\theta}K(r,\overline{\chi_J\phi\,\eta})\ind_{r\leq R}\|_{L^\infty_r}+ J^{\frac{\alpha\beta}{6}}\|r^{-\theta}K(r,\overline{\chi_J\phi\,\eta})\ind_{r\geq R}\|_{L^\infty_r}\\
&\lesssim 
\|\chi_J\phi\eta\|_{L^2_{\omega,t,x,\xi}} 
+R^{ \frac{3+\alpha}{4+\alpha}} J^{-\frac{1}{3}\beta}\left[\|\chi_{J}(0)\eta\|_{L^2_{\omega,x,\xi}}+\left(\sum_{k=1}^\infty\|{((\partial_\xi g_k)\chi\eta\phi)}_{J}\|_{L^2_{\omega,t,x,\xi}}^2\right)^{1/2}\right]\\
&\quad+ J^{-\frac{2-\alpha}{6}\beta+\epsilon}\|\phi\vartheta q_J(\eta+|\eta'|)\|_{L^1_{\omega}\mathcal{M}_{t,x,\xi}}\nonumber 
+R^{\frac{1+\alpha}{4+\alpha}}\left( \sum_{k=1}^\infty\|(g_k\chi\vartheta \phi)_{J}(\eta+|\eta'|)\|_{L^2_{\omega,t,x,\xi}}^2\right)^{1/2}\\
&\quad+ R^{\frac{3}{4+\alpha}}J^{-\frac{2-\alpha}{6}\beta}\left(\sum_{k=1}^\infty\|g_{k,J}(\cdot,0)\|_{L^2_x} ^2\right)^{1/2}
+R^{\frac{2}{4+\alpha}} J^{-\frac{4-\alpha}{6}\beta}\|\partial_t \phi \chi_{J}\eta\|_{L^1_{\omega,t,x,\xi}}\nonumber
\\
&\quad+J^{\frac{\alpha\beta}{6}}R^{-\theta}\|\overline{\chi_J\phi\,\eta}\|_{L^2_{\omega,t,x}}.
\end{align*}
Let us take $R=J^{\tau}$ for some $\tau>0$ to be chosen below. Then
\begin{equation}\label{eq:before_tau}\begin{aligned}
&J^{\frac{\alpha\beta}{6}}\|\overline{\chi_J\phi\,\eta}\|_{(L^2_{\omega,t,x},L^1_{\omega,t,x})_{\theta,\infty}}\\
&\lesssim 
\|\chi_J\phi\eta\|_{L^2_{\omega,t,x,\xi}} +(J^{\tau})^{ \frac{3+\alpha}{4+\alpha}} J^{-\frac{1}{3}\beta}\left[\|\chi_{J}(0)\eta\|_{L^2_{\omega,x,\xi}}+\left(\sum_{k=1}^\infty\|{((\partial_\xi g_k)\chi\eta\phi)}_{J}\|_{L^2_{\omega,t,x,\xi}}^2\right)^{1/2}\right]\\
&\quad+ J^{-\frac{2-\alpha}{6}\beta+\epsilon}\|\phi\vartheta q_J(\eta+|\eta'|)\|_{L^1_{\omega}\mathcal{M}_{t,x,\xi}}
+(J^{\tau})^{\frac{1+\alpha}{4+\alpha}}\left( \sum_{k=1}^\infty\|(g_k\chi\vartheta \phi)_{J}(\eta+|\eta'|)\|_{L^2_{\omega,t,x,\xi}}^2\right)^{1/2}\\
&\quad+ (J^{\tau})^{\frac{3}{4+\alpha}}J^{-\frac{2-\alpha}{6}\beta}\left(\sum_{k=1}^\infty\|g_{k,J}(\cdot,0)\|_{L^2_x} ^2\right)^{1/2}
+(J^{\tau})^{\frac{2}{4+\alpha}} J^{-\frac{4-\alpha}{6}\beta}\|\partial_t \phi \chi_{J}\eta\|_{L^1_{\omega,t,x,\xi}}
\\
&\quad+J^{\frac{\alpha\beta}{6}}(J^{\tau})^{-\theta}\|\overline{\chi_J\phi\,\eta}\|_{L^2_{\omega,t,x}}.
\end{aligned}\end{equation}
Now we aim to choose $\tau$ in order to minimize the maximum of the exponents of $J$ occurring on the right hand side in the previous inequality.
Optimizing for $\tau$ the terms
\begin{align*}
\frac{3+\alpha}{4+\alpha} \tau -\frac{\beta}{3},\ 
\frac{1+\alpha}{4+\alpha} \tau,\ 
\frac{3}{4+\alpha}\tau-\frac{2-\alpha}{6}\beta,\ 
\frac{2}{4+\alpha}\tau - \frac{4-\alpha}{6}\beta,\ 
-\frac{\alpha}{4+\alpha}\tau+\frac{\alpha\beta}{6}.
\end{align*}
yields the choice 
  $$\tau = \frac{\beta\alpha(4+\alpha)}{6(1+2\alpha)}.$$
With this choice, all the exponents in the right hand side of \eqref{eq:before_tau} are of order less than $\frac{\alpha \beta }{6}$. Multiplication with $J^\frac{-\alpha \beta }{6}$ thus leads to negative powers of $J$ on the right hand side, the worst (i.e.\ the maximal) one being 
  $$ -\frac{\alpha^2\beta}{6(1+2\alpha)}.$$
Therefore, since $\epsilon$ was chosen small, we obtain that
\begin{align*}
\begin{aligned}
\|\overline{\chi_J\phi\,\eta}\|_{(L^2_{\omega,t,x},L^1_{\omega,t,x})_{\theta,\infty}}&\lesssim 
J^{ -\frac{\alpha^2\beta}{6(1+2\alpha)}}\bigg[\|\chi_J\phi\eta\|_{L^2_{\omega,t,x,\xi}}
 +  \del{T^\frac{1}{2}} \|\chi_{J}(0)\eta\|_{L^2_{\omega,x,\xi}} \\
 &\quad\quad+  \del{T^\frac{1}{2}}\left(\sum_{k=1}^\infty\|{((\partial_\xi g_k)\chi\eta\phi)}_{J}\|_{L^2_{\omega,t,x,\xi}}^2\right)^{1/2} \\
&\quad\quad  + \del{T^\frac{1}{2}}\left(\sum_{k=1}^\infty \|(g_k\chi\vartheta\phi)_{J}(\eta+|\eta'|)\|_{L^2_{\omega,t,x,\xi}}^2\right)^{1/2}+ \del{T^\frac{1}{2}}\left(\sum_{k=1}^\infty \|g_{k,J}(\cdot,0)\|_{L^2_x}^2\right)^{1/2}\\
&\quad\quad
 + \|\phi q_J\vartheta (\eta+|\eta'|)\|_{L^1_{\omega}\mathcal{M}_{t,x,\xi}}+ \|\partial_t \phi \chi_{J}\eta\|_{L^1_{\omega,t,x,\xi}} + \|\overline{\chi_J\phi\,\eta}\|_{L^2_{\omega,t,x}}\bigg].
 \end{aligned}
\end{align*}
Note that although all the above norms are global in time, i.e. $t\in(-\infty,\infty)$, the integrands are localized on $[0,T]$ due to the cut-off $\phi=\phi^\lambda$. 

\subsection*{Conclusion}
The real interpolation of two Lebesgue spaces is given by a Lorentz space (see \cite[Subsection 1.18.6, Theorem 1]{triebel}, namely,
$$
(L^2_{\omega,t,x},L^1_{\omega,t,x})_{\theta,\infty}=L^{\tilde r,\infty}_{\omega,t,x},\qquad \frac{1}{\tilde r}=\frac{1-\theta}{2}+\frac{\theta}{1}.
$$
In the case of a bounded domain, $L_{\omega,t,x}^{\tilde r,\infty}$ is embedded in the Lebesgue space $L^r_{\omega,t,x}$ whenever $\tilde r> r$, see \cite[Exercise 1.1.11]{grafakos}
\begin{equation*}
(L^2_{\omega,t,x},L^1_{\omega,t,x})_{\theta,\infty}\hookrightarrow L^r_{\omega,t,x},\qquad \frac{1}{r}>\frac{1-\theta}{2}+\frac{\theta}{1}.
\end{equation*}
Thus letting
$$s<\frac{\alpha^2\beta}{6(1+2\alpha)}$$ we deduce
\begin{equation}\label{eq:45}
\begin{aligned}
&\|\overline{\chi\phi\eta}\|_{L^{r}(\Omega\times[0,T]; W^{s,r}(\T))}
\\&
\lesssim 
\|\chi\phi\eta\|_{L^2_{\omega,t,x,\xi}}
 +  \del{T^\frac{1}{2}} \|\chi(0)\eta\|_{L^2_{\omega,x,\xi}} +  \del{T^\frac{1}{2}}\left(\sum_{k=1}^\infty\|((\partial_\xi g_k)\chi\eta\phi)\|_{L^2_{\omega,t,x,\xi}}^2\right)^{1/2} \\
&\quad\quad  + \del{T^\frac{1}{2}}\left(\sum_{k=1}^\infty \|g_k\chi\vartheta \phi(\eta+|\eta'|)\|_{L^2_{\omega,t,x,\xi}}^2\right)^{1/2}+ \del{T^\frac{1}{2}}\left(\sum_{k=1}^\infty \|g_{k}(\cdot,0)\|_{L^2_x}^2\right)^{1/2}\\
&\quad\quad
 + \|\phi\vartheta q(\eta+|\eta'|)\|_{L^1_{\omega}\mathcal{M}_{t,x,\xi}}+ \|\partial_t \phi \chi\eta\|_{L^1_{\omega,t,x,\xi}} + \|\overline{\chi\phi\,\eta}\|_{L^2_{\omega,t,x}}.
 \end{aligned}
\end{equation}
Since the constant in inequality \eqref{eq:45} is independent of $\lambda$, we may send $\lambda\to 0$. First, we observe that
$$\limsup_{\lambda\to0}\|\partial_t \phi^\lambda \chi\eta\|_{L^1_{\omega,t,x,\xi}}\lesssim\sup_{0\leq t\leq T}\|\chi\eta\|_{L^1_{\omega,x,\xi}}.$$
Using the dominated convergence theorem for the remaining terms,
we obtain the following estimate for  $u^\eta=\overline{\chi\,\eta}$
\begin{align*}
\|u^\eta\|_{L^{r}(\Omega\times[0,T]; W^{s,r}(\T))}
&\lesssim 
\|\chi\eta\|_{L^2_{\omega,t,x,\xi}}
 +   \|\chi(0)\eta\|_{L^2_{\omega,x,\xi}} + \left(\sum_{k=1}^\infty\|(\partial_\xi g_k) \chi\eta\|_{L^2_{\omega,t,x,\xi}}^2\right)^{1/2} \\
&\quad  + \left( \sum_{k=1}^\infty\|g_k\chi\vartheta(\eta+|\eta'|)\|_{L^2_{\omega,t,x,\xi}}^2\right)^{1/2}+ \del{T^\frac{1}{2}}\left(\sum_{k=1}^\infty \|g_{k}(\cdot,0)\|_{L^2_x}^2\right)^{1/2}\\
&\quad
 + \| q\vartheta(\eta+|\eta'|)\|_{L^1_{\omega}\mathcal{M}_{t,x,\xi}} + \sup_{0\leq t\leq T}\|\chi\eta\|_{L^1_{\omega,x,\xi}}+ \|\overline{\chi\eta}\|_{L^2_{\omega,t,x}}.
\end{align*}
It remains to estimate the right hand side in terms of the available bounds for the kinetic solution $u$ and the corresponding kinetic measure $m$. Note that this estimate will depend on the localization $\eta$. First, due to the definition of the equilibrium function $\chi$, it follows immediately that
\[
\|\chi\eta\|_{L^2_{\omega,t,x,\xi}}\lesssim_\eta \|u^\eta\|_{L^1_{\omega,t,x}}^{1/2},\qquad \|\chi(0)\eta\|_{L^2_{\omega,x,\xi}}\lesssim_\eta\|u_0^\eta\|_{L^1_{\omega,x}}^{1/2},
\]
\[
\sup_{0\leq t\leq T}\|\chi\eta\|_{L^1_{\omega,x,\xi}}=\sup_{0\leq t\leq T}\|u^\eta\|_{L^1_{\omega,x}},\qquad\|\overline{\chi\eta}\|_{L^2_{\omega,t,x}}\lesssim_\eta\|u^\eta\|_{L^1_{\omega,t,x}}^{1/2}.
\]
Second, similarly and due to \eqref{linrust} we have 
with $\varTheta_{\eta}$ such that $\varTheta_{\eta}' = (|\xi|^2+1)\vartheta^2(\xi)(\eta(\xi)+|\eta'|(\xi))$,
\begin{align*}
\left(\sum_{k=1}^\infty\|(\partial_\xi g_k) \chi\eta\|_{L^2_{\omega,t,x,\xi}}^2\right)^{1/2}+\left(\sum_{k=1}^\infty\|g_k\chi\vartheta(\eta+|\eta'|)\|_{L^2_{\omega,t,x,\xi}}^2\right)^{1/2}+&\left(\sum_{k=1}^\infty \|g_{k}(\cdot,0)\|_{L^2_x}^2\right)^{1/2}\\
&\quad\lesssim_\eta \|\varTheta_{\eta}(u) \|_{L^1_{\omega,t,x}}^{1/2}+1.
\end{align*}
Finally, since $q=m-\frac{1}{2}G^2\delta_{u=\xi}$, we deduce
\[
\|q\vartheta(\eta+|\eta'|)\|_{L^1_{\omega}\mathcal{M}_{t,x,\xi}}\lesssim_\eta\|m\vartheta(\eta+|\eta'|)\|_{L^1_\omega\mathcal{M}_{t,x,\xi}}+ 1.
\]
At this point it is worth noticing that all estimates were uniform in $\eta$ up to constants depending on $\|\eta\|_{C^1}$ and \eqref{eq:non-deg-local}.
Since $\varTheta_{\eta}(u) \ge u^\eta$ we conclude
   \begin{equation*}\begin{aligned}
   \|\bar\eta(u)\|_{L^{r}(\Omega\times[0,T]; W^{s,r}(\T))}
   \lesssim_\eta& 
     \|\bar\eta(|u_{0}|)\|^{1/2}_{L^1_{\omega,x}} + \|\varTheta_{\eta}(|u|) \|_{L^1_{\omega,t,x}}^{1/2}+\sup_{0\leq t\leq T}\|\bar\eta(|u|)\|_{L^1_{\omega,x}}\\
     &+ \| m\vartheta(\eta+|\eta'|)\|_{L^1_{\omega}\mathcal{M}_{t,x,\xi}} + 1.
   \end{aligned}\end{equation*}
with
\begin{align*} 
   \frac{1}{r}>\frac{1-\theta}{2}+\frac{\theta}{1} \quad \text{and}\quad
   s<\frac{\alpha^2\beta}{6(1+2\alpha)}
\end{align*}
which completes the proof.
\end{proof}

\section{Well-posedness}
\label{sec:L1}

In this section we present the proof of the main well-posedness result Theorem \ref{thm:main}. The uniqueness part of Theorem \ref{thm:main} will be proved in Theorem \ref{thm:uniqueness} below, the existence in Theorem \ref{thm:ex} below.

\subsection{Uniqueness}
\label{subsec:uniq}

In this section we prove a comparison results and thus uniqueness for kinetic solutions to \eqref{eq}. We emphasize that we do not assume any higher $L^p$ estimates for the kinetic solutions, thus providing a proof of uniqueness in the general $L^1$ setting. \blue{Recall that in this case, the kinetic measure is not finite and we only have a weak control of its decay for large $\xi$, cf.\ Definition \ref{meesL1}~(ii). Moreover, the corresponding chain rule can only be formulated in a weaker sense.} In addition, we only assume that $\sigma$ is locally H\"older continuous and $b'$ is locally bounded. In particular, no polynomial growth condition for $b'$ is required. This generalizes previous related uniqueness results given in \cite[Section 3]{dehovo}. The additional difficulties are resolved here by introducing an additional cutoff argument.

Analogously to \cite[Proposition 3.1]{dehovo}, \cite[Proposition 3.1]{degen} and \cite[Proposition 10]{DV} one may prove the existence of left and right continuous representatives for kinetic solutions.
\begin{prop}\label{prop:left-right}
  Let $u$ be a kinetic solution to \eqref{eq}. Then, $f=1_{u>\xi}$ admits representatives $f^-$ and $f^+$ which are almost surely left- and right-continuous, respectively, at all points $t^* \in [0,T]$ in the sense of distributions over $\mt \times \R$. More precisely, for all $t^* \in [0, T ]$ there exist kinetic functions $f^{*,\pm}$ on $\Omega\times\mt\times\R$ such that setting $f^\pm(t^*)=f^{*,\pm}$ yields $f^\pm = f$ almost everywhere and 
    $$ \bl f^\pm(t^*\pm\varepsilon),\psi \br \to \bl f^\pm(t^*),\psi\br\quad\varepsilon\downarrow 0,\ \forall \psi \in C^2_c(\mt\times\R),\ \p\text{-a.s.},  $$
  where the zero set does not depend on $\psi$ nor $t^*$.
  Moreover, there is a countable set $Q\subseteq [0,T]$ such that $\p$-a.s.\ for all $t^*\in[0,T]\setminus Q$ we have $f^+(t^*) = f^-(t^*)$.
\end{prop}

Regarding the doubling of the variables \cite[Proposition 3.2]{dehovo} we make use of the following version, Proposition \ref{prop:doubling} below, which is more suitable for the $L^1$-setting.

Let $(\varrho_\varepsilon),\,(\psi_\delta)$ be standard Dirac sequences on $\mt$ and $\mr$, respectively. That is, let $\psi\in C^\infty_c(\mr)$ be symmetric nonnegative function such that
$\int_\mr\psi=1, \supp\psi\subset [-1,1]$
and set
$$\psi_\delta(\xi)=\frac{1}{\delta}\,\psi\Big(\frac{\xi}{\delta}\Big).$$
To define $(\varrho_\varepsilon)$, consider a nonnegative symmetric function $\tilde\varrho\in C^\infty_c(\R^N)$ satisfying
$
\int_{\R^N}\tilde\varrho =1$, $\supp\varrho\subset B(0,\pi)$ and let $\varrho$ denote its $2\pi\mathbb{Z}^N$-periodic modification. Then let
$$\varrho_\varepsilon(x)=\frac{1}{\varepsilon^N}\,\varrho\Big(\frac{x}{\varepsilon}\Big).$$
Let us now define a sequence of smooth cut-off functions $(K_\ell)$ as follows: Let $K\in C^\infty(\mr)$ be such that $0\leq K(\xi)\leq 1$, $K\equiv 1$ if $|\xi|\leq 1$, $K\equiv 0$ if $|\xi|\geq 2$, and $|K'(\xi)|\leq 1.$ Define
$$K_\ell(\xi):=K\Big(\frac{\xi}{2^\ell}\Big),\qquad \ell\in \mn.$$
Then, clearly,
$$|K'_\ell(\xi)|\leq \frac{1}{2^\ell}\ind_{2^\ell\leq |\xi|\leq 2^{\ell+1}}.$$

\begin{prop}[Doubling of variables]\label{prop:doubling}
Let $u_1,u_2$ be kinetic solutions to \eqref{eq}. Denote $f_1=\ind_{u_1>\xi},$ $f_2=\ind_{u_2>\xi}$ with the corresponding Young measures $\nu^1=\delta_{u^1},$ $\nu^2=\delta_{u^2}$, respectively. Then for all $t\in[0,T]$ we have
\begin{equation*}\label{doubling}
\begin{split}
\stred&\int_{(\mt )^2}\int_{\mr^3}\varrho_\varepsilon(x-y)K_\ell(\eta)\psi_\delta(\eta-\xi)\psi_\delta(\eta-\zeta)f_1^{\pm}(x,t,\xi)\bar{f}_2^{\pm}(y,t,\zeta)\,\dif\xi\,\dif\zeta\,\dif x\,\dif y\,\dif\eta\\
&\leq\stred\int_{(\mt )^2}\int_{\mr^3}\varrho_\varepsilon(x-y)K_\ell(\eta)\psi_\delta(\eta-\xi)\psi_\delta(\eta-\zeta)f_{1,0}(x,\xi)\bar{f}_{2,0}(y,\zeta)\,\dif\xi\,\dif\zeta\,\dif x\,\dif y\,\dif \eta\\
&\qquad+\mathrm{I}+\mathrm{J}+\mathrm{K}+\mathrm{L}(\delta, \ell),
\end{split}
\end{equation*}
where
$$\mathrm{I}=\,\stred\int_0^t\int_{(\mt )^2}\int_{\mr^3}f_1\bar{f}
_2\big(b(\xi)-b(\zeta)\big)\cdotp\nabla_x\varrho_\varepsilon(x-y)K_\ell(\eta)\psi_\delta(\eta-\xi)\psi_\delta(\eta-\zeta)\,\dif\xi\,
\dif\zeta\,\dif x\, \dif y\,\dif\eta\,\dif s,$$
\begin{equation*}
\begin{split}
\mathrm{J}&=\stred\int_0^t\int_{(\mt )^2}\int_{\mr^3}f_1\bar{f}
_2\big(A(\xi)+A(\zeta)\big):\totdif^2_x\varrho_\varepsilon(x-y)K_\ell(\eta)\psi_\delta(\eta-\xi)\psi_\delta(\eta-\zeta)\,
\dif\xi\,\dif\zeta\,\dif x\, \dif y\,\dif\eta\,\dif s\\
&\qquad\blue{-\stred\int_0^t\int_{(\mt )^2}\int_{\mr^3}\varrho_\varepsilon(x-y)K_\ell(\eta)\psi_\delta(\eta-\xi)\,\dif\nu^{1}_{x,s}
(\xi)\,\dif x\,\dif n^{{\psi_\delta(\eta-\cdot)}}_{2}(y,s)\,\dif\eta}\\
&\qquad\blue{-\stred\int_0^t\int_{(\mt )^2}\int_{\mr^3}\varrho_\varepsilon(x-y)K_\ell(\eta)\psi_\delta(\eta-\zeta)\,
\dif\nu^{2}_{y,s}
(\zeta)\,\dif y\,\dif n^{{\psi_\delta(\eta-\cdot)}}_{1}(x,s)\,\dif\eta,}
\end{split}
\end{equation*}
$$\mathrm{K}=\frac{1}{2}\stred\int_0^t\!\!\int_{(\mt )^2}\!\!\int_{\mr^3}
\!\!\varrho_\varepsilon(x-y)K_\ell(\eta)\psi_\delta(\eta-\xi)\psi_\delta(\eta-\zeta)\!\sum_{k\geq1}\!\big|g_k(x,\xi)-g_k(y,
\zeta)\big|^2\dif\nu^1_{x,s}(\xi)\dif\nu^2_{y,s}(\zeta)\dif x\,\dif y\,\dif\eta\,\dif s,$$
$$\lim_{\ell\to\infty}\lim_{\delta\to0}\mathrm{L}(\delta, \ell)=0.$$

\begin{proof}
A similar approach as in \cite[Proposition 3.2]{degen} and \cite[Proposition 9]{debus} yields for
$$\alpha(x,\xi,y,\zeta,\eta):=\varrho_\varepsilon(x-y)K_\ell(\eta)\psi_\delta(\eta-\xi)\psi_\delta(\eta-\zeta)$$
that
\begin{equation*}
\begin{split}
\stred\int_{(\mt)^2}\int_{\mr^3} f_1^+(t)&\bar{f}_2^+(t)\alpha\,\dif\xi\,\dif\zeta\,\dif x\, \dif y\,\dif \eta=\stred\int_{(\mt)^2}\int_{\mr^3}  f_{1,0}\bar{f}_{2,0}\alpha\,\dif\xi\,\dif\zeta\,\dif x\, \dif y\,\dif \eta\\
&+\stred\int_0^t\int_{(\mt )^2}\int_{\mr^3}f_1\bar{f}_2\big(b(\xi)-b(\zeta)\big)\cdotp\nabla_x\alpha\,\dif\xi\,\dif\zeta\,\dif x\, \dif y\,\dif \eta\,\dif s\\
&+\stred\int_0^t\int_{(\mt )^2}\int_{\mr^3}f_1\bar{f}_2\,A(\zeta):\totdif^2_y\alpha\,\dif\xi\,\dif\zeta\,\dif x\, \dif y\,\dif \eta\,\dif s\\
&+\stred\int_0^t\int_{(\mt )^2}\int_{\mr^3}f_1\bar{f}_2\,A(\xi):\totdif^2_x\alpha\,\dif\xi\,\dif\zeta\,\dif x\, \dif y\,\dif \eta\,\dif s\\
&+\frac{1}{2}\stred\int_0^t\int_{(\mt )^3}\int_{\mr^2}\bar{f}_2\partial_\xi\alpha\, G^2_1\,\dif\nu_{x,s}^1(\xi)\,\dif\zeta\,\dif y\,\dif x\,\dif \eta\,\dif s\\
&-\frac{1}{2}\stred\int_0^t\int_{(\mt )^3}\int_{\mr^2}f_1\partial_\zeta\alpha\, G^2_2\,\dif\nu_{y,s}^2(\zeta)\,\dif\xi\,\dif y\,\dif x\,\dif \eta\,\dif s\\
&-\stred\int_0^t\int_{(\mt )^2}\int_{\mr^3}G_{1,2}\alpha\,\dif\nu^1_{x,s}(\xi)\,\dif\nu^2_{y,s}(\zeta)\,\dif x\,\dif y\,\dif \eta\,\dif s\\
&-\stred\int_0^t\int_{(\mt )^2}\int_{\mr^3}\bar{f}_2^-\partial_\xi\alpha\,\dif m_1(x,s,\xi)\,\dif\zeta\,\dif \eta\,\dif y\\
&+\stred\int_0^t\int_{(\mt )^2}\int_{\mr^3}f_1^+\partial_\zeta\alpha\,\dif m_2(y,s,\zeta)\,\dif\xi\,\dif x\,\dif \eta=:I_1+\cdots+ I_9.
\end{split}
\end{equation*}
We need to treat the last five terms $I_5,\dots,I_9$. Using the fact that
\begin{align*}
&\int_\mr\psi_\delta(\eta-\xi)\partial_\zeta\psi_\delta(\eta-\zeta)K_\ell(\eta)\,\dif \eta\\
&\qquad=-\int_\mr\partial_\xi\psi_\delta(\eta-\xi)\psi_\delta(\eta-\zeta)\,\dif \eta+\int_\mr\psi_\delta(\eta-\xi)\psi_\delta(\eta-\zeta)K'_\ell(\eta)\,\dif \eta
\end{align*}
\blue{together with the decomposition of the kinetic measure $m_2$ from Definition \ref{kinsolL1}~(iii),}
it follows
\begin{equation*}
\begin{split}
I_9&
\blue{\leq-\stred\int_0^t\int_{(\mt )^2}\int_{\mr^3}\varrho_\varepsilon(x-y)K_\ell(\eta)\psi_\delta(\eta-\xi)\,\dif\nu^{1}_{x,s}
(\xi)\,\dif x\,\dif n^{{\psi_\delta(\eta-\cdot)}}_{2}(y,s)\,\dif\eta}\\
&\qquad+\frac{1}{2^\ell}\stred\int_0^t\int_{\mt }\int_{\mr}\ind_{2^\ell-\delta\leq|\zeta|\leq 2^{\ell+1}+\delta}\,\dif m_2(y,s,\zeta),
\end{split}
\end{equation*}
where, according to Definition \ref{meesL1} (ii), the second term on the right hand side vanishes if we let $\delta\to0$ and then $\ell\to\infty$. By symmetry
\begin{equation*}
\begin{split}
I_8&\leq\blue{-\stred\int_0^t\int_{(\mt )^2}\int_{\mr^3}\varrho_\varepsilon(x-y)K_\ell(\eta)\psi_\delta(\eta-\zeta)\,
\dif\nu^{2}_{y,s}
(\zeta)\,\dif y\,\dif n^{{\psi_\delta(\eta-\cdot)}}_{1}(x,s)\,\dif\eta}\\
&\qquad+\frac{1}{2^\ell}\stred\int_0^t\int_{\mt }\int_{\mr}\ind_{2^\ell-\delta\leq|\xi|\leq 2^{\ell+1}+\delta}\,\dif m_1(x,s,\xi).
\end{split}
\end{equation*}
Next, we have
\begin{align*}
&I_5+I_6+I_7=K\\
&\quad+\frac{1}{2}\E\int_0^t\int_{(\mt)^2}\int_{\mr^3}\bar{f}_2\varrho_\varepsilon(x-y)\psi_\delta(\eta-\xi)\psi_\delta(\eta-\zeta)K'_\ell(\eta)G^2_1\,\dif\nu^1_{s,x}(\xi)\,\dif x\,\dif\zeta\,\dif y\,\dif\eta\,\dif s\\
&\quad-\frac{1}{2}\stred\int_0^t\int_{(\mt )^2}\int_{\mr^3}f_1\,\varrho_\varepsilon(x-y)\psi_\delta(\eta-\xi)\psi_\delta(\eta-\zeta)K'_\ell(\eta) G^2_2\,\dif\nu_{y,s}^2(\zeta)\,\dif\xi\,\dif y\,\dif x\,\dif \eta\,\dif s\\
&=:K+I_{51}+I_{61}
\end{align*}
and due to \eqref{linrust}
\begin{align*}
I_{51}&\leq \frac{D}{2^{\ell}}\E\int_0^t\int_{(\mt)^2}\int_{\mr^3}\varrho_\varepsilon(x-y)\psi_\delta(\eta-\xi)\psi_\delta(\eta-\zeta)\ind_{2^\ell-\delta\leq|\xi|\leq 2^{\ell+1}+\delta}\,\dif\nu^1_{s,x}(\xi)\,\dif x\,\dif\zeta\,\dif y\,\dif\eta\,\dif s\\
&\quad+\frac{D}{2^{\ell}}\E\int_0^t\int_{(\mt)^2}\int_{\mr^3}\varrho_\varepsilon(x-y)\psi_\delta(\eta-\xi)\psi_\delta(\eta-\zeta)\ind_{2^\ell-\delta\leq|\xi|\leq 2^{\ell+1}+\delta}|\xi|^2\,\dif\nu^1_{s,x}(\xi)\,\dif x\,\dif\zeta\,\dif y\,\dif\eta\,\dif s\\
&\leq \frac{D}{2^{\ell}}\E\int_0^t\int_{\mt}\int_{\mr}\ind_{2^\ell-\delta\leq|\xi|\leq 2^{\ell+1}+\delta}\,\dif\nu^1_{s,x}(\xi)\,\dif x\,\dif s\\
&\quad+\frac{D}{2^{\ell}}\E\int_0^t\int_{\mt}\int_{\mr}\ind_{2^\ell-\delta\leq|\xi|\leq 2^{\ell+1}+\delta}|\xi|^2\,\dif\nu^1_{s,x}(\xi)\,\dif x\,\dif s.
\end{align*}
We further note that
\begin{align*}
   &\frac{D}{2^{\ell}}\E\int_0^t\int_{\mt}\int_{\mr}\ind_{2^\ell-\delta\leq|\xi|\leq 2^{\ell+1}+\delta}|\xi|^2\,\dif\nu^1_{s,x}(\xi)\,\dif x\,\dif s \\
   &\le D\frac{2^{\ell+1}+\delta}{2^{\ell}}\E\int_0^t\int_{\mt}\int_{\mr}\ind_{2^\ell-\delta\leq|\xi|\leq 2^{\ell+1}+\delta}|\xi|\,\dif\nu^1_{s,x}(\xi)\,\dif x\,\dif s \\
   &\le D\frac{2^{\ell+1}+\delta}{2^{\ell}}\E\int_0^t\int_{\mt}\ind_{2^\ell-\delta\leq|u^1(s,x)|}|u^1(s,x)|\,\dif x\,\dif s.
\end{align*}
Hence,
\begin{align*}
  \limsup_{\delta\to 0} I_{51}
  &\leq \frac{D}{2^{\ell}}t + D\E\int_0^t\int_{\mt}\ind_{2^\ell\leq|u^1(s,x)|}|u^1(s,x)|\,\dif x\,\dif s.
\end{align*}
By dominated convergence this implies that
\begin{align*}
  \lim_{\ell\to\infty}\lim_{\delta\to 0} I_{51}=0
\end{align*}
and $I_{61}$ may be treated analogously, which completes the proof.
\end{proof}
\end{prop}

Finally, we have all in hand to prove the comparison principle leading to the proof of uniqueness as well as continuous dependence on the  initial condition.

\begin{thm}[Comparison principle]\label{thm:uniqueness}
Let $u$ be a kinetic solution to \eqref{eq}. Then there exist $u^+$ and $u^-$, representatives of $u$, such that, for all $t\in[0,T]$, $f^\pm(t,x,\xi)=\ind_{u^{\pm}(t,x)>\xi}$ for a.e. $(\omega,x,\xi)$, where $f^\pm$ are as in Proposition \ref{prop:left-right}.

Moreover, if $\,u_1,u_2$ are kinetic solutions to \eqref{eq} with initial data $u_{1,0}, u_{2,0}$, respectively, then,
\begin{equation}\label{comparison}
 \sup_{t\in[0,T]}\stred\|(u_1^{\pm}(t)-u_2^{\pm}(t))^+\|_{L^1(\mt )}\leq \stred\|(u_{1,0}-u_{2,0})^+\|_{L^1(\mt )}.
\end{equation}

\begin{proof}
Using Proposition \ref{prop:doubling} we have
\begin{equation*}
\begin{split}
\stred&\int_{\mt }\int_\mr f_1^{\pm}(x,t,\xi)\bar{f}_2^{\pm}(x,t,\xi)\,\dif\xi\,\dif x\\
&=\stred\int_{(\mt )^2}\int_{\mr^3}\varrho_\varepsilon(x-y)K_\ell(\eta)\psi_\delta(\eta-\xi)\psi_\delta(\eta-\zeta)f_1^{\pm}(x,t,\xi)\bar{f}_2^{\pm}(y,t,\zeta)\,\dif\xi\,\dif \zeta\,\dif x \,\dif y+\eta_t(\varepsilon,\delta,\ell)\\
&\leq\stred\int_{(\mt )^2}\int_{\mr^3}\varrho_\varepsilon(x-y)K_\ell(\eta)\psi_\delta(\eta-\xi)\psi_\delta(\eta-\zeta)f_{1,0}(x,\xi)\bar{f}_{2,0}(y,\zeta)\,\dif\xi\,\dif\zeta\,\dif x\,\dif y\,\dif \eta\\
&\qquad+\mathrm{I}+\mathrm{J}+\mathrm{K}+\mathrm{L}(\delta, \ell)+\eta_t(\varepsilon,\delta,\ell),
\end{split}
\end{equation*}
with $\mathrm{I}$, $\mathrm{J}$, $\mathrm{K}$ as in Proposition \ref{prop:doubling} and
   $$\lim_{\ell\to\infty}\lim_{\varepsilon,\delta\rightarrow 0}\eta_t(\varepsilon,\delta,\ell)=0.$$
We aim to find suitable bounds for the terms $\mathrm{I},\,\mathrm{J},\,\mathrm{K}$.

Since $b'$ is locally bounded, setting $\|b'\|_{\infty,\delta,\ell}:= \|b'\|_{L^\infty(-2^{\ell+1}-\delta,2^{\ell+1}+\delta)}$, we have
\begin{align*}
|\mathrm{I}|
&\leq \|b'\|_{\infty,\delta,\ell}
\stred\int_0^t\int_{(\mt )^2}\int_{\mr^3}f_1\bar{f}_2|\xi-\zeta|\psi_\delta(\eta-\zeta)\psi_\delta(\eta-\xi)K_\ell(\eta)\,\dif\xi\,\dif \zeta\,\big|\nabla_x\varrho_\varepsilon(x-y)\big|\dif x\,\dif y\,\dif\eta\,\dif s\\
&\leq 2\delta\|b'\|_{\infty,\delta,\ell}
\stred\int_0^t\int_{(\mt )^2}\int_{\mr^3}\psi_\delta(\eta-\zeta)\psi_\delta(\eta-\xi)K_\ell(\eta)\,\dif\xi\,\dif \zeta\,\big|\nabla_x\varrho_\varepsilon(x-y)\big|\dif x\,\dif y\,\dif\eta\,\dif s\\
&\leq 8 \,2^\ell|\mt|^2\varepsilon^{-1}\delta\|b'\|_{\infty,\delta,\ell}t.
\end{align*}

In order to estimate the term $\mathrm{J}$, we observe that
\begin{equation*}
\begin{split}
\mathrm{J}&=\stred\int_0^t\int_{(\mt )^2}\int_{\mr^3}f_1\bar{f}
_2\,\big(\sigma(\xi)-\sigma(\zeta)\big)^2:\totdif^2_x\varrho_\varepsilon(x-y)K_\ell(\eta)\psi_\delta(\eta-\xi)\psi_\delta(\eta-\zeta)\,
\dif\xi\,\dif\zeta\,\dif x\, \dif y\,\dif \eta\,\dif s\\
&\quad+2\,\stred\int_0^t\int_{(\mt )^2}\int_{\mr^3}f_1\bar{f}
_2\,\sigma(\xi)\sigma(\zeta):\totdif^2_x\varrho_\varepsilon(x-y)K_\ell(\eta)\psi_\delta(\eta-\xi)\psi_\delta(\eta-\zeta)\,
\dif\xi\,\dif\zeta\,\dif x\, \dif y\,\dif \eta\,\dif s\\
&\quad\blue{-\stred\int_0^t\int_{(\mt )^2}\int_{\mr^3}\varrho_\varepsilon(x-y)K_\ell(\eta)\psi_\delta(\eta-\xi)\,\dif\nu^{1}_{x,s}
(\xi)\,\dif x\,\dif n^{{\psi_\delta(\eta-\cdot)}}_{2}(y,s)\,\dif\eta}\\
&\quad\blue{-\stred\int_0^t\int_{(\mt )^2}\int_{\mr^3}\varrho_\varepsilon(x-y)K_\ell(\eta)\psi_\delta(\eta-\zeta)\,
\dif\nu^{2}_{y,s}
(\zeta)\,\dif y\,\dif n^{{\psi_\delta(\eta-\cdot)}}_{1}(x,s)\,\dif\eta,}\\
&=\mathrm{J}_1+\mathrm{J}_2+\mathrm{J}_3+\mathrm{J}_4.
\end{split}
\end{equation*}
Since $\sigma$ is locally $\gamma$-H\"older continuous due to \eqref{sigma}, it holds
\begin{equation*}
\begin{split}
|\mathrm{J}_1|&\leq Ct\delta^{2\gamma}\varepsilon^{-2}\|\sigma\|_{C^\gamma([-2^{\ell+1}-\delta,2^{\ell+1}+\delta])}.
\end{split}
\end{equation*}
Next, we will show that 
\begin{equation}\label{eq:Js}
  \mathrm{J}_2+\mathrm{J}_3+\mathrm{J}_4\leq 0
\end{equation}
From the definition of the parabolic dissipation in Definition \ref{kinsolL1}, we have 
\begin{equation*}
\begin{split}
\mathrm{J}_3+\mathrm{J}_4&=-\,\stred\int_0^t\int_{(\mt )^2}\int_\mr\varrho_\varepsilon(x-y)K_\ell(\eta)\psi_\delta(\eta-u_1)\bigg|\diver_y\int^{u_2}_0\sqrt{\psi_\delta(\eta-\zeta)}\sigma(\zeta)\,\dif\zeta\bigg|^2\dif x\, \dif y\,\dif \eta\,\dif s\\
&\qquad-\,\stred\int_0^t\int_{(\mt )^2}\int_\mr\varrho_\varepsilon(x-y)K_\ell(\eta)\psi_\delta(\eta-u_2)\bigg|\diver_x\int^{u_1}_0\sqrt{\psi_\delta(\eta-\xi)}\sigma(\xi)\,\dif\xi\bigg|^2\dif x\, \dif y\,\dif \eta\,\dif s\\
&\leq -2\stred\int_0^t\int_{(\mt )^2}\int_\mr\varrho_\varepsilon(x-y)K_\ell(\eta)\sqrt{\psi_\delta(\eta-u_1)}\sqrt{\psi_\delta(\eta-u_2)}\\
&\qquad\times \diver_x\int^{u_1}_0\sqrt{\psi_\delta(\eta-\xi)}\sigma(\xi)\,\dif\xi\cdot\diver_y\int^{u_2}_0\sqrt{\psi_\delta(\eta-\zeta)}\sigma(\zeta)\,\dif\zeta\dif x\, \dif y\,\dif \eta\,\dif s\\
&=-2\stred\int_0^t\int_{(\mt )^2}\int_\mr\varrho_\varepsilon(x-y)K_\ell(\eta)\\
&\qquad\times \diver_x\int^{u_1}_0\psi_\delta(\eta-\xi)\sigma(\xi)\,\dif\xi\cdot\diver_y\int^{u_2}_0\psi_\delta(\eta-\zeta)\sigma(\zeta)\,\dif\zeta\dif x\, \dif y\,\dif \eta\,\dif s,
\end{split}
\end{equation*}
where we used the chain rule formula \eqref{eq:chainruleL1}, i.e.
$$\sqrt{\psi_\delta(\eta-\xi)}\diver_x\int^{u_1}_0\sqrt{\psi_\delta(\eta-\xi)}\sigma(\xi)\,\dif\xi=\diver_x\int^{u_1}_0\psi_\delta(\eta-\xi)\sigma(\xi)\,\dif\xi.$$
Furthermore, it holds
\begin{equation*}
\begin{split}
\mathrm{J}_2
&=2\,\stred\int_0^t\int_{(\mt )^2}\int_\mr\varrho_\varepsilon(x-y)K_\ell(\eta)\diver_y\int_{0}^{u_2}\sigma(\zeta)\psi_\delta(\eta-\zeta)\,\dif\zeta\cdot\diver_x\int_{0}^{u_1}\sigma(\xi)\psi_\delta(\eta-\xi)\,\dif\xi\,\dif x\, \dif y\,\dif \eta\,\dif s\\
&\leq-J_3-J_4
\end{split}
\end{equation*}
so \eqref{eq:Js} follows.

The last term is, due to \eqref{skorolip}, bounded as follows
\begin{equation*}
\begin{split}
\mathrm{K}&\leq C\,\stred\int_0^t\int_{(\mt )^2}\!\varrho_\varepsilon(x-y)|x-y|^2\!\int_{\mr^3}\psi_\delta(\eta-\zeta)\psi_\delta(\eta-\xi)K_\ell(\eta)\,\dif\nu^1_{x,s}(\xi)\,\dif\nu^2_{y,s}(\zeta)\,\dif x\,\dif y\,\dif \eta\,\dif s\\
&\;+C\,\stred\int_0^t\int_{(\mt )^2}\!\varrho_\varepsilon(x-y)\!\int_{\mr^3}\psi_\delta(\eta-\zeta)\psi_\delta(\eta-\xi)K_\ell(\eta)|\xi-\zeta|^2\,\dif\nu^1_{x,s}(\xi)\,\dif\nu^2_{y,s}(\zeta)\,\dif x\,\dif y\,\dd \eta\,\dif s\\
&\leq Ct\delta^{-1}\varepsilon^2+Ct \delta.
\end{split}
\end{equation*}
As a consequence, we deduce, for all $t\in[0,T]$,
\begin{equation*}
\begin{split}
\stred&\int_{\mt }\int_\mr f_1^{\pm}(x,t,\xi)\bar{f}_2^{\pm}(x,t,\xi)\,\dif\xi\,\dif x\\
&\leq\int_{(\mt )^2}\int_{\mr^3}\varrho_\varepsilon(x-y)K_\ell(\eta)\psi_\delta(\eta-\xi)\psi_\delta(\eta-\zeta)f_{1,0}(x,\xi)\bar{f}_{2,0}(y,\zeta)\,\dif\xi\,\dif\zeta\,\dif x\,\dif y\,\dif \eta\\
&\quad+C\varepsilon^{-1}\delta 2^\ell\|b'\|_{\infty,\delta,\ell}t+ Ct\delta^{2\gamma}\varepsilon^{-2}\|\sigma\|_{C^\gamma([-2^{\ell+1}-\delta,2^{\ell+1}+\delta])}+ Ct\delta^{-1}\varepsilon^2\\
&\quad+Ct \delta+\mathrm{L}(\delta, \ell)+\eta_t(\varepsilon,\delta,\ell)
\end{split}
\end{equation*}
Taking $\delta=\varepsilon^{\beta}$ with $\beta\in(1/\gamma,2)$ and letting $\varepsilon\rightarrow0$ yields
\begin{equation*}
\begin{split}
\stred\int_{\mt }\int_\mr f_1^{\pm}(x,t,\xi)\bar{f}_2^{\pm}(x,t,\xi)\,\dif\xi\,\dif x
&\leq\stred\int_{(\mt )^2}\int_{\mr^3}K_\ell(\eta)f_{1,0}(x,\eta)\bar{f}_{2,0}(x,\eta)\,\dif x\,\dif \eta\\
&+\lim_{\delta\to 0}\mathrm{L}(\varepsilon^\beta, \ell)+\lim_{\varepsilon\to 0}\eta_t(\varepsilon,\varepsilon^\beta,\ell).
\end{split}
\end{equation*}
Taking $\ell\to\infty$ we conclude
\begin{equation*}
\begin{split}
\stred\int_{\mt }\int_\mr f_1^{\pm}(x,t,\xi)\bar{f}_2^{\pm}(x,t,\xi)\,\dif\xi\,\dif x
&\leq\stred\int_{(\mt )^2}\int_{\mr^3}f_{1,0}(x,\eta)\bar{f}_{2,0}(x,\eta)\,\dif x\,\dif \eta\end{split}.
\end{equation*}
Let us now consider $f_1=f_2=f$. Since $f_0=\ind_{u_0>\xi}$ we have the identity $f_0\bar{f}_0=0$ and therefore $f^\pm(1-f^\pm)=0$ a.e. $(\omega,x,\xi)$ and for all $t$. The fact that $f^\pm$ is a kinetic function and Fubini's theorem then imply that, for any $t\in[0,T]$, there exists a set $\Sigma_t\subset\Omega\times\mt $ of full measure such that, for $(\omega,x)\in\Sigma_t$, $f^\pm(\omega,x,t,\xi)\in\{0,1\}$ for a.e. $\xi\in\mr$.
Therefore, there exist $u^\pm:\Omega\times\mt \times[0,T]\rightarrow\mr$ such that $f^\pm=\ind_{u^\pm>\xi}$ for a.e $(\omega,x,\xi)$ and all $t$. In particular, $u^\pm=\int_\mr(f^\pm-\ind_{0>\xi})\,\dif \xi$ for a.e. $(\omega,x)$ and all $t$. It follows now from Proposition \ref{prop:left-right} and the identity
\begin{equation*}
\begin{split}
|\alpha-\beta|=\int_\mr|\ind_{\alpha>\xi}-\ind_{\beta>\xi}|\,\dif \xi,\qquad\alpha,\,\beta\in\mr,
\end{split}
\end{equation*}
that $u^+=u^-=u$ for a.e. $t\in[0,T]$.
Since
$$\int_\mr \ind_{u^\pm_1>\xi}\overline{\ind_{u^\pm_2>\xi}}\,\dif\xi=(u^\pm_1-u^\pm_2)^+$$
we obtain the comparison principle \eqref{comparison}.
\end{proof}
\end{thm}

\subsection{Existence}\label{sec:ex}

In this section we prove the existence of kinetic solutions. The proof of existence is based on a three level approximation procedure. One of which is a vanishing viscosity approximation. One difficulty in the construction of a kinetic solution to \eqref{eq} is the verification of the chain-rule in Definition \ref{kinsolL1}, (ii). In contrast to the other conditions, the chain-rule is not necessarily preserved under taking weak limits, so that strong convergence of the approximating solutions is needed. This strong convergence is particularly hard to obtain in the vanishing viscosity approximation. We resolve this obstacle by employing the regularity estimates established in Section \ref{sec:averaging}.

As mentioned above, the proof of existence proceeds in three steps, corresponding to three layers of approximation. In the first step, we replace the initial condition $u_0$ by a smooth, bounded approximation $u^\kappa_0\in L^\infty(\Omega\times\T)$ such that $u^\kappa_0 \in C^\infty_c(\mt )$ $\p$-a.s. and
\begin{equation*}
\begin{split}
u^\kappa_0\to u_0\qquad&\text{in}\qquad L^1(\Omega;L^1(\mt )).
\end{split}
\end{equation*}
We also replace the diffusion matrix $A$ by a symmetric, \textit{positive definite} matrix $A^\kappa $ given by
$$A^\kappa(\xi):=\kappa Id+A(\xi),\qquad\xi\in\mr.$$
In the second step, we replace $A^\kappa$ by a \textit{bounded}, symmetric and positive definite matrix $A^{\kappa,\tau}$ given by
its square root
$$\sigma_{ij}^{\kappa,\tau} (\xi)=\sqrt{\kappa}  \delta_{ij}+
\begin{cases}
                 \sigma_{ij}(\xi),&\text{ if }\ |\xi|\leq\frac{1}{\tau },\\
		 \sigma_{ij}(\frac{sgn(\xi)}{\tau }),&\text{ if }\ |\xi|>\frac{1}{\tau },
\end{cases}$$
and we further approximate the flux $B$ by $B^\tau $, defined by setting
$$(b^\tau )'(\xi):=
\begin{cases}
   b'(\xi),&\text{ if } |\xi|\le \frac{1}{\tau } \\
   b'(\frac{sgn(\xi)}{\tau }),&\text{ if } |\xi|> \frac{1}{\tau }.
\end{cases}
$$
Since $\sigma$ and $b'$ are locally bounded, $\sigma^{\kappa,\tau }$ and $(b^\tau )'$ are bounded for each $\kappa,\tau >0$ fixed and $B^\tau $ is of sub-quadratic growth. Moreover, 
\begin{equation*}\begin{aligned}
  A^{\kappa,\tau }(\xi) &= \kappa  Id + A(\xi),\quad\forall |\xi| \le \frac{1}{\tau },\\
  B^\tau (\xi) &= B(\xi),\quad\forall |\xi| \le \frac{1}{\tau },
\end{aligned}\end{equation*}
and thus $A^{\kappa,\tau  }\to A$ locally uniformly. Hence, we consider
\begin{equation}\label{eq:approx}
\begin{split}
\dif u^{\kappa,\tau }+\diver(B^\tau (u^{\kappa,\tau }))\dd t&=\diver(A^{\kappa,\tau }(u^{\kappa,\tau })\nabla u^{\kappa,\tau } )\dif t+\varPhi(u^{\kappa,\tau })\dif W, \quad x\in\mt ,\,t\in(0,T),\\
u^{\kappa,\tau }(0)&=u_0^\kappa.
\end{split}
\end{equation}

In the case of \eqref{eq:approx}, the results from  \cite[Section 4]{dehovo} are applicable, which yields the existence and uniqueness of a weak solution to \eqref{eq:approx} and appropriate bounds. We will then pass to the limit, first employing the compactness method from \cite[Subsection 4.3]{degen} for the limit $\tau \to 0$, then proving the strong convergence in $L^1$ as $\kappa\to 0$ directly using the regularity properties established in Section \ref{sec:averaging}.

In the following subsection we establish uniform $L^p$ bounds on the approximating solutions. Next, in Subsection \ref{sec:decay} we prove uniform bounds on the corresponding kinetic dissipation measures and we conclude the proof of existence in Subsection \ref{sec:proof_ex}.

\subsubsection{$L^p$-estimates}\label{sec:lp-estimates}

Let us start with an a-priori $L^p$-estimate for solutions to \eqref{eq}.

\begin{prop}\label{prop:Lp}
Let $u$ be a kinetic solution to \eqref{eq}. Then for all $p,q\in[1,\infty)$
\begin{equation}\label{eq:Lp}
\E\esssup_{0\leq t\leq T}\|u(t)\|_{L^p}^{pq} \le C_{T,p,q} (1+\E\|u_0\|_{L^p}^{pq}),
\end{equation}
for some constant $C_{T,p,q}>0$.
\end{prop}

\begin{proof}
The proof relies on the It\^o formula applied to \eqref{eq} and the function
$$u\mapsto \int_\T(1+|u|^2)^\frac{p}{2}\dd x=\left\|(1+|u|^2)^\frac{p}{4}\right\|_{L^2}^2.$$
In order to make the following calculations rigorous, one works on the level of the approximations $u^{\kappa,\tau}$ introduced above. This leads a uniform estimate which implies \eqref{eq:Lp} for the (unique) limiting kinetic solution $u$ by lower-semicontinuity of the norm. Since this limiting procedure is standard we restrict to presenting the main, informal arguments here.

It\^o's formula yields
\begin{align*}
\left\|(1+|u(t)|^2)^\frac{p}{4}\right\|_{L^2}^2&= \left\|(1+|u_0|^2)^\frac{p}{4}\right\|_{L^2}^2 -p\int_0^t\int_\T(1+|u|^2)^{\frac{p}{2}-1}u\diver(B(u))\dd x\dd s\\
&\qquad +p\int_0^t\int_\T(1+|u|^2)^{\frac{p}{2}-1}u\diver(A(u)\nabla u)\dd x\dd s\\
&\qquad +p\sum_{k\geq 1}\int_0^t\int_\T(1+|u|^2)^{\frac{p}{2}-1}u\, g_k(x,u)\dd x\dd \beta_k\\
&\qquad +\frac{p}{2}\int_0^t\int_\T(1+|u|^2)^{\frac{p}{2}-2}\big(1+(p-1)|u|^2\big) G^2(x,u)\dd x\dd s.
\end{align*}
Due to the periodic boundary conditions, the second term on the right hand side vanishes after an integration by parts. The third term can be rewritten using integration by parts and positive semidefinitness of $A$ as
\begin{align*}
&p\int_0^t\int_\T(1+|u|^2)^{\frac{p}{2}-1}u\diver(A(u)\nabla u)\dd x\dd s\\
&=-p\int_0^t\int_\T(1+|u|^2)^{\frac{p}{2}-2}\big(1+(p-1)|u|^2\big) (\nabla u)^*A(u)(\nabla u)\dd x\dd s\leq 0.
\end{align*}
The fourth my be estimated, using \eqref{growth_new}, by
\begin{align*}
\bigg|\frac{p}{2}\int_0^t\int_\T(1+|u|^2)^{\frac{p}{2}-2}\big(1+(p-1)|u|^2\big) G^2(x,u)\dd x\dd s\bigg|&\leq C\int_0^t\int_\T(1+|u|^2)^\frac{p}{2}\,\dd x\dd s\\
&=C\int_0^t\left\|(1+|u(s)|^2)^\frac{p}{4}\right\|_{L^2}^2\dd s.
\end{align*}
Finally, for the stochastic integral, taking the supremum, the $q^{th}$ power and the expectation, we have by Burkholder-Davis-Gundy's inequality, \eqref{growth_new} and Young's inequality
\begin{align*}
&\E\sup_{0\leq t\leq T}\left|p\sum_{k\geq 1}\int_0^t\int_\T(1+|u|^2)^{\frac{p}{2}-1}u\, g_k(x,u)\dd x\dd \beta_k\right|^q\\
&\leq C\E\bigg(\int_0^T\sum_{k\geq 1}\bigg(\int_\T (1+|u|^2)^{\frac{p}{2}-1}u\, g_k(x,u)\dd x\bigg)^2\dd t\bigg)^\frac{q}{2}\\
&\leq C\E\bigg(\int_0^T\left\|(1+|u|^2)^{\frac{p}{4}}\right\|_{L^2}^2\sum_{k\geq 1}\left\|(1+|u|^2)^{\frac{p}{4}-\frac{1}{2}}|g_k(u)|\right\|_{L^2}^2\dd t\bigg)^\frac{q}{2}\\
&\leq C\E\bigg(\sup_{0\leq t\leq T}\left\|(1+|u|^2)^{\frac{p}{4}}\right\|^2_{L^2}\int_0^T\left\|(1+|u|^2)^{\frac{p}{4}}\right\|^2_{L^2}\dd t\bigg)^\frac{q}{2}\\
&\leq \frac{1}{2}\sup_{0\leq t\leq T}\left\|(1+|u|^2)^{\frac{p}{4}}\right\|^{2q}_{L^2}+C_T\E\int_0^T\left\|(1+|u|^2)^{\frac{p}{4}}\right\|^{2q}_{L^2}\dd t.
\end{align*}
In conclusion, we obtain
\begin{align*}
\E \sup_{0\leq t\leq T}\left\|(1+|u(t)|^2)^\frac{p}{4}\right\|_{L^2}^{2q}
\le&  \E \left\|(1+|u_0|^2)^\frac{p}{4}\right\|_{L^2}^{2q} +C_{T,q}\E\int_0^T\left\|(1+|u|^2)^{\frac{p}{4}}\right\|^{2q}_{L^2}\dd t.
\end{align*}
Thus, Gronwall's lemma yields
\begin{align*}
\E\sup_{0\leq t\leq T}\left\|(1+|u|^2)^\frac{p}{4}\right\|_{L^2}^{2q}\leq C_{T,q}\left\|(1+|u_0|^2)^\frac{p}{4}\right\|_{L^2}^{2q}.
\end{align*}
Since
$$\|u\|_{L^p}^{pq}\leq \left\|(1+|u|^2)^\frac{p}{4}\right\|_{L^2}^{2q}\leq C_{p,q}\big(1+\|u\|_{L^p}^{pq}\big),$$
this concludes the proof.
\end{proof}

\subsubsection{Decay of the kinetic measure}\label{sec:decay}

To appreciate the difficulty and methods introduced in the following we recall that in the deterministic case
  $$
 \partial_t f(t)+ b(\xi) \cdot \nabla f(t) + A(\xi):\totdif^2 f(t)=\partial_{\xi} m.
$$
bounds on the kinetic measure $m$ are easily derived (informally) by testing with $1_{[k,\infty)}(\xi)$ and integrating in $t,x,\xi$, which corresponds to computing the derivative $\partial_t (u-k)_+$ via the chain-rule. In the stochastic case, this has to be replaced by the It\^o formula informally leading to terms of the form $\int g^2_k(x,u)\delta_{u=k}d\xi dxdt$ which are not easy to control. Therefore, new techniques are needed in the stochastic case and a less restrictive decay assumption on the kinetic measure is used (cf.\ the discussion before \eqref{eq:m_decay}).

\begin{lemma}\label{lem:loc_unif_bound2}
Let $u_0\in L^r(\Omega;L^1(\mt))$ for some $r\in[1,\infty)$ and let $u$ be a kinetic solution to \eqref{eq}. Then, for all $k\in\N$,
\begin{equation*}
 \E| m([0,T]\times\mt \times [-k,k])|^r
 \le C(r,k,T,\E\|u_0\|^r_{L^1}),
\end{equation*}
for some $C>0$ depending on $D$ only.
\end{lemma}
\begin{proof}
{\em Step 1:} For $k>0$, set
$$\theta_k(u)=\ind_{[-k,k]}(u),\qquad\Theta_k(u)=\int_{-k}^u\int_{-k}^r\theta_k(s)\,\dd s\,\dd r.$$ Let $\gamma\in C^1_c([0,T))$ be nonnegative such that $\gamma(0)=1$, $\gamma'\leq 0$. 
Then after a preliminary step of regularization we may take $\varphi(t,x,\xi)=\gamma(t)\Theta'_{k}(\xi)$ in \eqref{eq:kinformulL1} to get
\begin{align*} 
\begin{aligned}
&\E\bigg|\int_0^T\int_{\mt}\Theta_k(u(t,x))|\gamma'(t)|\,\dd x\,\dd t\bigg|^r+\,\E\bigg|\int_{A^+_k}\gamma(t)\,\dd m(t,x,\xi)\bigg|^r\\
&\lesssim \E\bigg|\int_0^T\int_{\mt}\gamma(t) G^2(x,u(t,x))\theta_k(u(t,x))\,\dd x\,\dd t\bigg|^r\\
&+\E\bigg|\sum_{j\geq 1}\int_0^T\int_\T g_j(x,u(t,x))\Theta'_k(u(t,x))\gamma(t)\dd x\,\dd\beta_k(t)\bigg|^r+\E\bigg|\int_{\mt}\Theta_k(u_0(x))\,\dd x\bigg|^r
\end{aligned}
\end{align*}
where $A^+_{k}=[0,T]\times\mt\times[-k,k]$. Since $0 \le \Theta_k(u) \le 2k(k+|u|)$ and, due to \eqref{growth_new},
\begin{equation*}
\frac{1}{2}G^2(x,u)\theta_k(u)\leq \frac{1}{2}D(1+|u|)^2\ind_{[-k,k]}(u)\leq D(1+k^2).
\end{equation*}
Since $0\leq \Theta'_k(u)\leq 2k\ind_{-k\leq u}$ we may estimate the stochastic integral using the Burkholder-Davis-Gundy inequality, the Minkowski integral inequality, \eqref{growth_new} and the Young inequality as follows
\begin{align*}
&\E\bigg|\sum_{j\geq 1}\int_0^T\int_\T g_j(x,u(t,x))\Theta'_k(u(t,x))\gamma(t)\dd x\,\dd\beta_j(t)\bigg|^r\\
&\lesssim \E\bigg(\int_0^T\sum_{j\geq 1}\bigg(\int_\T g_j(x,u)\Theta'_k(u)\dd x\bigg)^2\dd t\bigg)^\frac{r}{2}\\
&\leq 2k^r\,\E\bigg(\int_0^T\bigg(\int_\T \bigg(\sum_{j\geq1}|g_j(x,u)|^2\bigg)^\frac{1}{2}\dd x\bigg)^2\dd t\bigg)^\frac{r}{2}\\
&\lesssim 2Dk^r\,\E\bigg(\int_0^T\bigg(\int_\T \left(1+|u|\right)\dd x\bigg)^2\dd t\bigg)^\frac{r}{2}\\
&\leq 2Dk^r\,\E\bigg(\sup_{0\leq t\leq T}\|1+u\|_{L^1}\int_0^T\|1+u\|_{L^1}\dd t\bigg)^\frac{r}{2}\\
&\leq Dk^r\,\E\sup_{0\leq t\leq T}\|1+u\|^r_{L^1}+Dk^r\,\E\bigg(\int_0^T\|1+u\|_{L^1}\dd t\bigg)^{r}\\
&\leq Dk^r(1+T^r)+Dk^r(1+T^r)\,\E\sup_{0\leq t\leq T}\|u\|^r_{L^1}\\
&\leq C(D,k^r,T^r,\E\|u_0\|^r_{L^1}),
\end{align*}
where we also used Proposition \ref{prop:Lp} with $p=1, q=r$ for the last step. The claim follows.
\end{proof}

\begin{lemma}\label{lem:energy_bounds}
Let $u_0\in L^1(\Omega;L^1(\mt))$, $u$ be a kinetic solution to \eqref{eq} and $\Theta \in C^2(\R)$ be nonnegative, convex such that $\Theta''(\xi)(1+|\xi|^2) \le C_\Theta (1+\Theta(\xi))$ for some constant $ C_\Theta >0$. Then, 
\begin{align*} 
\begin{aligned}
&\sup_{t\in[0,T]}\E \int_{\mt}\Theta(u(t,x))\,\dd x +\,\E\int_0^T\int\Theta''(\xi)\,\dd m(t,x,\xi)\le C \left( \E\int_{\mt}\Theta(u_0(x))\,\dd x +1\right)
\end{aligned}
\end{align*}
for some $C>0$ depending on $D, C_\Theta$ only.
\end{lemma}
\begin{proof}
Let $\Theta$ be as in the statement and  
$\gamma\in C^1_c([0,T))$ be nonnegative such that $\gamma(0)=1$, $\gamma'\leq 0$. 
After a preliminary step of regularization we may take $\varphi(t,x,\xi)=\gamma(t)\Theta'(\xi)$ in \eqref{eq:kinformulL1} to get
\begin{align}\label{eq:57a}
\begin{aligned}
&\E \int_0^T\int_{\mt}\Theta(u(t,x))|\gamma'(t)|\,\dd x\,\dd t +\,\E\int\Theta''(\xi)\gamma(t)\,\dd m(t,x,\xi)\\
&\lesssim \E\int_0^T\int_{\mt}\gamma(t) G^2(x,u(t,x))\Theta''(u(t,x))\,\dd x\,\dd t+\E\int_{\mt}\Theta(u_0(x))\,\dd x 
\end{aligned}
\end{align}
By assumption 
\begin{equation*}
\frac{1}{2}G^2(x,u)\Theta''(u)\leq D(1+|u|^2)\Theta''(u)\leq C(1+\Theta(u)).
\end{equation*}
Letting now $\gamma \to 1_{[0,t]}$ an application of Gronwall's Lemma finishes the proof.
\end{proof}

We proceed with an estimate which is a modification of \cite[Proposition 16]{debus2} and applies to the case multiplicative noise.

\begin{prop}\label{prop:decay1}
Let $u_0\in L^1(\Omega;L^1(\mt))$ and let $u$ be an kinetic solution to \eqref{eq}. Then
\begin{equation*}
\begin{split}
&\esssup_{0\leq t\leq T}\E\|(u(t)- 2^n)^+\|_{L^1_x}+\esssup_{0\leq t\leq T}\E\|(u(t)+ 2^n)^-\|_{L^1_x}+\frac{1}{2^n}\E m(A_{2^n})\\
&\qquad\qquad\leq C\big(T,\E\|u_0\|_{L^1_x}\big)\delta(n),\quad\forall n\in\mn_0,
\end{split}
\end{equation*}
where
$$A_{2^n}=[0,T]\times\mt\times\{\xi\in\mr;\,2^n\leq |\xi|\leq 2^{n+1}\}$$ and $\delta(n)$ depends only on $D$ and on the functions
\begin{equation*}
R\mapsto\E\|(u_0-R)^+\|_{L^1_x},\qquad R\mapsto\E\|(u_0+R)^-\|_{L^1_x},
\end{equation*}
and satisfies $\lim_{n\to\infty}\delta(n)=0.$

Furthermore, 
\begin{equation}\label{equiint23}
\begin{split}
&\E\Big(\esssup_{0\leq t\leq T}\|(u(t)- 2^{n})^+\|_{L^1_x}\Big)+\E\Big(\esssup_{0\leq t\leq T}\|(u(t)+ 2^{n})^-\|_{L^1_x}\Big)\\
&\qquad\qquad\leq C\big(T,D,\E\|u_0\|_{L^1_x}\big)\big[\delta(n)+\delta^{\frac{1}{2}}(n)\big],
\end{split}
\end{equation}
where $\delta(n)$ is as above, in addition possibly depending on the function
 $$R\mapsto \E\|(1+|u|)1_{R\le |u|}\|_{L^1_{x,t}}.$$
\end{prop}

\begin{proof}
{\em Step 1:} For $k>0$, set
$$\theta_k(u)=\frac{1}{k}\ind_{k\leq u\leq 2k},\qquad\Theta_k(u)=\int_0^u\int_0^r\theta_k(s)\,\dd s\,\dd r.$$
Let $\eta=(21D)\vee 1$, $\gamma\in C^1_c([0,\eta T))$ be nonnegative such that $\gamma(0)=1$, $\gamma'\leq 0$. 
Then, after a preliminary step of regularization, we may take $\varphi(t,x,\xi)=\gamma(\eta t)\Theta'_{k}(\xi)$ in \eqref{eq:kinformulL1} to get
\begin{align}\label{eq:57-2}
\begin{aligned}
\eta\E\int_0^T\int_{\mt}&\Theta_k(u(t,x))|\gamma'(\eta t)|\,\dd x\,\dd t+\frac{1}{k}\,\E\int_{A^+_k}\gamma(\eta t)\,\dd m(t,x,\xi)\\
&=\frac{1}{2}\,\E\int_0^T\int_{\mt}\gamma(\eta t) G^2(x,u(t,x))\theta_k(u(t,x))\,\dd x\,\dd t+\E\int_{\mt}\Theta_k(u_0(x))\,\dd x
\end{aligned}
\end{align}
where $A^+_{k}=[0,T]\times\mt\times\{\xi\in\mr;\,k\leq \xi\leq 2k\}$. Note that
\begin{equation*}
 \bigg(u-\frac{3}{2}k\bigg)^+\leq \Theta_k(u)\leq (u-k)^+,
\end{equation*}
and, using \eqref{growth_new},
\begin{equation}\label{theta}
\frac{1}{2 \eta}G^2(x,u)\theta_k(u)\leq D\frac{1+|u|^2}{\eta k}\ind_{k\leq u\leq 2k}\leq D\frac{1+4k^2}{ \eta k}\frac{(u-l)^+}{k-l}
\end{equation}
for all $k>l\geq 0,\,u\in\mr$. We choose $l=\frac{3}{4}k$ and observe, by choice of $\eta$ and for $k\ge 1$,
\begin{equation}\label{alpha}
\alpha:=D\frac{1+4k^2}{\eta k(k-l)} =4D\frac{1+4k^2}{\eta k^2} <1. 
\end{equation}
Consequently, we deduce from \eqref{eq:57-2} that
\begin{align}\label{it}
\begin{aligned}
&\E\int_0^T\int_{\mt}(u(t,x)-2l)^+|\gamma'(\eta t)|\,\dd x\,\dd t\\
&\qquad\leq \alpha\,\E\int_0^T\int_{\mt}(u(t,x)-l)^+\gamma(\eta t)\,\dd x\,\dd t+\E\int_{\mt}(u_0(x)-l)^+\,\dd x.
\end{aligned}
\end{align}
We can now iterate the above procedure by replacing $k$
by $2 k$.
To do so, we need to check that applying \eqref{eq:57-2} to $\tilde k=2 k$, the chosen constant $\alpha$ appearing in \eqref{it} is the same as before. 
The condition \eqref{alpha} now reads
$$4D\frac{1+4(2k)^2}{4 \eta k^2}<1, $$
which holds true due to \eqref{alpha}.
We deduce that
\begin{align*}
&\E\int_0^T\int_{\mt}(u(t,x)-4l)^+|\gamma'(\eta t)|\,\dd x\,\dd t\\
&\qquad\leq \alpha\,\E\int_0^T\int_{\mt}(u(t,x)-2l)^+\gamma(\eta t)\,\dd x\,\dd t+\E\int_{\mt}(u_0(x)-2l)^+\,\dd x.
\end{align*} 
Since the same argument can be applied to $\tilde k=2^nk$, $\tilde l=2^nl$ for any $n\in \N$, we set
$$\psi_n(t):=\E\int_{\mt}(u(t,x)-2^nl)^+\,\dd x$$
and after letting $\gamma$ approximate $\ind_{[0,\eta t]}$ we finally 
obtain
\begin{align*}
\psi_{n+1}(t)\leq \alpha\int_0^t\psi_n(s)\,\dd s+\psi_n(0).
\end{align*}
Now we proceed similarly as in \cite[Proposition 16]{debus2}.
Since $(u - l)^+ \leqslant (u - 2 l)^+ + l$ we have
\[ \psi_0 (t) =\mathbb{E} \int_{\mathbb{T}^N} (u (t, x) - l)^+ d x \leqslant
   \mathbb{E} \int_{\mathbb{T}^N} (u (t, x) - 2 l)^+ d x + l = \psi_1 (t) +
   l \]
so
\[ \psi_1 (t) \leqslant \alpha \int_0^t \psi_0 (s) d s + \psi_0 (0) \leqslant
   \alpha \int_0^t \psi_1 (s) d s + \alpha l t + \psi_0 (0) \]
and Gronwall's lemma implies
\[ \psi_0 (t) \leqslant M := C (T, \| u_0 \|_{L^1}) . \]
Thus we deduce
\[ \psi_1 (t) \leqslant \alpha t M + \psi_0 (0), \qquad \psi_2 (t) \leqslant
   \alpha^2 \frac{t^2}{2} M + \alpha t \psi_0 (0) + \psi_1 (0), \]
   and generally
\[ \psi_{n + 1} (t) \leqslant \alpha^{n + 1} \frac{t^{n + 1}}{n + 1!} M +
   \sum_{k = 0}^n \alpha^k \frac{t^k}{k!} \psi_{n - k} (0) . \]
Let
\[ \delta (n) = \alpha^n + \sum_{k = 0}^{n - 1} \alpha^k \psi_{n - 1 - k} (0),
\]
which satisfies $\delta(n)\to 0$ as $n\to \infty$ due to the assumption on $u_0$.
Therefore, it follows that
\[ \psi_{n + 1} (t) \leqslant (M + 1) \mathrm{e}^T \delta (n + 1)  \]
and consequently
$$\esssup_{0\leq t\leq T}\E\int_{\mt}(u(t,x)-2^{n+1}l)^+\,\dd x\leq C\big(T,\E\|u_0\|_{L^1_x}\big)\delta(n+1).$$
Thus, as a consequence of \eqref{eq:57-2}, \eqref{theta} and the fact that $k>l$ we get
\begin{align*}
\esssup_{0\leq t\leq T}\E\int_{\mt}(u(t,x)-2^{n}k)^+\,\dd x+\frac{1}{2^nk \eta}\E\int_{A^+_{2^nk}}\dd m(t,x,\xi)\leq C(T,\E\|u_0\|_{L^1_x})\delta(n).
\end{align*}
Choosing $k=1$ finishes the proof.

{\em Step 2:} 
To prove \eqref{equiint23} we start similarly as in {\em Step 1} but take the supremum in time before taking the expectation. The iterative inequality then reads
\begin{align*}
&\E\esssup_{0\leq t\leq T}\int_{\mt}(u(t,x)-2^{n+1}l)^+\,\dd x\,\dd t\\
\leq& \alpha\,\E\int_0^T\int_{\mt}(u(t,x)-2^nl)^+\,\dd x\,\dd t+\E\int_{\mt}(u_0(x)-2^nl)^+\,\dd x\\
& +\E\sup_{0\leq t\leq T}\bigg|\sum_{j\geq 1}\int_0^t\int_{\mt}g_j(x,u(t,x))\Theta'_{2^nk}(u(t,x))\,\dd x\,\dd \beta_j(t)\bigg|\\
\leq& C\big(T,\alpha,\E\|u_0\|_{L^1_x}\big)\delta(n) +C\,\E\bigg(\int_0^T\int_{\mt}G^2(x,u(t,x))|\Theta'_{2^nk}(u(t,x))|^2\,\dd x\,\dd t\bigg)^\frac{1}{2}.
\end{align*}
Using \eqref{growth_new} we estimate
\begin{align*}
&\E\sup_{0\leq t\leq T}\bigg|\sum_{j\geq 1}\int_0^t\int_{\mt}g_j(x,u(t,x))\Theta'_{2^nk}(u(t,x))\,\dd x\,\dd \beta_j(t)\bigg|\\
&\quad\lesssim \E\bigg(\int_0^T\sum_j\bigg(\int_{\mt} g_j(u)\Theta'_{2^nk}(u)\,\dif x\bigg)^2\dif t\bigg)^\frac{1}{2}\\
&\quad\lesssim\E\bigg(\int_0^T\bigg(\int_{\mt} (1+|u|)\ind_{2^nk\leq u}\,\dif x\bigg)^2\dif t\bigg)^\frac{1}{2}\\
&\quad \leq\E\bigg(\sup_{0\leq t\leq T}\|1+|u|\|_{L^1_x}\int_0^T\int_{\mt}(1+|u|)\ind_{2^nk\leq u}\dif x\,\dif t\bigg)^\frac{1}{2}\\
&\quad\leq \Big(\E\sup_{0\leq t\leq T}\|1+|u|\|_{L^1_x}\Big)^\frac{1}{2}\bigg(\E\int_0^T\int_{\mt}(1+|u|)\ind_{2^nk\leq u}\,\dif x\,\dif t\bigg)^\frac{1}{2}.
\end{align*}
The right hand side converges to $0$ as $n\to\infty$ due to the dominated convergence theorem. Hence the estimate \eqref{equiint23} follows using Proposition \ref{prop:Lp}.
\end{proof}

Based on the equiintegrability estimate \eqref{equiint23} we can deduce that kinetic solutions have continuous paths in $L^1(\mt)$ a.s.
\begin{cor}[Continuity in time]
Let $u_0\in L^1(\Omega;L^1(\mt))$ and let $u$ be a kinetic solution to \eqref{eq}. Then there exists a representative of $u$ with almost surely continuous trajectories in $L^1(\mt)$.
\end{cor}
\begin{proof} Based on Proposition \ref{prop:left-right}, Theorem \ref{thm:uniqueness} and \eqref{equiint23} we are in a position to apply \cite[Lemma 17]{debus2} which implies the continuity in $L^1$. Indeed, let us first show that $u^+$ constructed in Theorem \ref{thm:uniqueness} is $\p$-a.s.\ right-continuous in $L^1(\T)$. Due to Proposition \ref{prop:left-right}, we have that $f^+(t+\varepsilon)\rightharpoonup^*f^+(t)$ in $L^\infty(\T\times\R)$  $\p$-a.s. as $\varepsilon\to0$. Due to Theorem \ref{thm:uniqueness}, for all $t\in[0,T)$ the kinetic function $f^+(t)$ is at equilibrium, that is, $f^+(t,x,\xi)=\ind_{u^+(t,x)>\xi}$ for a.e.\ $(\omega,x,\xi)$. Finally, \eqref{equiint23} implies
$$\lim_{n\to\infty}\E\sup_{t\in[0,T]}\|(u^+(t)-2^n)^\pm\|_{L^1_x}=0.$$
Hence, there exists a subsequence (not relabeled) which converges $\p$-a.s., that is, $\p$-a.s.
$$\lim_{n\to\infty}\sup_{t\in[0,T]}\|(u^+(t)-2^n)^\pm\|_{L^1_x}=0.$$
Consequently, \cite[Lemma 17]{debus2} applies and yields the convergence
$$u^+(t+\varepsilon)\to u^+(t)\qquad\text{in}\qquad L^1(\T)\qquad\text{as}\qquad \varepsilon\to 0.$$
The same arguments show that $u^-$ constructed in Theorem \ref{thm:uniqueness} is $\p$-a.s.\ left-continuous in $L^1(\T)$. Finally, the fact that $u^+(t)=u^-(t)$ for all $t\in [0,T]$ can be proved as in \cite[Corollary 12]{debus}.
\end{proof}

\subsubsection{The proof of existence}\label{sec:proof_ex}

\begin{thm}\label{thm:ex}
  Let $u_0\in L^r(\Omega;L^1(\T))$ for some $r>1$ and assume \eqref{growth_new}. 
  Then there exists a kinetic solution $u$ to \eqref{eq} satisfying $u\in C([0,T];L^1(\T))$, $\p$-a.s. and for all $p,q\in[1,\infty)$ there exists a constant $C>0$ such that
  \begin{equation*}
     \E \esssup_{t\in [0,T]}\|u(t)\|_{L^p}^{pq} \le C (1+\E\|u_0\|_{L^p}^{pq}).
  \end{equation*}
\end{thm}

\begin{proof} 
\textit{Step 1:} According to  \cite[Section 4]{dehovo}, there exists a unique kinetic solution $u^{\kappa,\tau }$ to \eqref{eq:approx}. Let us denote the corresponding kinetic measure by $m^{\kappa,\tau }$  and observe that it consists of only the parabolic dissipation measure, that is, 
\begin{align*}
  m^{\kappa,\tau }(\dif t,\dif x,\dif\xi)=|\sigma^{\kappa,\tau}(\xi)\nabla u^{\kappa,\tau}|^2\,\dif\delta_{u^{\kappa,\tau}(t,x)}(\xi)\,\dif x\,\dif t.
\end{align*}
By \cite[Theorem 5.2]{dehovo} that for all $p\in[2,\infty)$, uniformly in ${\kappa,\tau }$,
\begin{equation*}
\begin{split}
&\E\sup_{0\leq t\leq T}\|u^{\kappa,\tau }(t)\|_{L^p_x}^p+p(p-1)\E\int_0^T\int_{\mt}|u^{\kappa,\tau}|^{p-2}|\sigma^{\kappa,\tau}(u^{\kappa,\tau})\nabla u^{\kappa,\tau}|^2\dif x\,\dif t\\
&\leq C\big(1+\E\|u_0^\kappa\|_{L^p_x}^p\big),
\end{split}
\end{equation*}
and consequently
\begin{equation*}
\begin{split}
  \E\big|m^{\kappa,\tau }([0,T]\times\mt \times\mr)\big|\leq C(1+\E\|u_0^\kappa\|_{L^2_x}^2)
\end{split}\end{equation*}
with a constant $C$ independent of $\kappa,\tau $. Moreover, since $\sigma^{\kappa,\tau}\geq \sqrt{\kappa} Id$ the following space regularity holds true uniformly in $\tau$
\begin{equation}\label{spacereg}
\E\int_0^T\int_{\mt}|\nabla u^{\kappa,\tau}|^2\dif x\,\dif t\le \frac{C}{\kappa}(1+\E\|u_0^\kappa\|_{L^2_x}^2).
\end{equation}

Now, we have all in hand to apply the compactness method developed in \cite[Subection 4.3]{degen}.
To be more precise, we replace the result of \cite[Theorem 4.4]{degen} with the estimate \eqref{spacereg}, then all of the bounds from \cite[Corollary 4.5, Lemma 4.6, Corollary 4.7, Corollary 4.8]{degen} hold true uniformly in $\tau $ and therefore we obtain the results of \cite[Theorem 4.9, Proposition 4.10]{degen} as well. Namely, there exists a probability space $(\tilde\Omega,\tilde\mf,\tilde\prst)$ with a sequence of random variables $(\tilde{u}^{\kappa,\tau },\tilde{W}^\tau )$ and $(\tilde{u}^{\kappa},\tilde{W})$ such that
\begin{enumerate}
 \item the laws of $(\tilde{u}^{\kappa,\tau },\tilde{W}^\tau )$ and $({u}^{\kappa,\tau },{W})$ coincide,
 \item $\tilde{u}^{\kappa,\tau }\to \tilde u^{\kappa}$ in $L^2(0,T;L^2(\mt ))\cap C([0,T];H^{-1}(\mt ))$ a.s.,
 \item $\tilde u^{\kappa,\tau}(0)\to \tilde u^{\kappa}(0)$ in $L^2(\mt )$ a.s.,
 \item $\tilde W^\tau \to \tilde W$ in $C([0,T];\mathfrak{U}_0)$ a.s.
\end{enumerate}
Finally, we obtain the result  \cite[Lemma 4.12]{degen} (with the corresponding expression of the parabolic dissipative measure, cf. \cite[Theorem 6.4]{dehovo}). To be more precise, denoting
$$\dif\tilde m^{\kappa,\tau}(t,x,\xi)=|\sigma^{\kappa,\tau}(\tilde u^{\kappa,\tau})\nabla \tilde u^{\kappa,\tau}|^2\,\dif\delta_{\tilde u^{\kappa,\tau}(t,x)}(\xi)\,\dif t\,\dif x,$$ it holds true that

\begin{enumerate}
\item there exists a set of full Lebesgue measure $\mathcal{D}\subset[0,T]$ which contains $t=0$ such that
$$\ind_{\tilde u^{\kappa,\tau}>\xi}\rightharpoonup^* \ind_{\tilde u>\xi}\quad\text{in}\quad L^\infty(\tilde\Omega\times\mt\times\R)\quad\forall t\in[0,T],$$
\item there exists a kinetic measure $\tilde m^\kappa$ such that\footnote{The space $L^2_w(\tilde\Omega;\mathcal{M}_b([0,T]\times\mt\times\R))$ contains all weak-star-measurable mappings $n:\tilde\Omega\to \mathcal{M}_b([0,T]\times\mt\times\R)$ such that $\E\|n\|^2_{\mathcal{M}_b}<\infty$.}
$$\tilde m^{\kappa,\tau}\rightharpoonup^* \tilde m^\kappa\quad\text{in}\quad L^2_w(\tilde\Omega;\mathcal{M}_b([0,T]\times\mt\times\R)).$$
Moreover, $\tilde m^\kappa$ can be written as $\tilde n^\kappa_1+\tilde n^\kappa_2$, where
$$\dif\tilde n_1^{\kappa}(t,x,\xi)=\bigg|\diver\int_0^{\tilde u^{\kappa}}\sigma^{\kappa}(\zeta)\,\dif\zeta\bigg|^2\,\dif\delta_{\tilde u^{\kappa}(t,x)}(\xi)\,\dif t\,\dif x,$$
and $\tilde n_2^\kappa $ is a.s. a nonnegative measure on $[0,T]\times\mt\times\R$.
\end{enumerate}

As a consequence, we deduce that $(\tilde u^\kappa,\tilde m^\kappa,\tilde W)$ is a martingale kinetic solution to
\begin{equation}\label{eq:approx1}
\begin{split}
\dif u^{\kappa }+\diver(B(u^{\kappa }))\dd t&=\diver(A^{\kappa }(u^{\kappa })\nabla u^{\kappa } )\dif t+\varPhi(u^{\kappa })\dif W, \qquad x\in\mt ,\,t\in(0,T),\\
u^{\kappa }(0)&=u_0^\kappa.
\end{split}
\end{equation}
Therefore, in view of pathwise uniqueness, the same Yamada-Watanabe type argument  as in \cite[Subsection 4.5]{degen} applies and yields the existence of a unique (pathwise) kinetic solution $u^{\kappa}$  to \eqref{eq:approx1} in the sense of \cite[Definition 2.2]{dehovo}.
This solution is defined on the original stochastic basis $(\Omega,\mf,(\mf_t),\p)$ and satisfies the equation with the original Wiener process $W$.

\textit{Step 2:}
First,
we observe that the assumptions of Theorem \ref{thm:regularity}, namely, \eqref{eq:non-deg-local} is satisfied uniformly in $\kappa$. Indeed, let
  $$\mathcal{L}^\kappa(iu,in,\xi):=i(u+b(\xi)\cdot n)+n^*A^\kappa(\xi)n.$$
Then
  $$\mathcal{L}^\kappa(iu,in,\xi)=\mathcal{L}(iu,in,\xi)+\kappa|n|^2$$
hence, for some constant $C>0$,
  $$\{\xi\in\mr;\,|\mathcal{L}^\kappa(iu,in,\xi)|\leq\delta\}
  =\{\xi\in\mr;\,|\mathcal{L}(iu,in,\xi)|\leq C(\delta-\kappa|n|^2)\}
  \subset \{\xi\in\mr;\,|\mathcal{L}(iu,in,\xi)|\leq C\delta\}$$
which implies for all $\phi\in C^\infty_c(\R)$
  $$\omega^\phi_{\mathcal L^\kappa}(J;\delta)
  \leq\omega^\phi_{\mathcal L}(J;C\delta)
  \lesssim_\phi \bigg(\frac{\delta}{J^\beta}\bigg)^\alpha\qquad\forall \delta>0,\;\forall J\gtrsim 1.$$
Moreover,
  $$\mathcal L^\kappa_\xi(iu,in,\xi)= \mathcal L_\xi(iu,in,\xi)$$
and thus
  $$
\sup_{\substack{u\in\mr,n\in \mz^N\\|n|\sim J}}\sup_{\xi\in \supp\eta}|\mathcal L^\kappa_\xi(iu,in,\xi)|\leq\sup_{\substack{u\in\mr,n\in \mz^N\\|n|\sim J}}\sup_{\xi\in  \supp\eta}|\mathcal L_\xi(iu,in,\xi)|\lesssim_\phi J^{\beta}.
$$
Consequently, Theorem \ref{thm:regularity} applies and we obtain the estimate for $(u^{\kappa})^\phi=\int_\R \chi_{u^{\kappa}} \,\phi\, \dif\xi$
\begin{equation*}
\begin{split}
&\bigg(\E\int_0^T\|(u^{\kappa})^\phi(t)\|_{W^{s,r}_x}^r\dif t\bigg)^{1/r}\\
&\qquad\lesssim_\phi \del{T^\frac{1}{2}} \|u^\kappa_{0}\|^{1/2}_{L^1_{\omega,x}}   + \|u^{\kappa}\|^{1/2}_{L^1_{\omega,t,x}}+\sup_{0\leq t\leq T}\|u^{\kappa}(t)\|_{L^1_{\omega,x}}+ \| m^{\kappa}\ind_{\supp\phi}\|_{L^1_{\omega}\mathcal{M}_{t,x,\xi}} + 1
\end{split}
\end{equation*}
for some $r>1$ and $s>0$. In view of Proposition \ref{prop:Lp} and Lemma \ref{lem:loc_unif_bound2} the right hand side  can be further estimated uniformly in $\kappa$ as follows
\begin{equation}\label{average}
\begin{split}
&\bigg(\E\int_0^T\|(u^{\kappa})^\phi(t)\|_{W^{s,r}_x}^r\dif t\bigg)^{1/r}\lesssim_\phi \del{T^\frac{1}{2}} \|u_{0}\|_{L^1_{\omega,x}}   + 1.
\end{split}
\end{equation}

As the next step, we prove that $(u^{\kappa})$ is Cauchy in $L^1(\Omega\times[0,T],\mathcal{P},\dif\p\otimes\dif t;L^1(\mt))$. This is based on computations similar to Section \ref{subsec:uniq}. Accordingly, let us again fix the two standard Dirac sequences $(\varrho_\varepsilon)$ and $(\psi_\delta)$. Let us also define a cut-off function $(K_\ell)$ as follows: Let $K\in C^\infty(\mr)$ be such that $0\leq K(\xi)\leq 1$, $K\equiv 1$ if $|\xi|\leq 1$, $K\equiv 0$ if $|\xi|\geq 2$, and $|K'(\xi)|\leq 1,$  and define
$$K_\ell(\xi):=K\Big(\frac{\xi}{2^\ell}\Big),\qquad \ell\in \mn.$$

For any two approximate solutions $u^{\kappa_1},\,u^{\kappa_2}$ we have
\begin{equation}\label{doubling1}
\begin{split}
&\stred\int_{\mt }\big(u^{\kappa_1}(t)-u^{\kappa_2}(t)\big)^+\,\dif x=\stred\int_{\mt }\int_\mr f^{\kappa_1}(x,t,\xi) \bar{f}^{\kappa_2}(x,t,\xi)\,\dif \xi\,\dif x\\
&= \stred\int_{(\mt )^2}\int_{\mr^3}f^{\kappa_1}(x,t,\xi)\bar{f}^{\kappa_2}(y,t,\zeta)K_\ell(\eta)\varrho_\varepsilon(x-y)\psi_\delta(\eta-\zeta)\psi_\delta(\eta-\xi)
\,\dif \xi\,\dif\zeta\,\dif \eta\,\dif x\,\dif y\\
&\qquad+\eta_t(\kappa_1,\kappa_2,\varepsilon,\delta,\ell),
\end{split}
\end{equation}
where $\varepsilon$, $\delta$ and $\ell$ are chosen arbitrarily and their value will be fixed later. The idea now is to show that the mollification error $\eta_t(\kappa_1,\kappa_2,\varepsilon,\delta,\ell)$ can be made arbitrarily small uniformly in $\kappa_1,\kappa_2$, which will rely on the equi-integrability estimate, Proposition \ref{prop:decay1}, as well as \eqref{average} based on the averaging lemma, Theorem \ref{thm:regularity}, the a priori $L^p$-estimates, Proposition \ref{prop:Lp}, and  the bound for the kinetic measure from Lemma \ref{lem:loc_unif_bound2}. Indeed, we write
\begin{equation*}
\begin{split}
&\eta_t(\kappa_1,\kappa_2,\varepsilon,\delta,\ell)=\stred\int_{\mt }\int_\mr f^{\kappa_1}(x,t,\eta)\bar{f}^{\kappa_2}(x,t,\eta)\,\dif \eta\,\dif x\\
&\quad-\stred\int_{(\mt )^2}\int_{\mr^3}f^{\kappa_1}(x,t,\xi)\bar{f}^{\kappa_2}(y,t,\zeta)K_\ell(\eta)\varrho_\varepsilon(x-y)\psi_\delta(\eta-\zeta)\psi_\delta(\eta-\xi)
\,\dif \xi\,\dif\zeta\,\dif \eta\,\dif x\,\dif y\\
&=\stred\int_{\mt }\int_\mr f^{\kappa_1}(x,t,\eta)\bar{f}^{\kappa_2}(x,t,\eta)\big(1-K_\ell(\eta)\big)\,\dif \eta\,\dif x\\
&\quad+\bigg(\stred\int_{\mt }\int_\mr f^{\kappa_1}(x,t,\eta)\bar{f}^{\kappa_2}(x,t,\eta)K_\ell(\eta)\,\dif \eta\,\dif x\\
&\quad\quad-\stred\int_{(\mt )^2}\int_{\mr}f^{\kappa_1}(x,t,\eta)\bar{f}^{\kappa_2}(y,t,\eta)K_\ell(\eta)\varrho_\varepsilon(x-y)\,\dif \eta\,\dif x\,\dif y\bigg)\\
&\quad+\bigg(\stred\int_{(\mt )^2}\int_\mr f^{\kappa_1}(x,t,\eta)\bar{f}^{\kappa_2}(y,t,\eta)K_\ell(\eta)\varrho_\varepsilon(x-y)\,\dif \eta\,\dif x\,\dif y\\
&\quad\quad-\stred\int_{(\mt )^2}\int_{\mr^2}f^{\kappa_1}(x,t,\eta)\bar{f}^{\kappa_2}(y,t,\zeta)K_\ell(\eta)\varrho_\varepsilon(x-y)\psi_\delta(\eta-\zeta)\,\dif \eta\,\dif\zeta\,\dif x\,\dif y\bigg)\\
&\quad+\bigg(\stred\int_{(\mt )^2}\int_{\mr^2}f^{\kappa_1}(x,t,\eta)\bar{f}^{\kappa_2}(y,t,\zeta)K_\ell(\eta)\varrho_\varepsilon(x-y)\psi_\delta(\eta-\zeta)\,\dif \eta\,\dif\zeta\,\dif x\,\dif y\\
&\quad\quad-\stred\int_{(\mt )^2}\int_{\mr^3}f^{\kappa_1}(x,t,\xi)\bar{f}^{\kappa_2}(y,t,\zeta)K_\ell(\eta)\varrho_\varepsilon(x-y)\psi_\delta(\eta-\xi)\psi_\delta(\eta-\zeta)
\,\dif \xi\,\dif\zeta\,\dif \eta\,\dif x\,\dif y\bigg)\\
&=\mathrm{H}_1+\mathrm{H}_2+\mathrm{H}_3+\mathrm{H}_4
\end{split}
\end{equation*}
and  estimate each of the error terms on the right hand side separately using the above mentioned results.
First,
\begin{align*}
|\mathrm{H}_1|&\leq\E\int_{\mt}\ind_{u^{\kappa_1}(t,x)>u^{\kappa_2}(t,x)}\int_{u^{\kappa_2}(t,x)}^{u^{\kappa_1}(t,x)}\ind_{|\eta|\ge 2^\ell}\,\dif \eta\,\dif x\\
&\leq \E\int_{\mt}(u^{\kappa_1}(t,x)-2^\ell)^+\,\dif x+\E\int_{\mt}(-2^\ell-u^{\kappa_2}(t,x))^+\,\dif x.
\end{align*}
Now we observe that the result of Proposition \ref{prop:decay1} holds true for the approximate solutions $u^{\kappa_1},u^{\kappa_2}$ uniformly in $\kappa_1,\kappa_2$. Consequently,
$$|\mathrm{H}_1|\leq \delta(\ell),$$
where $\delta(\ell)$ was defined in Proposition \ref{prop:decay1}, and
$$\lim_{\ell\to\infty}\sup_{0\leq t\leq T}|\mathrm{H}_1|=0\quad\text{ uniformly in }\quad \kappa_1,\kappa_2,\varepsilon,\delta.$$
Second, in order to estimate
\begin{equation*}
\begin{split}
|\mathrm{H}_2|&=\bigg|\stred\int_{(\mt )^2}\varrho_\varepsilon(x-y)\int_\mr K_\ell(\eta)\ind_{u^{\kappa_1}(t,x)>\eta}\Big[\ind_{u^{\kappa_2}(t,x)\leq\eta}-\ind_{u^{\kappa_2}(t,y)\leq \eta}\Big]\,\dif\eta\,\dif x\,\dif y\bigg|,
\end{split}
\end{equation*}
we write
\begin{align*}
&\bigg|\int_\mr K_\ell(\eta)\ind_{u^{\kappa_1}(x)>\eta}\Big[\ind_{u^{\kappa_2}(x)\leq\eta}-\ind_{u^{\kappa_2}(y)\leq \eta}\Big]\,\dif\eta\bigg|\\
&\quad=\bigg|\int_\mr K_\ell(\eta)\ind_{u^{\kappa_1}(t,x)>\eta}\Big[\ind_{\eta\in[u^{\kappa_2}(x),u^{\kappa_2}(y))}-\ind_{\eta\in[u^{\kappa_2}(y),u^{\kappa_2}(x))}\Big]\,\dif\eta\bigg|\\
&\quad=\int_\mr K_\ell(\eta)\ind_{u^{\kappa_1}(t,x)>\eta}\Big[\ind_{\eta\in[u^{\kappa_2}(x),u^{\kappa_2}(y))}+\ind_{\eta\in[u^{\kappa_2}(y),u^{\kappa_2}(x))}\Big]\,\dif\eta\\
&\quad\leq\ind_{u^{\kappa_2}(x)<u^{\kappa_2}(y)}\int_\mr K_\ell(\eta)\Big[\ind_{u^{\kappa_2}(x)\leq\eta}-\ind_{u^{\kappa_2}(y)\leq \eta}	\Big]\dif\eta\\
&\quad\quad+\ind_{u^{\kappa_2}(y)<u^{\kappa_2}(x)}\int_\mr K_\ell(\eta)\Big[\ind_{u^{\kappa_2}(y)\leq\eta}-\ind_{u^{\kappa_2}(x)\leq \eta}\Big]\dif\eta\\
&\quad=\ind_{u^{\kappa_2}(x)<u^{\kappa_2}(y)}\int_\mr K_\ell(\eta)\Big[\chi_{u^{\kappa_2}(x)}(\eta)-\chi_{u^{\kappa_2}(y)}(\eta)\Big]\dif\eta\\
&\quad\quad+\ind_{u^{\kappa_2}(y)<u^{\kappa_2}(x)}\int_\mr K_\ell(\eta)\Big[\chi_{u^{\kappa_2}(y)}(\eta)-\chi_{u^{\kappa_2}(x)}(\eta)\Big]\dif\eta.
\end{align*}
Thus using \eqref{average}
\begin{equation*}
\begin{split}
|H_2|&\leq \E\int_{(\mt)^2}\varrho_\varepsilon(x-y)\big|(u^{\kappa_2})^{K_\ell}(t,x)-(u^{\kappa_2})^{K_\ell}(t,y)\big|\,\dif x\,\dif y,\\
&\le C_\ell\varepsilon^s.
\end{split}
\end{equation*}

Thus, for all $\ell\in\mn$,
$$\lim_{\varepsilon\to0}\int_0^T|\mathrm{H}_2|\,\dif t=0\quad\text{ uniformly in }\quad \kappa_1,\kappa_2,\delta.$$
Third,
\begin{equation*}
\begin{split}
|\mathrm{H}_3|
=&\bigg|\stred\!\int_{(\mt )^2}\!\varrho_\varepsilon(x-y)\!\int_\mr K_\ell(\eta)\ind_{u^{\kappa_1}(t,x)>\eta}\int_{\mr}\psi_\delta(\eta-\zeta)\Big[\ind_{u^{\kappa_1}(t,y)\leq\eta}-\ind_{u^{\kappa_1}(t,y)\leq\zeta}\Big]\dif\zeta\dif\eta\dif x\dif y\bigg|\\
\leq&\stred\int_{(\mt )^2}\int_\mr\varrho_\varepsilon(x-y)\,\ind_{u^{\kappa_1}(t,x)>\eta}\int_{\eta-\delta}^\eta\psi_\delta(\eta-\zeta)\,\ind_{\zeta<u^{\kappa_1}(t,y)\leq\eta}\,\dif\zeta\,\dif\eta\,\dif x\,\dif y\\
&+\stred\int_{(\mt )^2}\int_\mr\varrho_\varepsilon(x-y)\,\ind_{u^{\kappa_1}(t,x)>\eta}\int_{\eta}^{\eta+\delta}\psi_\delta(\eta-\zeta)\,\ind_{\eta<u^{\kappa_1}(t,y)\leq\zeta}\,\dif\zeta\,\dif\eta\,\dif x\,\dif y\\
\leq&\frac{1}{2}\,\stred\int_{(\mt )^2}\varrho_\varepsilon(x-y)\int_{u^{\kappa_1}(t,y)}^{\min\{u^{\kappa_1}(t,x),u^{\kappa_1}(t,y)+\delta\}}\dif \eta\,\dif x\,\dif y\\
&+\frac{1}{2}\,\stred\int_{(\mt )^2}\varrho_\varepsilon(x-y)\int_{u^{\kappa_1}(t,y)-\delta}^{\min\{u^{\kappa_1}(t,x),u^{\kappa_1}(t,y)\}}\dif \eta\,\dif x\,\dif y\\
\leq& C\delta
\end{split}
\end{equation*}
hence
$$\lim_{\delta\to0}\sup_{0\leq t\leq T}|\mathrm{H}_3|=0\quad\text{ uniformly in }\quad \kappa_1,\kappa_2,\varepsilon,\ell.$$
And finally
\begin{align*}
|\mathrm{H}_4|&=\bigg|\stred\int_{(\mt )^2}\varrho_\varepsilon(x-y)\int_{\mr^2}\ind_{u^{\kappa_2}(t,y)\le\zeta}K_\ell(\eta)\psi_\delta(\eta-\zeta)\dif\zeta\,\\
&\qquad\qquad\times\int_{\mr}\big[\ind_{u^{\kappa_1}(t,x)>\eta}-\ind_{u^{\kappa_1}(t,x)>\xi}\big]\psi_\delta(\eta-\xi)
\,\dif \xi\,\dif \eta\,\dif x\,\dif y\bigg|\\
&\leq\stred\int_{(\mt )^2}\varrho_\varepsilon(x-y)\int_{\mr^2}\ind_{u^{\kappa_2}(t,y)\le\zeta}\psi_\delta(\eta-\zeta)\int_{\eta}^{\eta+\delta}\ind_{\eta<u^{\kappa_1}(t,x)\leq \xi}\,\psi_\delta(\eta-\xi)
\,\dif \xi\,\dif\zeta\,\dif \eta\,\dif x\,\dif y\\
&\quad+\stred\int_{(\mt )^2}\varrho_\varepsilon(x-y)\int_{\mr^2}\ind_{u^{\kappa_2}(t,y)\le\zeta}\psi_\delta(\eta-\zeta)\int_{\eta-\delta}^\eta\ind_{\xi<u^{\kappa_1}(t,x)\leq \eta}\,\psi_\delta(\eta-\xi)
\,\dif \xi\,\dif\zeta\,\dif \eta\,\dif x\,\dif y\\
&\leq\stred\int_{(\mt )^2}\varrho_\varepsilon(x-y)\int_{u^{\kappa_2}(t,y)}^\infty\int^{u^{\kappa_1}(t,x)}_{u^{\kappa_1}(t,x)-\delta}\psi_\delta(\eta-\zeta)\,\dif\zeta\,\dif \eta\,\dif x\,\dif y\\
&\quad+\stred\int_{(\mt )^2}\varrho_\varepsilon(x-y)\int_{u^{\kappa_2}(t,y)}^\infty\int_{u^{\kappa_1}(t,x)}^{u^{\kappa_1}(t,x)+\delta}\psi_\delta(\eta-\zeta)\,\dif\zeta\,\dif \eta\,\dif x\,\dif y\leq\delta
\end{align*}
and therefore
$$\lim_{\delta\to0}\sup_{0\leq t\leq T}|\mathrm{H}_4|=0\quad\text{ uniformly in }\quad \kappa_1,\kappa_2,\varepsilon,\ell.$$
Heading back to \eqref{doubling1} and using the same calculations as in Proposition \ref{prop:doubling}, we deduce that
\begin{equation*}
\begin{split}
\stred&\int_{\mt }\big(u^{\kappa_1}(t)-u^{\kappa_2}(t)\big)^+\,\dif x\leq \eta_t(\kappa_1,\kappa_2,\varepsilon,\delta,\ell)+\eta_0(\kappa_1,\kappa_2,\varepsilon,\delta,\ell)+\mathrm{I}+\mathrm{J}+\mathrm{J}^\#+\mathrm{K}+L(\delta,\ell),
\end{split}
\end{equation*}
where, with $\delta(\ell)$ as in Proposition \ref{prop:decay1},
$$\lim_{\ell\to\infty}\sup_{0\leq t\leq T}L(\delta,\ell)=\lim_{\ell\to\infty}\delta(\ell)=0\quad\text{ uniformly in }\quad \kappa_1,\kappa_2,\varepsilon,\delta.$$
The terms $\mathrm{I},\,\mathrm{J},\,\mathrm{K}$ are defined and can be dealt with exactly as in Proposition \ref{prop:doubling} and Theorem \ref{thm:uniqueness}. The term $\mathrm{J}^\#$ is defined as
\begin{equation*}
\begin{split}
\mathrm{J}^\#=&(\tau+\sigma)\,\stred\int_0^t\int_{(\mt )^2}\int_{\mr^3}f^{\kappa_1}\bar{f}
^{n,\kappa_2}\Delta_x\varrho_\varepsilon(x-y)K_\ell(\eta)\psi_\delta(\eta-\xi)\psi_\delta(\eta-\zeta)\,
\dif\xi\,\dif\zeta\,\dif x\, \dif y\,\dif s\\
&-\stred\int_0^t\int_{(\mt )^2}\int_{\mr^3}\varrho_\varepsilon(x-y)K_\ell(\eta)\psi_\delta(\eta-\xi)\psi_\delta(\eta-\zeta)\,
\dif\nu^ {\kappa_1} _ { x , s }
(\xi)\,\dif x\,\dif n_{2}^{\kappa_2}(y,s,\zeta)\\
&-\stred\int_0^t\int_{(\mt )^2}\int_{\mr^3}\varrho_\varepsilon(x-y)K_\ell(\eta)\psi_\delta(\eta-\xi)\psi_\delta(\eta-\zeta)\,
\dif\nu^{\kappa_2}_{y,s}
(\zeta)\,\dif y\,\dif n_{2}^{\kappa_1}(x,s,\xi),
\end{split}
\end{equation*}
where we used the notation $\nu^{\kappa_1}_{x,s}(\xi)=\delta_{u^{\kappa_1}(s,x)}(\xi)$ and similarly for $\nu^{\kappa_2}_{y,s}(\zeta)$.
Thus,
\begin{equation*}
\begin{split}
\mathrm{J}^\#=&(\kappa_1+\kappa_2)\,\stred\int_0^t\int_{(\mt )^2}\int_\mr\varrho_\varepsilon(x-y)K_\ell(\eta)\psi_\delta(\eta-u^{\kappa_1})\psi_\delta(\eta-u^{\kappa_2})\nabla_x u^{\kappa_1}\cdot\nabla_y u^{\kappa_2}\dif\eta\dif x\dif y\dif s\\
&-\kappa_1\,\stred\int_0^t\int_{(\mt )^2}\int_\mr\varrho_\varepsilon(x-y)K_\ell(\eta)\psi_\delta(\eta-u^{\kappa_1})\psi_\delta(\eta-u^{\kappa_2})|\nabla_x u^{\kappa_1}|^2\dif\eta\dif x\dif y\dif s\\
&-\kappa_2\,\stred\int_0^t\int_{(\mt )^2}\int_\mr\varrho_\varepsilon(x-y)K_\ell(\eta)\psi_\delta(\eta-u^{\kappa_1})\psi_\delta(\eta-u^{\kappa_2})|\nabla_y u^{\kappa_2}|^2\dif\eta\dif x\dif y\dif s\\
=&-\,\stred\int_0^t\int_{(\mt )^2}\int_\mr\varrho_\varepsilon(x-y)K_\ell(\eta)\psi_\delta(\eta-u^{\kappa_1})\psi_\delta(\eta-u^{\kappa_2})\big|\sqrt{\kappa_1}\,\nabla_x u^{\kappa_1}-\sqrt{\kappa_2}\,\nabla_y u^{\kappa_2}\big|^2\dif\eta\dif x\dif y\dif s\\
&+(\sqrt{\kappa_1}-\sqrt{\kappa_2}\,)^2\,\stred\int_0^t\int_{(\mt )^2}\int_\mr\varrho_\varepsilon(x-y)K_\ell(\eta)\psi_\delta(\eta-u^{\kappa_1})\psi_\delta(\eta-u^{\kappa_2})\nabla_x u^{\kappa_1}\cdot\nabla_y u^{\kappa_2}\dif\eta\dif x\dif y\dif s\\
=&\mathrm{J}^\#_1+\mathrm{J}^\#_2.
\end{split}
\end{equation*}
The first term on the right hand side is nonpositive, for the second one we have
\begin{equation*}
\begin{split}
\big|\mathrm{J}^\#_2\big|&\leq (\sqrt{\kappa_1}-\sqrt{\kappa_2}\,)^2\,\stred\int_0^t\int_{(\mt )^2}\int_{\mr^3}f^{\kappa_1} \bar{f}^{\kappa_2}K_\ell(\eta)\psi_\delta(\eta-\xi)\psi_\delta(\eta-\zeta)\big|\Delta_x\varrho_\varepsilon(x-y)\big|\,\dif \eta\,\dif\xi\dif \zeta\dif x\dif y\dif s
\end{split}
\end{equation*}
and proceeding similarly as for $\mathrm{I}$ we get
\begin{equation*}
\begin{split}
\big|\mathrm{J}^\#_2\big|
&\leq C(\sqrt{\kappa_1}-\sqrt{\kappa_2}\,)^2\varepsilon^{-2}2^\ell.
\end{split}
\end{equation*}
Consequently, we see that 
\begin{equation*}
\begin{split}
\stred&\int_0^T\int_{\mt }\big(u^{\kappa_1}(t)-u^{\kappa_2}(t)\big)^+\,\dif x\,\dif t\lesssim_T \delta(\ell)\\
&\quad+C_\ell\varepsilon^s+\delta+\varepsilon^{-1}\delta 2^\ell\|b'\|_{L^\infty(-2\ell-\delta,2\ell+\delta)}t+ \delta^{2\gamma}\varepsilon^{-2}\|\sigma\|_{C^\gamma([-\ell-\delta,\ell+\delta])}+ \delta^{-1}\varepsilon^2+\delta\\
&\quad+(\kappa_1+\kappa_2)\,\varepsilon^{-2}2^\ell.
\end{split}
\end{equation*}
Therefore, given $\vartheta>0$ one can fix $\ell$ sufficiently large so that the first term on the right hand side is estimated by $\vartheta/3,$ then fix $\varepsilon$ and $\delta$ small enough so that the second line also estimated by $\vartheta/3$ and then find $\iota>0$ such that the third line is estimated by $\vartheta/3$ for any $\kappa_1,\kappa_2<\iota$. Thus, we have shown that the set of approximate solutions $(u^{\kappa})$ is Cauchy in $L^1(\Omega\times[0,T],\mathcal{P},\dif\prst\otimes\dif t;L^1(\mt ))$, as $\kappa\rightarrow0$. Hence there exists $u\in L^1(\Omega\times[0,T],\mathcal{P},\dif\prst\otimes\dif t;L^1(\mt ))$ such that
\begin{equation}\label{eq:sc}
u^\kappa\to u\quad\text{in }\quad L^1(\Omega\times[0,T],\mathcal{P},\dif\prst\otimes\dif t;L^1(\mt )).
\end{equation}

Since $u_0\in L^r(\Omega;L^1(\T))$ for some $r>1$, we can choose $(u_0^\kappa)$ uniformly bounded in $L^r(\Omega;L^1(\T))$. By Lemma \ref{lem:loc_unif_bound2}, we obtain that for each $k>0$
\begin{equation}\label{eq:mnr}
\sup_\kappa \E |m^\kappa(B_k)|^r \le C_k,
\end{equation}
where $B_k := [0,T]\times\mt \times [-k,k]$. Consequently, the sequence $(m^\kappa)$ is bounded in $L^r(\Omega;\mathcal{M}( B_k))$. Following the same arguments as \cite[proof of Theorem 20]{debus2}, we extract a subsequence (not relabeled) and a random Borel measure $m$ on $[0,T]\times\T\times \R$ such that $m^\kappa \rightharpoonup^* m$ weakly$^*$ in $L^r(\Omega;\mathcal{M}(B_k))$ for every $k\in\N$. 

Since the estimates derived in Proposition \ref{prop:decay1} are uniform with respect to $\kappa$, the limit $m$ satisfies Definition \ref{meesL1},  (ii). We further note that $m$ satisfies Definition \ref{meesL1}, (i), since this property is stable with respect to weak limits. Hence, $m$ is a kinetic measure.

We next check that $(u,m)$ is a kinetic solution to \eqref{eq} in the sense of Definition \ref{kinsolL1}. Let $\phi\in C^\infty_c(\mr)$ be nonnegative and denote by $\Phi$ a function satisfying $\Phi''=\phi$ and $\Phi\ge 0$. Then, similarly to \eqref{eq:57a}, we obtain
\begin{align*}
\begin{aligned}
&\E\int_{[0,T]\times\mt \times\mr}\phi(\xi)\,\dd m^\kappa(t,x,\xi)\\
&\leq\frac{1}{2}\,\E\int_0^T\int_{\mt }G^2(x,u^\kappa(t,x))\phi(u^\kappa(t,x))\,\dd x\,\dd t+\E\int_{\mt }\Phi(u^\kappa_0(x))\,\dd x.
\end{aligned}
\end{align*}
Hence, due to \eqref{growth_new} and since
 $$\E\int_{[0,T]\times\mt \times\mr}\phi(\xi)\,\dd n_1^\kappa(t,x,\xi)=\E\int_0^T\int_{\mt  }\bigg|\diver\int_0^{u^\kappa}\sqrt{\phi(\zeta)}\,\big[\sqrt{\kappa}Id+\sigma(\zeta)\big]\,\dif \zeta\bigg|^2\dif x\,\dif t,$$
we obtain that
\begin{align}\label{eq:n1_bound}
&\E\int_{0}^T\int_{\mt }\bigg|\diver\int_0^{u^\kappa}\sqrt{\phi(\zeta)}\big[\sqrt{\kappa}Id+\sigma(\zeta)\big]\,\dd\zeta\bigg|^2\dd x\,\dd t\leq C\big(T,\|\phi\|_{L^\infty},\supp\phi,\E\|u_0\|_{L^1_x}\big).
\end{align}
From the strong convergence \eqref{eq:sc} and the fact that $\sqrt{\phi}\,\sigma\in C_c(\mr)$, we conclude using integration by parts, for all $\eta\in L^2(0,T;C^1(\mt ))$, $\psi \in L^\infty(\Omega)$,
\begin{equation*}
\begin{split}
&\E\psi\int_0^T\!\int_{\mt  }\!\bigg(\diver\int_0^{u^\kappa}\sqrt{\phi(\zeta)}\,\big[\sqrt{\kappa}Id+\sigma(\zeta)\big]\,\dif \zeta\bigg)\eta(t,x)\,\dif x\,\dif t\\
&\qquad\to\E\psi\int_0^T\!\int_{\mt  }\!\bigg(\diver\int_0^u\sqrt{\phi(\zeta)}\,\sigma(\zeta)\,\dif \zeta\bigg)\eta(t,x)\,\dif x\,\dif t,
\end{split}
\end{equation*}
and therefore, using \eqref{eq:n1_bound},
\begin{equation}\label{eq:convergence_2}
\diver\int_0^{u^\kappa}\sqrt{\phi(\zeta)}\,\big[\sqrt{\kappa}Id+\sigma(\zeta)\big]\,\dif \zeta\overset{}{\rightharpoonup}\diver\int_0^u\sqrt{\phi(\zeta)}\,\sigma(\zeta)\,\dif\zeta\quad\text{ in }\quad L^2(\Omega\times[0,T]\times\mt  ).
\end{equation}
Hence, Definition \ref{kinsolL1}, (i)  is satisfied.

Concerning the chain rule formula \eqref{eq:chainruleL1}, we observe that the corresponding version  holds true for all $u^\kappa$, since $u^\kappa$ is a kinetic solution, i.e. for any $\phi_1,\phi_2\in C_c(\mr)$, $\phi_1,\phi_2\geq0$,
\begin{equation}\label{chain}
\diver\int_0^{u^\kappa}\phi_1(\zeta)\phi_2(\zeta)\big[\sqrt{\kappa}Id+\sigma(\zeta)\big]\,\dif \zeta=\phi_1(u^\kappa)\diver\int_0^{u^\kappa}\phi_2(\zeta)\big[\sqrt{\kappa}Id+\sigma(\zeta)\big]\,\dif\zeta
\end{equation}
holds true as an equality in  $L^2(\Omega\times[0,T]\times\mt  ).$
Due to \eqref{eq:convergence_2} we can pass to the limit on the left hand side and, making use of the strong-weak convergence, also on the right hand side of \eqref{chain}. In conclusion, Definition \ref{kinsolL1}, (ii) holds.

Let now \blue{$\phi\in C^\infty_c(\R)$, $\phi\geq0$, and let $n^\phi$ be defined as in Definition \ref{kinsolL1}, (iii). Since $u^\kappa$ is a kinetic solution, $(u^\kappa,m^\kappa)$ satisfy \eqref{eq:kinformulL1} with the corresponding diffusion matrix $A^\kappa$. Passing to the limit $\kappa\to0$ yields \eqref{eq:kinformulL1} for $(u,m)$ with the original diffusion matrix $A$. It remains to prove that for all $\varphi\in C^\infty_c([0,T]\times\mt)$, $\varphi\geq0$, $m(\varphi\phi) \ge n^\phi(\varphi)$ $\p$-a.s.} Since each $m^\kappa$ can be decomposed into the sum of the parabolic dissipation measure $n_1^\kappa$ and the corresponding (nonnegative) entropy dissipation measure $n_2^\kappa$, that is $m^\kappa=n_1^\kappa + n_2^\kappa$, from \eqref{eq:mnr} it follows that
$$
\sup_\kappa \E |n_1^\kappa(B_k)|^r \le C_k.
$$
By the same argument as above, we extract a subsequence (not relabeled) and a random measure $o_1$ such that $n_1^\kappa \rightharpoonup^* o_1$ weakly$^*$ in $L^r(\Omega;\mathcal{M}(B_k))$ for all $k\in\N$. Since $m^\kappa \ge n_1^\kappa$ $\prst$-a.s.\ we have $m \ge o_1$, $\prst$-a.s. Moreover, since by sequentially weak lower semicontinuity of the norm, it follows for all $\varphi\in L^\infty([0,T]\times\T)$, $\psi \in L^\infty(\Omega)$, $\p$-a.s.,
\begin{equation*}
\begin{split}
\E \psi n^\phi(\varphi^2)
&=\E \psi \int_0^T\int_{\T}\bigg|\diver\int_0^{u}\sqrt{\phi(\zeta)}\,\sigma(\zeta)\,\dif\zeta\bigg|^2\varphi^2(t,x)\,\dif x\,\dif t\\
&\leq\liminf_{\kappa\rightarrow0}\E \psi\int_0^T\int_{\T}\bigg|\diver\int_0^{u^\kappa}\sqrt{\phi(\zeta)}\,\big[\sqrt{\kappa}Id+\sigma(\zeta)\big]\,\dif \zeta\bigg|^2\varphi^2(t,x)\,\dif x\,\dif t=\E \psi o_1(\phi\varphi^2).
\end{split}
\end{equation*}
and, thus, \blue{$n^\phi$ given by \eqref{eq:par_meas} satisfies $n^\phi(\cdot)\leq o_1(\phi \,\cdot)$, $\prst$-a.s.}\ which completes the proof.
\end{proof}

\appendix

\section{Multiplier lemmas}

In this section we establish various auxiliary results concerning Fourier multipliers used in Section \ref{sec:averaging}.

We employ the following definition of the Fourier and inverse Fourier transform on $\mt=[0,2\pi]^N$
\begin{align*}
\mathcal{F}_x v(n)&=\frac{1}{(2\pi)^{N/2}}\int_{\mt}v(x)\me^{-in\cdot x}\,\dd x,\qquad n\in\mz^N,\quad v\in L^1(\mt),\\
\mathcal{F}^{-1}_x w(x)&=\frac{1}{(2\pi)^{N/2}}\sum_{n\in\mz^N}w(n)\me^{in\cdot x},\qquad x\in\mt,\quad w\in L^1(\mz^N).
\end{align*}
and correspondingly on the whole space $\mr^N$ or with respect to time.

We start with a result that will be useful in the sequel.

\begin{lemma}[Truncation property]\label{lemma:truncation}
Let $m(n,\xi)=ib(\xi)\cdot n+n^*A(\xi)n$. Then $m$ satisfies the truncation property in $L^1$ uniformly in $\xi\in\mr$, i.e. for every bump function $\psi$ and $\delta>0$ it holds true
\begin{align*}
\bigg\|\mathcal{F}_x^{-1}\psi\bigg(\frac{m(n,\xi)}{\delta}\bigg)\hat f(n)\bigg\|_{L^1_x} &\leq C\|f\|_{L^1_x},
\end{align*}
where the constant $C$ is independent of $\delta$ and $\xi$.
 
\begin{proof}
{\em Step 1:} First, we observe that it is enough to prove the claim for each of the (real) multipliers $m_1(n,\xi)=b(\xi)\cdot n$ and $m_2(n,\xi)=n^*A(\xi)n$ separately. Indeed, using Fourier series we can write
\begin{align*}
\psi\bigg(\frac{m(n,\xi)}{\delta}\bigg)&=\mathcal{F}_z^{-1}\mathcal{F}_z\psi\bigg(\frac{m(n,\xi)}{\delta}\bigg)=\sum_{j,k\in\mz}\hat\psi(j,k)\,\me^{2\pi i j \frac{m_1(n,\xi)}{\delta}}\,\me^{2\pi i k \frac{m_2(n,\xi)}{\delta}}.
\end{align*}
Let $\tilde\psi$ be a bump function that equals to 1 in a ball containing the one-dimensional projections of the support of $\psi$. Then
\begin{align*}
\psi\bigg(\frac{m(n,\xi)}{\delta}\bigg)&=\sum_{j,k\in\mz}\hat\psi(j,k)\,\me^{2\pi i j \frac{m_1(n,\xi)}{\delta}}\tilde\psi\bigg(\frac{m_1(n,\xi)}{\delta}\bigg)\,\me^{2\pi i k \frac{m_2(n,\xi)}{\delta}}\tilde\psi\bigg(\frac{m_2(n,\xi)}{\delta}\bigg)\\
&=\sum_{j,k\in\mz}\hat\psi(j,k)\tilde\psi_j\bigg(\frac{m_1(n,\xi)}{\delta}\bigg)\tilde\psi_{k}\bigg(\frac{m_2(n,\xi)}{\delta}\bigg),
\end{align*}
where $\tilde\psi_j(x):=\me^{2\pi ijx}\tilde\psi(x)$. Since $\hat\psi(j,k)$ decays rapidly in $j$ and $k$ and the $C^l$-norm of $\tilde\psi_j$ grows at most polynomially in $j$, the conclusion follows from the fact that a product of $L^1$-multipliers is an $L^1$-multiplier with the norm given by the product of the corresponding norms.

{\em Step 2:} Let us consider $m_2$ and assume in addition that $A(\xi)$ positive definite, that is $\sigma(\xi)$ is invertible. It is well-known that the set of $L^1$-Fourier multipliers coincides with the set of Fourier transforms of finite Borel measures and the norm is given by the total variation of the corresponding measure, see \cite[Theorem 2.5.8]{grafakos}. Therefore we need to estimate
$$\bigg\|\mathcal{F}_x^{-1} \psi\bigg(\frac{m_2(n,\xi)}{\delta}\bigg)\bigg\|_{L^1_x}$$
by a constant independent of $\delta$ and $\xi$.
To this end, we study the continuous version of the above norm first and then make use of the Poisson summation formula. We have
\begin{align}\label{eq:normmult}
\begin{aligned}
\bigg\|\mathcal{F}^{-1}& \psi\bigg(\frac{z^*A(\xi)z}{\delta}\bigg)\bigg\|_{L^1_y}=\frac{1}{(2\pi)^{N/2}}\int_{\mr^N}\bigg|\int_{\mr^N}\psi\bigg(\frac{z^*A(\xi)z}{\delta}\bigg)\,\me^{i y\cdot z}\,\dif z\bigg|\,\dif y\\
&=\frac{1}{(2\pi)^{N/2}}\frac{\delta^{N/2}}{|\det\sigma(\xi)|}\int_{\mr^N}\bigg|\int_{\mr^N}\psi\big(|z|^2\big)\,\me^{i y\cdot \sqrt{\delta}\sigma(\xi)^{-1} z}\,\dif z\bigg|\,\dif y\\
&=\frac{\delta^{N/2}}{|\det\sigma(\xi)|}\int_{\mr^N}\Big|\mathcal{F}^{-1}\Big[\psi\big(|\cdot|^2\big)\Big]\big(\sqrt{\delta}\sigma(\xi)^{-1}y\big)\Big|\,\dif y\\
&=\int_{\mr^N}\Big|\mathcal{F}^{-1}\Big[\psi\big(|\cdot|^2\big)\Big](y)\Big|\,\dif y.
\end{aligned}
\end{align}
Since $\psi(|\cdot|^2)$ is a bump function, it is a Fourier transform of some function, say $\varphi$, from the Schwartz space hence
\begin{equation*}
\bigg\|\mathcal{F}^{-1}\psi\bigg(\frac{z^*A(\xi)z}{\delta}\bigg)\bigg\|_{L^1_y}=\|\varphi\|_{L^1_y}.
\end{equation*}
In other words, uniformly in $\delta $ and $\xi$,
$$\psi\bigg(\frac{z^*A(\xi)z}{\delta}\bigg)$$
is a Fourier transform of a finite Borel measure $\varphi$ hence it is an $L^1$-multiplier on the continuous space $\mr^N$. Let us now define
$$\varphi^{\text{per}}(x):=\sum_{l\in \mz^N}\varphi(x+l).$$
Then
$$\|\varphi^{\text{per}}\|_{L^1(\mt)}\leq \|\varphi\|_{L^1(\mr^N)}$$
and
\begin{align*}
\mathcal{F}_{\mt}\varphi^{\text{per}}(n)&=\frac{1}{(2\pi)^{N/2}}\int_{\mt}\sum_{l\in\mz^N}\varphi(x+l)\me^{- i n\cdot x}\dd x=\frac{1}{(2\pi)^{N/2}}\int_{\mt}\sum_{l\in\mz^N}\varphi(x+l)\me^{- i n\cdot (x+l)}\dd x\\
&=\frac{1}{(2\pi)^{N/2}}\int_{\mr^N}\varphi(x)\me^{-i n\cdot x}\dd x=\mathcal{F}_{\mr^N}\varphi(n)=\psi\bigg(\frac{n^*A(\xi)n}{\delta}\bigg)
\end{align*}
so we deduce that, uniformly in $\delta$ and $\xi$,
$$\psi\bigg(\frac{n^*A(\xi)n}{\delta}\bigg)$$
is an $L^1$-multiplier on $\mt$.

{\em Step 3:} Let us now assume that $A(\xi)$ is degenerate. If $A(\xi)=0$ then clearly
\begin{align*}
\bigg\|\mathcal{F}_x^{-1}\psi\bigg(\frac{m_2(n,\xi)}{\delta}\bigg)\hat f(n)\bigg\|_{L^1_x}&=\psi(0)\|f\|_{L^1_x}.
\end{align*}
Let $A(\xi)$ be a matrix with rank $K$. It was seen in \eqref{eq:normmult} that the desired multiplier norm is invariant under invertible linear transformations and therefore we may assume without loss of generality that $A(\xi)$ is diagonal with all the eigenvalues equal to 1. Then we denote by $z^K=(z_1,\dots,z_J)$, $z^{N-K}=(z_{K+1},\dots,z_N)$ and proceed as above to obtain
\begin{align*}
\bigg\|\mathcal{F}^{-1}& \psi\bigg(\frac{|z^K|^2}{\delta}\bigg)\bigg\|_{L^1_y}\\
&=\frac{1}{(2\pi)^{N/2}}\int_{\mr^N}\bigg|\int_{\mr^{N-K}}\me^{i y^{N-K}\cdot z^{N-K}}\int_{\mr^K}\psi\bigg(\frac{|z^K|^2}{\delta}\bigg)\,\me^{i y^K\cdot z^K}\,\dif z^K\,\dif z^{N-K}\bigg|\,\dif y\\
&=\frac{1}{(2\pi)^{N/2}}\delta^{K/2}\int_{\mr^N}\bigg|\int_{\mr^{N-K}}\me^{i y^{N-K}\cdot z^{N-K}}\,\dif z^{N-K}\,\mathcal{F}_{y^K}^{-1}\Big[\psi\big(|\cdot|^2\big)\Big]\big(\sqrt{\delta}y^K\big)\bigg|\,\dif y\\
&=\delta^{K/2}\int_{\mr^K}\Big|\mathcal{F}_{y^K}^{-1}\Big[\psi\big(|\cdot|^2\big)\Big]\big(\sqrt{\delta}y^K\big)\Big|\,\dif y^K\\
&=\int_{\mr^K}\Big|\mathcal{F}_{y^K}^{-1}\Big[\psi\big(|\cdot|^2\big)\Big](y^K)\Big|\,\dif y^K.
\end{align*}
Since $\psi(|\cdot|^2)$ is a bump function on $\mr^K$ there exists $\varphi$ from the Schwartz space $\mathcal{S}(\mr^K)$ such that $\mathcal{F}_{y^K}^{-1}\psi(|\cdot|^2)=\varphi(\cdot)$.
Hence
\begin{align*}
\bigg\|\mathcal{F}^{-1}& \psi\bigg(\frac{|z^K|^2}{\delta}\bigg)\bigg\|_{L^1_y}=\big\|\mathcal{F}_{y^K}^{-1}\psi(|\cdot|^2)\big\|_{L^1_{y^K}}=\|\varphi\|_{L^1_{y^K}}
\end{align*}
so, uniformly in $\delta$,
$$\psi\bigg(\frac{|z^K|^2}{\delta}\bigg)$$
is an $L^1$-multiplier on the continuous space $\mr^N$. The proof for the discreet multiplier
$$\psi\bigg(\frac{|n^K|^2}{\delta}\bigg)$$
follows the reasoning of {\em Step 2} by including only minor modifications.

{\em Step 4:} The same calculations lead to the desired estimate for the case of $m_1$.
\end{proof}
\end{lemma}

\begin{cor}\label{cor:truncation}
Let $m(u,n,\xi)=i(u+b(\xi)\cdot n)+n^*A(\xi)n$. Then $m$ satisfies the truncation property in $L^1_{t,x}$ uniformly in $\xi\in\mr$, i.e. for every bump function $\psi$ and $\delta>0$ it holds true
\begin{align*}
\bigg\|\mathcal{F}_{tx}^{-1}\psi\bigg(\frac{m(u,n,\xi)}{\delta}\bigg)\mathcal{F}_{tx} f(u,n)\bigg\|_{L^1_{t,x}} &\leq C\|f\|_{L^1_{t,x}},
\end{align*}
where the constant $C$ is independent of $\delta$ and $\xi$.
\end{cor}

\begin{proof}
Let $\tilde\psi$ be a bump function that equals to 1 in a ball containing the one-dimensional projections of the support of $\psi$. We proceed similarly as in Lemma \ref{lemma:truncation} and write
\begin{align}\label{eq:mult}
\psi\bigg(\frac{m(u,n,\xi)}{\delta}\bigg)
&=\sum_{j,k\in\mz}\hat\psi(j,k)\tilde\psi_j\bigg(\frac{u}{\delta}\bigg)\tilde\psi_j\bigg(\frac{m_1(n,\xi)}{\delta}\bigg)\tilde\psi_{k}\bigg(\frac{m_2(n,\xi)}{\delta}\bigg),
\end{align}
where $\tilde\psi_j(x):=\me^{2\pi ijx}\tilde\psi(x)$. Due to Lemma \ref{lemma:truncation}
\begin{align*}
&\bigg\|\mathcal{F}_{tx}^{-1}\tilde\psi_j\bigg(\frac{m_1(n,\xi)}{\delta}\bigg)\mathcal{F}_{tx}f(u,n)\bigg\|_{L^1_{t,x}}=\bigg\|\mathcal{F}_{x}^{-1}\tilde\psi_j\bigg(\frac{m_1(n,\xi)}{\delta}\bigg)\mathcal{F}_{x}f(t,n)\bigg\|_{L^1_{t,x}}
\leq C(j)\|f\|_{L^1_{t,x}},
\end{align*}
where the dependence on $j$ is at most polynomial. Similar estimate holds true for the multiplier
$$\tilde\psi_k\bigg(\frac{m_2(n,\xi)}{\delta}\bigg)$$
and for the remaining one we have
\begin{align*}
\bigg\|\mathcal{F}_{t}^{-1}\tilde\psi_j\bigg(\frac{u}{\delta}\bigg)\bigg\|_{L^1_{t}}&=\frac{1}{(2\pi)^{1/2}}\int_\mr \bigg|\int_\mr\tilde\psi_j\bigg(\frac{u}{\delta}\bigg)\me^{ i ut}\,\dd u\bigg|\,\dd t=\frac{\delta }{(2\pi)^{1/2}}\int_\mr \bigg|\int_\mr\tilde\psi_j(u)\me^{ i \delta ut}\,\dd u\bigg|\,\dd t\\
&=\delta\int_\mr\big|\mathcal{F}_t^{-1}\tilde\psi_j(\delta t)\big|\dd t=\int_\mr\big|\mathcal{F}_t^{-1}\tilde\psi_j( t)\big|\dd t
\end{align*}
so we deduce that
$$\tilde\psi_j\bigg(\frac{u}{\delta}\bigg)$$
is an $L^1$-multiplier on $\mr$. Therefore
\begin{align*}
\bigg\|\mathcal{F}_{tx}^{-1}\tilde\psi_j\bigg(\frac{u}{\delta}\bigg)\mathcal{F}_{tx}f(u,n)\bigg\|_{L^1_{t,x}}=\bigg\|\mathcal{F}_{t}^{-1}\tilde\psi_j\bigg(\frac{u}{\delta}\bigg)\mathcal{F}_{t}f(u,x)\bigg\|_{L^1_{t,x}}\leq C(j)\|f\|_{L^1_{t,x}},
\end{align*}
which completes the proof.
\end{proof}

\begin{lemma}[Multiplier Lemma]\label{lemma:multiplier}

Let $\psi$ be a bump function and let $m(u,n,\xi)$ satisfy the truncation property in $L^1_{t,x}$ uniformly in $
\xi\in\mr$. For each $\xi\in\mr$ and $\delta>0$ let $\Omega(u,n;\delta)\subset\mr$ be the velocity set
$$\Omega(u,n;\delta):=\bigg\{\xi\in\mr;\,\frac{m(u,n,\xi)}{\delta}\in\supp\psi\bigg\}.$$
Consider the velocity-averaged multiplier operator
\[
\overline{M_{\psi}f}(t,x):=\int_{\mr}M_{\psi}f(t,x,\xi)\,\dif\xi=\int_{\mr}\mathcal{F}_{tx}^{-1}\psi\bigg(\frac{m(u,n,\xi)}{\delta}\bigg)\mathcal{F}_{tx} f(u,n,\xi)\,\dif\xi,
\]
then for every $p\in[1,2]$ we have the estimate
$$\|\overline{M_{\psi}f}\|_{L^{p}_{t,x}}\le C\sup_{u,n}|\Omega(u,n;\delta)|^{1/p'}\|f\|_{L^{p}_{t,x,\xi}}.$$

\begin{proof}
Let $p=2$. Then due to Plancherel's theorem, H\"older's inequality and the fact that $\psi$ is compactly supported, we obtain
\begin{align*}
\|\overline{M_{\psi}f}\|_{L^{2}_{t,x}}^2&= \bigg\|\int_\mr\psi\bigg(\frac{m(u,n,\xi)}{\delta}\bigg)\mathcal{F}_{tx} f(u,n,\xi)\,\dif\xi\bigg\|_{L^2_{u,n}}^2\\
&\leq \sum_{n\in\mz^N}\int_{\mr}\int_\mr\bigg|\psi\bigg(\frac{m(u,n,\xi)}{\delta}\bigg)\bigg|^2\,\dif\xi\int_\mr |\mathcal{F}_{tx}  f(u,n,\xi)|^2\,\dif \xi\,\dd u\\
&\le C\sup_{u,n}|\Omega(u,n;\delta)|\|f\|^2_{L^{2}_{t,x,\xi}}.
\end{align*}

Let $p=1$. Then we observe that due to the truncation property
\begin{align*}
\|\overline{M_{\psi}f}\|_{L^1_{t,x}}&\le\int_{\mr}\bigg\|\mathcal{F}_{t,x}^{-1}\psi\bigg(\frac{m(u,n,\xi)}{\delta}\bigg)\mathcal{F}_{tx} f(u,n,\xi)\bigg\|_{L^1_{t,x}}\,\dif\xi\leq C\int_\mr\|f(\xi)\|_{L^1_{t,x}}\,\dif \xi=C\|f\|_{L^1_{t,x,\xi}}.
\end{align*}

Let $p\in(1,2).$ Then by Riesz-Thorin interpolation result we get for $\theta=2\frac{p-1}{p}$ that
\begin{align*}
\|\overline{M_{\psi}}\|_{L^p_{t,x}\rightarrow L^p_{t,x}}&\le \|\overline{M_{\psi}}\|_{L^1_{t,x}\rightarrow L^1_{t,x}}^{1-\theta}\|\overline{M_{\psi}}\|_{L^2_{t,x}\rightarrow L^2_{t,x}}^\theta\leq C \sup_{u,n}|\Omega(u,n;\delta)|^{\theta/2}=C \sup_{u,n}|\Omega(u,n;\delta)|^{1/p'}
\end{align*}
and the proof is complete.
\end{proof}
\end{lemma}

\begin{cor}\label{cor:mult}
Let $m(u,n,\xi)=i(u+b(\xi)\cdot n)+n^*A(\xi)n$. Then for every bump function $\psi$ and every pair $(\epsilon,q_\epsilon)$ satisfying
$$
\frac{N}{q'_\epsilon}<\epsilon<1<q_\epsilon<\frac{N}{N-\epsilon},
$$
where $q'_\epsilon$ is the conjugate exponent to $q_\epsilon$,
it holds true that
$$\|\overline{M_{\psi}f}\|_{L^{1}_{t}W^{-\epsilon,q_\epsilon}_x}\le C \|f\|_{\mathcal{M}_{t,x,\xi}}.$$
\end{cor}

\begin{proof}
Let $(\varrho_\varepsilon)$ be an approximation to the identity on $\R_t\times\mathbb{T}^N_x\times\R_\xi$. If $f\in\mathcal{M}(\R\times\T\times\R)$ then $f^\varepsilon=f*\varrho_\varepsilon$ satisfies
$$
\|f^\varepsilon\|_{ L^1_{t,x,\xi}}\leq \|f\|_{\mathcal{M}_{t,x,\xi}},
$$
hence Lemma \ref{lemma:multiplier} yields
\begin{align}\label{eq:meas123}
\|\overline{M_\psi f^\varepsilon}\|_{L^1_{t,x}} &\leq C\|f^\varepsilon\|_{L^1_{t,x,\xi}}\leq C\|f\|_{\mathcal{M}_{t,x,\xi}},
\end{align}
where the constant $C$ is independent of $\delta$, $\xi$ and $\varepsilon$.
Besides, due to the Banach-Alaoglu theorem, the sequence $(f^\varepsilon)$ converges weak* in $\mathcal{M}(\R\times\T\times\R)$ and the limit is necessarily $f$. Consequently, we may apply lower semicontinuity to the left hand side \eqref{eq:meas123} to deduce
$$
\|\overline{M_\psi f}\|_{\mathcal{M}_{t,x}} \leq C\|f\|_{\mathcal{M}_{t,x,\xi}}.
$$

As the next step, we will prove that the measure $\overline{M_\psi f}$
is absolutely continuous with respect to the Lebesgue measure in time. Indeed,
using the decomposition \eqref{eq:mult}, it was shown in Corollary \ref{cor:truncation} that the multiplier operator
$$
f\mapsto \mathcal{F}^{-1}_{t,x} \psi\left(\frac{m(u,n,\xi)}{\delta}\right)\mathcal{F}_{t,x} f(u,n)
$$
rewrites as a (weighted) sum of multiplier operators of the form
\begin{align*}
f&\mapsto \mathcal{F}^{-1}_{t,x} \tilde\psi_j\left(\frac{u}{\delta}\right)\tilde\psi_j\left(\frac{m_1(n,\xi)}{\delta}\right)\tilde\psi_k\left(\frac{m_2(n,\xi)}{\delta}\right)\mathcal{F}_{t,x} f(u,n)\\
&\qquad=\left[\mathcal{F}^{-1}_{t,x} \tilde\psi_j\left(\frac{u}{\delta}\right)\right]*\left[\mathcal{F}^{-1}_{t,x}\tilde\psi_j\left(\frac{m_1(n,\xi)}{\delta}\right)\tilde\psi_k\left(\frac{m_2(n,\xi)}{\delta}\right)\mathcal{F}_{t,x} f(u,n)\right].
\end{align*}
Now, it is enough to observe (cf. Lemma \ref{lemma:truncation}) that the measure
$$
\mathcal{F}^{-1}_{t,x} \tilde\psi_j\left(\frac{u}{\delta}\right)
$$
is absolutely continuous with respect to the Lebesgue measure in time, i.e. in time is is an $L^1$ function.
Due to possible degeneracies in the remaining multipliers given by
we cannot argue similarly in the space variable. Indeed, the diffusion matrix $A$ can only have partial rank and still satisfy the nondegeneracy assumption \eqref{eq:non-deg-local}, cf. \cite[Corollary 4.2]{tadmor}. To be more precise, with respect to the space variable, the corresponding multiplier operators are given as convolutions with finite Borel measures and not $L^1$-functions.
The claim now follows since for two measures $\mu,\nu\in\mathcal{M}(\R\times\T)$, where $\mu$ is absolutely continuous with respect to the Lebesgue measure in time, the convolution $\mu*\nu$ is also absolutely continuous with respect to the Lebesgue measure in time. Indeed, we obtain that
$$
\mathcal{F}^{-1}_{t,x} \psi\left(\frac{m(u,n,\xi)}{\delta}\right)\mathcal{F}_{t,x} f(u,n)
$$
and hence $\overline{M_\psi f}$ is absolutely continuous with respect to the Lebesgue measure in time.

Finally, due to the Sobolev embedding the space of finite Borel measures on $\T$ is embedded into $W^{-\epsilon,q_\epsilon}
(\T)$ for all $(\epsilon,q_\epsilon)$ satisfying
$$
\frac{N}{q'_\epsilon}<\epsilon<1<q_\epsilon<\frac{N}{N-\epsilon},
$$
where $q'_\epsilon$ is the conjugate exponent to $q_\epsilon$.
The proof is complete.
\end{proof}

\bibliographystyle{abbrv.bst}

\end{document}